\documentclass[11pt]{amsart}

\usepackage{enumitem}
\usepackage{epigamath}

\usepackage[english]{babel}

\numberwithin{equation}{section}

\usepackage{subfigure}
\usepackage{tikz}
\usepackage{tikz-cd}
\usepackage{stmaryrd}
\usepackage{mathdots}

\newtheorem{theorem}{Theorem}[section]
\newtheorem{lemma}[theorem]{Lemma}
\newtheorem{proposition}[theorem]{Proposition}
\newtheorem{corollary}[theorem]{Corollary}
\newtheorem{conjecture}[theorem]{Conjecture}

\newtheorem{recipe}[theorem]{Recipe}

\theoremstyle{definition}
\newtheorem{definition}[theorem]{Definition}
\newtheorem{example}[theorem]{Example}

\theoremstyle{remark}
\newtheorem{remark}[theorem]{Remark}
\newtheorem{remarks}[theorem]{Remarks}

\usepackage{cancel}
\usepackage{color}
\usetikzlibrary{calc,3d,arrows}
\colorlet{ggrey}{gray!80}
\colorlet{bblue}{red!60}
\colorlet{dbblue}{red!80}
\colorlet{bluee}{blue}
\colorlet{ggreen}{green!90!black!90}

\makeatletter
\pgfarrowsdeclare{latexnew}{latexnew}
{
  \ifdim\pgfgetarrowoptions{latexnew}=-1pt%
    \pgfutil@tempdima=0.28pt%
    \pgfutil@tempdimb=\pgflinewidth%
    \ifdim\pgfinnerlinewidth>0pt%
      \pgfmathsetlength\pgfutil@tempdimb{.6\pgflinewidth-.4*\pgfinnerlinewidth}%
    \fi%
    \advance\pgfutil@tempdima by.3\pgfutil@tempdimb%
  \else%
    \pgfutil@tempdima=\pgfgetarrowoptions{latexnew}%
    \divide\pgfutil@tempdima by 10%
  \fi%
  \pgfarrowsleftextend{+-1\pgfutil@tempdima}%
  \pgfarrowsrightextend{+9\pgfutil@tempdima}%
}
{
  \ifdim\pgfgetarrowoptions{latexnew}=-1pt%
    \pgfutil@tempdima=0.28pt%
    \pgfutil@tempdimb=\pgflinewidth%
    \ifdim\pgfinnerlinewidth>0pt%
      \pgfmathsetlength\pgfutil@tempdimb{.6\pgflinewidth-.4*\pgfinnerlinewidth}%
    \fi%
    \advance\pgfutil@tempdima by.3\pgfutil@tempdimb%
  \else%
    \pgfutil@tempdima=\pgfgetarrowoptions{latexnew}%
    \divide\pgfutil@tempdima by 10%
    \pgfsetlinewidth{0bp}%
  \fi%
  \pgfpathmoveto{\pgfqpoint{9\pgfutil@tempdima}{0pt}}
  \pgfpathcurveto
  {\pgfqpoint{6.3333\pgfutil@tempdima}{.5\pgfutil@tempdima}}
  {\pgfqpoint{2\pgfutil@tempdima}{2\pgfutil@tempdima}}
  {\pgfqpoint{-1\pgfutil@tempdima}{3.75\pgfutil@tempdima}}
  \pgfpathlineto{\pgfqpoint{-1\pgfutil@tempdima}{-3.75\pgfutil@tempdima}}
  \pgfpathcurveto
  {\pgfqpoint{2\pgfutil@tempdima}{-2\pgfutil@tempdima}}
  {\pgfqpoint{6.3333\pgfutil@tempdima}{-.5\pgfutil@tempdima}}
  {\pgfqpoint{9\pgfutil@tempdima}{0pt}}
  \pgfusepathqfill
}

\pgfarrowsdeclarereversed{latexnew reversed}{latexnew reversed}{latexnew}{latexnew}

\pgfarrowsdeclare{latex'new}{latex'new}
{
  \ifdim\pgfgetarrowoptions{latex'new}=-1pt%
    \pgfutil@tempdima=0.28pt%
    \advance\pgfutil@tempdima by.3\pgflinewidth%
  \else%
    \pgfutil@tempdima=\pgfgetarrowoptions{latex'new}%
    \divide\pgfutil@tempdima by 10%
  \fi%
  \pgfarrowsleftextend{+-4\pgfutil@tempdima}
  \pgfarrowsrightextend{+6\pgfutil@tempdima}
}
{
  \ifdim\pgfgetarrowoptions{latex'new}=-1pt%
    \pgfutil@tempdima=0.28pt%
    \advance\pgfutil@tempdima by.3\pgflinewidth%
  \else%
    \pgfutil@tempdima=\pgfgetarrowoptions{latex'new}%
    \divide\pgfutil@tempdima by 10%
    \pgfsetlinewidth{0bp}%
  \fi%
  \pgfpathmoveto{\pgfqpoint{6\pgfutil@tempdima}{0\pgfutil@tempdima}}
  \pgfpathcurveto
  {\pgfqpoint{3.5\pgfutil@tempdima}{.5\pgfutil@tempdima}}
  {\pgfqpoint{-1\pgfutil@tempdima}{1.5\pgfutil@tempdima}}
  {\pgfqpoint{-4\pgfutil@tempdima}{3.75\pgfutil@tempdima}}
  \pgfpathcurveto
  {\pgfqpoint{-1.5\pgfutil@tempdima}{1\pgfutil@tempdima}}
  {\pgfqpoint{-1.5\pgfutil@tempdima}{-1\pgfutil@tempdima}}
  {\pgfqpoint{-4\pgfutil@tempdima}{-3.75\pgfutil@tempdima}}
  \pgfpathcurveto
  {\pgfqpoint{-1\pgfutil@tempdima}{-1.5\pgfutil@tempdima}}
  {\pgfqpoint{3.5\pgfutil@tempdima}{-.5\pgfutil@tempdima}}
  {\pgfqpoint{6\pgfutil@tempdima}{0\pgfutil@tempdima}}
  \pgfusepathqfill
}

\pgfarrowsdeclarereversed{latex'new reversed}{latex'new reversed}{latex'new}{latex'new}

\pgfarrowsdeclare{onew}{onew}
{
  \pgfarrowsleftextend{+-.5\pgflinewidth}
  \ifdim\pgfgetarrowoptions{onew}=-1pt%
    \pgfutil@tempdima=0.4pt%
    \advance\pgfutil@tempdima by.2\pgflinewidth%
    \pgfutil@tempdimb=9\pgfutil@tempdima\advance\pgfutil@tempdimb by.5\pgflinewidth%
    \pgfarrowsrightextend{+\pgfutil@tempdimb}%
  \else%
    \pgfutil@tempdima=\pgfgetarrowoptions{onew}%
    \advance\pgfutil@tempdima by -0.5\pgflinewidth%
    \pgfarrowsrightextend{+\pgfutil@tempdima}%
  \fi%
}
{ 
  \ifdim\pgfgetarrowoptions{onew}=-1pt%
    \pgfutil@tempdima=0.4pt%
    \advance\pgfutil@tempdima by.2\pgflinewidth%
    \pgfutil@tempdimb=0pt%
  \else%
    \pgfutil@tempdima=\pgfgetarrowoptions{onew}%
    \divide\pgfutil@tempdima by 9%
    \pgfutil@tempdimb=0.5\pgflinewidth%
  \fi%
  \pgfsetdash{}{+0pt}
  \pgfpathcircle{\pgfpointadd{\pgfqpoint{4.5\pgfutil@tempdima}{0bp}}%
                             {\pgfqpoint{-\pgfutil@tempdimb}{0bp}}}%
                {4.5\pgfutil@tempdima-\pgfutil@tempdimb}%
  \pgfusepathqstroke
}

\pgfarrowsdeclare{squarenew}{squarenew}
{
 \ifdim\pgfgetarrowoptions{squarenew}=-1pt%
   \pgfutil@tempdima=0.4pt
   \advance\pgfutil@tempdima by.275\pgflinewidth%
   \pgfarrowsleftextend{+-\pgfutil@tempdima}
   \advance\pgfutil@tempdima by.5\pgflinewidth
   \pgfarrowsrightextend{+\pgfutil@tempdima}
 \else%
   \pgfutil@tempdima=\pgfgetarrowoptions{squarenew}%
   \divide\pgfutil@tempdima by 8%
   \pgfarrowsleftextend{+-7\pgfutil@tempdima}%
   \pgfarrowsrightextend{+1\pgfutil@tempdima}%
 \fi%
}
{
 \ifdim\pgfgetarrowoptions{squarenew}=-1pt%
   \pgfutil@tempdima=0.4pt%
   \advance\pgfutil@tempdima by.275\pgflinewidth%
   \pgfutil@tempdimb=0pt%
 \else%
   \pgfutil@tempdima=\pgfgetarrowoptions{squarenew}%
   \divide\pgfutil@tempdima by 8%
   \pgfutil@tempdimb=0.5\pgflinewidth%
 \fi%
 \pgfsetdash{}{+0pt}
 \pgfsetroundjoin
 \pgfpathmoveto{\pgfpointadd{\pgfqpoint{1\pgfutil@tempdima}{4\pgfutil@tempdima}}
                            {\pgfqpoint{-\pgfutil@tempdimb}{-\pgfutil@tempdimb}}}
 \pgfpathlineto{\pgfpointadd{\pgfqpoint{-7\pgfutil@tempdima}{4\pgfutil@tempdima}}
                            {\pgfqpoint{\pgfutil@tempdimb}{-\pgfutil@tempdimb}}}
 \pgfpathlineto{\pgfpointadd{\pgfqpoint{-7\pgfutil@tempdima}{-4\pgfutil@tempdima}}
                            {\pgfqpoint{\pgfutil@tempdimb}{\pgfutil@tempdimb}}}
 \pgfpathlineto{\pgfpointadd{\pgfqpoint{1\pgfutil@tempdima}{-4\pgfutil@tempdima}}
                            {\pgfqpoint{-\pgfutil@tempdimb}{\pgfutil@tempdimb}}}
 \pgfpathclose
 \pgfusepathqfillstroke
}

\pgfarrowsdeclare{stealthnew}{stealthnew}
{
  \ifdim\pgfgetarrowoptions{stealthnew}=-1pt%
    \pgfutil@tempdima=0.28pt%
    \pgfutil@tempdimb=\pgflinewidth%
    \ifdim\pgfinnerlinewidth>0pt%
      \pgfmathsetlength\pgfutil@tempdimb{.6\pgflinewidth-.4*\pgfinnerlinewidth}%
    \fi%
    \advance\pgfutil@tempdima by.3\pgfutil@tempdimb%
  \else%
    \pgfutil@tempdima=\pgfgetarrowoptions{stealthnew}%
    \divide\pgfutil@tempdima by 8%
  \fi%
  \pgfarrowsleftextend{+-3\pgfutil@tempdima}
  \pgfarrowsrightextend{+5\pgfutil@tempdima}
}
{
  \ifdim\pgfgetarrowoptions{stealthnew}=-1pt%
    \pgfutil@tempdima=0.28pt%
    \pgfutil@tempdimb=\pgflinewidth%
    \ifdim\pgfinnerlinewidth>0pt%
      \pgfmathsetlength\pgfutil@tempdimb{.6\pgflinewidth-.4*\pgfinnerlinewidth}%
    \fi%
    \advance\pgfutil@tempdima by.3\pgfutil@tempdimb%
  \else%
    \pgfutil@tempdima=\pgfgetarrowoptions{stealthnew}%
    \divide\pgfutil@tempdima by 8%
    \pgfsetlinewidth{0bp}%
  \fi%
  \pgfpathmoveto{\pgfqpoint{5\pgfutil@tempdima}{0pt}}
  \pgfpathlineto{\pgfqpoint{-3\pgfutil@tempdima}{4\pgfutil@tempdima}}
  \pgfpathlineto{\pgfpointorigin}
  \pgfpathlineto{\pgfqpoint{-3\pgfutil@tempdima}{-4\pgfutil@tempdima}}
  \pgfusepathqfill
}

\pgfarrowsdeclarereversed{stealthnew reversed}{stealthnew reversed}{stealthnew}{stealthnew}

\pgfarrowsdeclare{tonew}{tonew}
{
  \ifdim\pgfgetarrowoptions{tonew}=-1pt%
    \pgfutil@tempdima=0.84pt%
    \advance\pgfutil@tempdima by1.3\pgflinewidth%
    \pgfutil@tempdimb=0.21pt%
    \advance\pgfutil@tempdimb by.625\pgflinewidth%
  \else%
    \pgfutil@tempdima=\pgfgetarrowoptions{tonew}%
    \pgfarrowsleftextend{+-0.8\pgfutil@tempdima}%
    \pgfarrowsrightextend{+0.2\pgfutil@tempdima}%
  \fi%
}
{
  \ifdim\pgfgetarrowoptions{tonew}=-1pt%
    \pgfutil@tempdima=0.28pt%
    \advance\pgfutil@tempdima by.3\pgflinewidth%
    \pgfutil@tempdimb=0pt,%
  \else%
    \pgfutil@tempdima=\pgfgetarrowoptions{tonew}%
    \multiply\pgfutil@tempdima by 100%
    \divide\pgfutil@tempdima by 375%
    \pgfutil@tempdimb=0.4\pgflinewidth%
  \fi%
  \pgfsetdash{}{+0pt}
  \pgfsetroundcap
  \pgfsetroundjoin
  \pgfpathmoveto{\pgfpointorigin}
  \pgflineto{\pgfpointadd{\pgfpoint{0.75\pgfutil@tempdima}{0bp}}
                         {\pgfqpoint{-2\pgfutil@tempdimb}{0bp}}}
  \pgfusepathqstroke
  \pgfsetlinewidth{0.8\pgflinewidth}
  \pgfpathmoveto{\pgfpointadd{\pgfqpoint{-3\pgfutil@tempdima}{4\pgfutil@tempdima}}
                             {\pgfqpoint{\pgfutil@tempdimb}{0bp}}}
  \pgfpathcurveto
  {\pgfpointadd{\pgfqpoint{-2.75\pgfutil@tempdima}{2.5\pgfutil@tempdima}}
               {\pgfqpoint{0.5\pgfutil@tempdimb}{0bp}}}
  {\pgfpointadd{\pgfqpoint{0pt}{0.25\pgfutil@tempdima}}
               {\pgfqpoint{-0.5\pgfutil@tempdimb}{0bp}}}
  {\pgfpointadd{\pgfqpoint{0.75\pgfutil@tempdima}{0pt}}
               {\pgfqpoint{-\pgfutil@tempdimb}{0bp}}}
  \pgfpathcurveto
  {\pgfpointadd{\pgfqpoint{0pt}{-0.25\pgfutil@tempdima}}
               {\pgfqpoint{-0.5\pgfutil@tempdimb}{0bp}}}
  {\pgfpointadd{\pgfqpoint{-2.75\pgfutil@tempdima}{-2.5\pgfutil@tempdima}}
               {\pgfqpoint{0.5\pgfutil@tempdimb}{0bp}}}
  {\pgfpointadd{\pgfqpoint{-3\pgfutil@tempdima}{-4\pgfutil@tempdima}}
               {\pgfqpoint{\pgfutil@tempdimb}{0bp}}}
  \pgfusepathqstroke
}

\pgfarrowsdeclarealias{<new}{>new}{tonew}{tonew}

\makeatother

\pgfsetarrowoptions{latexnew}{-1pt}
\pgfsetarrowoptions{latex'new}{-1pt}
\pgfsetarrowoptions{onew}{-1pt}
\pgfsetarrowoptions{squarenew}{-1pt}
\pgfsetarrowoptions{stealthnew}{-1pt}
\pgfsetarrowoptions{tonew}{-1pt}
\pgfkeys{/tikz/.cd, arrowhead/.default=-1pt, arrowhead/.code={
  \pgfsetarrowoptions{latexnew}{#1},
  \pgfsetarrowoptions{latex'new}{#1},
  \pgfsetarrowoptions{onew}{#1},
  \pgfsetarrowoptions{squarenew}{#1},
  \pgfsetarrowoptions{stealthnew}{#1},
  \pgfsetarrowoptions{tonew}{#1},
}}

\newcommand{\one}{\ensuremath{(\mathrm{i})}}
\newcommand{\two}{\ensuremath{(\mathrm{ii})}}
\newcommand{\three}{\ensuremath{(\mathrm{iii})}}
\newcommand{\four}{\ensuremath{(\mathrm{iv})}}

\newcommand{\CC}{\ensuremath{\mathbb{C}}} 
\newcommand{\kk}{\ensuremath{\mathbb{C}}} 
\newcommand{\NN}{\ensuremath{\mathbb{N}}} 
\newcommand{\RR}{\ensuremath{\mathbb{R}}}
\newcommand{\TT}{\ensuremath{\mathbb{T}}} 
\newcommand{\vv}{\ensuremath{\mathbf{v}}} 
\newcommand{\ZZ}{\ensuremath{\mathbb{Z}}} 

\newcommand{\e}{\ensuremath{\mathfrak{e}}}

\newcommand{\m}{\ensuremath{\mathfrak{m}}}
\newcommand{\n}{\ensuremath{\mathfrak{n}}}

\newcommand{\BKR}{\operatorname{BKR}}

\newcommand{\coh}{\operatorname{coh}}
\newcommand{\conv}{\operatorname{conv}}

\newcommand{\Dimer}{\ensuremath{\Gamma}}
\renewcommand{\div}{\operatorname{div}} 
\newcommand{\ghilb}{\ensuremath{G}\operatorname{-Hilb}}
\newcommand{\End}{\operatorname{End}} 
\newcommand{\git}{\operatorname{\!/\!\!/\!}}
\newcommand{\head}{\operatorname{\mathsf{h}}}
\newcommand{\hex}{\operatorname{Hex}}
\newcommand{\Hom}{\operatorname{Hom}} 
\newcommand{\jig}{\operatorname{J}}
\newcommand{\Image}{\operatorname{Image}}
\newcommand{\Irr}{\operatorname{Irr}}

\newcommand{\KV}{\operatorname{KV}}
\newcommand{\ltensor}{\overset{\mathbf{L}}{\otimes}}

 \newcommand{\modA}{\operatorname{mod-}\ensuremath{\!A}}
  
\newcommand{\Pic}{\operatorname{Pic}}
\newcommand{\Proj}{\operatorname{Proj}}
\newcommand{\soc}{\operatorname{soc}}
\newcommand{\SL}{\operatorname{SL}}
\newcommand{\Spec}{\operatorname{Spec}}
\newcommand{\supp}{\operatorname{supp}}
\newcommand{\tail}{\operatorname{\mathsf{t}}}
\newcommand{\tiling}{\operatorname{Graph}} 
\newcommand{\wt}{\operatorname{wt}}

\EpigaVolumeYear{5}{2021} \EpigaArticleNr{4} \ReceivedOn{February 11,
2020}
\InFinalFormOn{December 10, 2020}
\AcceptedOn{January 14, 2021}

\title{Combinatorial Reid's recipe for consistent dimer models}
\titlemark{Combinatorial Reid's recipe}

\author{Alastair Craw} 
\address{Department of Mathematical Sciences, University of Bath, Claverton Down, Bath BA2 7AY, UK.}
\email{a.craw@bath.ac.uk}
 
\author{Liana Heuberger} 
\address{Laboratoire angevin de recherche en math\'{e}matiques, Facult\'{e} des Sciences, 
2 Boulevard Lavoisier, 49045 Angers cedex 01, France.}
\email{liana.heuberger@univ-angers.fr}

\author{Jesus Tapia Amador} 
\address{Downside School, Stratton-on-the-Fosse, Radstock, Somerset, BA3 4RJ, UK.}
\email{j.ta87@yahoo.co.uk}

\authormark{A. Craw, L. Heuberger, and J. Tapia Amador}

\AbstractInEnglish{\footnotesize{Reid's recipe~\cite{Reid97, Craw05} for a finite abelian subgroup $G\subset \SL(3,\CC)$ is a combinatorial procedure that marks the toric fan of the $G$-Hilbert scheme with irreducible representations of $G$. The geometric McKay correspondence conjecture of Cautis--Logvinenko~\cite{CautisLogvinenko09} that describes certain objects in the derived category of $\ghilb$ in terms of Reid's recipe was later proved by Logvinenko et.\ al.~\cite{Logvinenko10, CCL17}. We generalise Reid's recipe to any consistent dimer model by marking the toric fan of a crepant resolution of the vaccuum moduli space in a manner that is compatible with the geometric correspondence of Bocklandt--Craw--Quintero--V\'{e}lez~\cite{BCQ15}. Our main tool generalises the jigsaw transformations of Nakamura~\cite{Nakamura01} to consistent dimer models.}}

\MSCclass{14E16; 14M25; 16E35; 16G20}


\TitleInFrench{Proc\'d\'e de Reid combinatoire pour les mod\`eles de dim\`eres coh\'erents}

\AbstractInFrench{\footnotesize{Le proc\'ed\'e de Reid~\cite{Reid97, Craw05} pour un sous-groupe ab\'elien fini $G\subset \SL(3,\CC)$ est une proc\'edure combinatoire qui fournit un marquage de l'\'eventail torique  du $G$-sch\'ema de Hilbert par les composantes irr\'eductibles de $G$. La version g\'eom\'etrique de la correspondance de McKay, conjecture formul\'ee par Cautis--Logvinenko~\cite{CautisLogvinenko09} qui d\'ecrit cerains objets de la cat\'egorie d\'eriv\'ee de $\ghilb$ en termes de proc\'ed\'e de Reid, a \'et\'e d\'emontr\'ee ensuite par  Logvinenko et.\ al.~\cite{Logvinenko10, CCL17}. Nous g\'en\'eralisons \`a tout mod\`ele de dim\`ere le proc\'ed\'e de Reid qui consiste \`a marquer l'\'eventail torique d'une r\'esolution cr\'epante de l'espace des modules d'une mani\`ere compatible avec la correspondance g\'eom\'etrique de Bocklandt--Craw--Quintero--V\'{e}lez~\cite{BCQ15}. Notre outil principal g\'en\'eralise les transformations de puzzles de Nakamura~\cite{Nakamura01} aux mod\`eles de dim\`eres coh\'erents.}}

\acknowledgement{The second author was funded by grant EP/S004130/1 (PI Elisa Postinghel) and the third author was funded by CONACYT.}

\begin{document}


\removeabove{0,8cm}
\removebetween{0,8cm}
\removebelow{0,8cm}

\maketitle

\begin{prelims}

\DisplayAbstractInEnglish

\bigskip

\DisplayKeyWords

\medskip

\DisplayMSCclass

\bigskip

\languagesection{Fran\c{c}ais}

\bigskip

\DisplayTitleInFrench

\medskip

\DisplayAbstractInFrench

\end{prelims}


\newpage

\setcounter{tocdepth}{1}

\tableofcontents


\section{Introduction}
For a finite subgroup $G\subset \SL(2,\CC)$, the construction by McKay~\cite{McKay80} of an affine Dynkin diagram of type ADE directly from the representation theory of $G$  established a bijection between the nontrivial irreducible representations of $G$ and the irreducible components of the exceptional divisor on the minimal resolution $S$ of $\CC^2/G$. For a geometric interpretation, identify  
$S$ with the $G$-Hilbert scheme following Ito--Nakamura~\cite{ItoNakamura99} and use the induced universal family on $S$ to obtain an equivalence of derived categories of coherent sheaves  \[
 \Phi_{\KV}\colon D^b\big(G\text{-}\coh(\mathbb{C}^2)\big)\longrightarrow D^b\big(\!\coh(S)\big)
 \]
 as in Kapranov--Vasserot~\cite{KapranovVasserot00}. For a nontrivial irreducible representation $\rho$ of $G$, \emph{ibid}.\ shows that the functor $\Phi_{\KV}$ sends the simple $G$-sheaf $\mathcal{O}_0\otimes \rho$ on $\CC^2$ to the pure sheaf $\mathcal{O}_{C_\rho}(-1)[1]$ on $S$ whose support is the irreducible component $C_\rho$ of the exceptional locus associated to $\rho$ by McKay's bijection. Alternatively, the first Chern classes of the vector bundles $\Phi_{\KV}(\mathcal{O}_{\CC^2}\otimes \rho)$ introduced by Gonzalez-Sprinberg--Verdier~\cite{GSV83} provide the basis of $H^2(S,\ZZ)$ dual to the basis of $H_2(S,\ZZ)$ given by the curves classes $[C_\rho]$ for nontrivial $\rho$.
 
A similar story emerges for any finite subgroup $G\subset \SL(3,\CC)$, though the final step is well-understood only when $G$ is abelian. Indeed, building on work of Nakamura~\cite{Nakamura01} in the abelian case, Bridgeland--King--Reid~\cite{BKR01} 
 used the universal family on $\ghilb$ to obtain an equivalence of derived categories of coherent sheaves  
 \[
 \Phi_{\BKR}\colon D^b\big(\!\coh(\ghilb)\big)\longrightarrow  D^b\big(G\text{-}\coh(\mathbb{C}^3)\big)
 \]
 from which they deduced that $\ghilb$ is a crepant resolution of $\CC^3/G$. The conventions from \cite{KapranovVasserot00} and \cite{BKR01} are such that the functors $\Phi_{\KV}$ and $\Phi_{\BKR}$ determined by the universal family on $\ghilb$ point in opposite directions, so to implement the final step in the programme extending the McKay correspondence to dimension three, Cautis--Logvinenko~\cite{CautisLogvinenko09} considered the quasi-inverse functor
 \[
 \Psi:=\Phi_{\BKR}^{-1}\colon D^b\big(G\text{-}\coh(\mathbb{C}^3)\big)\longrightarrow  D^b\big(\!\coh(\ghilb)\big)
 \]
 and studied the images under $\Psi$ of the simple $G$-sheaves $\mathcal{O}_0\otimes \rho$ on $\CC^3$ when $G$ is abelian. Their conjecture, proved by Logvinenko~\cite{Logvinenko10} when $\CC^3/G$ has a single isolated singularity and by Cautis--Craw--Logvinenko~\cite{CCL17} for any finite abelian subgroup $G\subset \SL(3,\CC)$, asserted that the objects $\Psi(\mathcal{O}_0\otimes \rho)$ can be computed explicitly from a recipe of Reid~\cite{Reid97} for marking torus-invariant subvarieties of $\ghilb$ with irreducible representations of $G$. \emph{Reid's recipe},  shown to hold for any finite abelian subgroup of $\SL(3,\CC)$ by Craw~\cite{Craw05}, therefore provides the combinatorial input allowing for a complete understanding of the geometric McKay correspondence that sends each nontrivial irreducible representation $\rho$ of $G$ to the pure sheaf $\Psi(\mathcal{O}_0\otimes \rho)$ on $\ghilb$. 
 
  We study a more general situation in which the singularity $\CC^3/G$ is replaced by an arbitrary Gorenstein affine toric threefold $\Spec R$, and where $\ghilb$ is replaced by a moduli space of cyclic modules over a noncommutative crepant resolution (NCCR) $A$ of $R$. Bocklandt~\cite{Bocklandt12, Bocklandt13} shows that every such algebra $A$ arises as the path algebra of a quiver $Q$ modulo the ideal of relations encoded by a consistent dimer model $\Gamma$ in a real 2-torus. To study cyclic $A$-modules we specify a vertex of $Q$, denoted $0\in Q_0$, that determines a \emph{0-generated stability condition} $\theta$ (see \eqref{eqn:0generated} below). The fine moduli space $\mathcal{M}_\theta$ of $\theta$-stable $A$-modules of dimension vector $(1,..,1)$ 
  comes armed with a tautological vector bundle $T=\bigoplus_{i\in Q_0} L_i$, where each $L_i$ has rank one. Results of Ishii--Ueda~\cite{IshiiUeda08, IshiiUeda15} and Broomhead~\cite{Broomhead12} imply that $\mathcal{M}_\theta$ is a toric, crepant resolution of $\Spec R$ and, moreover,  the universal family over $\mathcal{M}_\theta$ induces a derived equivalence $\Phi_\theta$ with quasi-inverse 
  \[
 \Psi_\theta:=\Phi_\theta^{-1}\colon D^b(\modA) \longrightarrow D^b\big(\coh(\mathcal{M}_\theta)\big).
\]
 Each $i\in Q_0$ defines a vertex simple $A$-module $S_i$ and the object $\Psi_\theta(S_i)$ is known to be a pure sheaf for every $i\neq 0$ by work of Bocklandt--Craw--Quintero-V\'{e}lez~\cite{BCQ15}. In light of the works of Logvinenko et.~al.~\cite{CautisLogvinenko09, Logvinenko10, CCL17}, it is natural to seek a recipe that marks cones in the fan $\Sigma_\theta$ of the toric variety $\mathcal{M}_\theta$ with nonzero vertices of the quiver in such a way that the objects $\Psi_\theta(S_i)$ for $i\neq 0$ can be computed from this marking. 
 
 To state our first main result, each cone $\sigma\in \Sigma_\theta$ determines a $\theta$-stable $A$-module $M_\sigma$ that is isomorphic to the fibre of the tautological bundle $T$ over a distinguished point $y\in \mathcal{M}_\theta$ in the torus-orbit determined by $\sigma$ (see Lemma~\ref{lem:conepoint}). In what follows, it is convenient to identify $\Sigma_\theta$ with the induced triangulation of the polygon obtained as the height one slice of $\Sigma_\theta$, so we refer to cones in $\Sigma_\theta$ of dimension one, two and three as \emph{lattice points, line segments} and \emph{lattice  triangles} respectively. The following statement combines results from Definitions~\ref{def:RRpoints} and \ref{def:RRlines}, as well as Corollary~\ref{cor:RRcompatible}:
 
 \begin{theorem}[Combinatorial Reid's recipe]
 \label{thm:main}
 Let $Q$ be the quiver dual to a consistent dimer model and choose a vertex $0\in Q_0$. There is a combinatorial recipe that marks every internal lattice point and line segment of the fan $\Sigma_\theta$ with one or more nonzero vertices of $Q$; specifically, vertex $i\in Q_0$ marks:
\begin{enumerate}
 \item[\one] a lattice point $\rho$ iff $S_i$ lies in the socle of each torus-invariant $\theta$-stable $A$-module $M_\sigma$ defined by a lattice triangle $\sigma\in \Sigma_\theta$ satisfying $\rho\subset \sigma$; and
 \item[\two] a line segment $\tau$ iff $i$ is a source vertex of a pair of quivers $Q(\tau)^+$ and $Q(\tau)^-$ that we obtain from the $\theta$-stable $A$-modules $M_{\sigma_+}$ and $M_{\sigma_-}$ defined by the pair of lattice triangles $\sigma_\pm\in \Sigma_\theta$ satisfying $\tau\subset \sigma_\pm$.
 \end{enumerate}
  Moreover, this recipe agrees with Reid's original recipe for marking cones in the toric fan of $\ghilb$ in the special case when $Q$ is the  McKay quiver of a finite abelian subgroup $G\subset \SL(3,\CC)$.
  \end{theorem}
 
  The torus-invariant $\theta$-stable $A$-modules $M_\sigma$ that provide the input for Theorem~\ref{thm:main} are each encoded combinatorially by a fundamental domain in $\RR^2$ for the action of the group $\ZZ^2$ of deck transformations of the real 2-torus containing the dimer model $\Gamma$. These \emph{fundamental hexagons}, introduced originally by Ishii--Ueda~\cite[Section~4]{IshiiUeda08}, are easy to calculate by hand in small examples, while a computer program written by Raf Bocklandt can perform the calculation for larger examples. As a result, our recipe is easy to implement in practice, and we illustrate this by presenting two examples in Section~\ref{sec:examples}. These examples exhibit many new features that were not present in Reid's original recipe for the toric fan of $\ghilb$ (see Remark~\ref{rem:features}).
  
 
  The statement in Theorem~\ref{thm:main} that is by no means obvious, however, is that the quivers $Q(\tau)^{\pm}$ associated to the interior line segment $\tau$ share the same source vertices. To define these quivers we first reconstruct the $A$-module $M_{\sigma_-}$ directly from $M_{\sigma_+}$, and vice versa, using a procedure that we call a \emph{jigsaw transformation}. This procedure (see Theorem~\ref{thm:genjigsaw}) establishes that the choice of the facet $\tau\subset \sigma_+$ determines a way to cut up the fundamental hexagon $\hex(\sigma_+)$ into `jigsaw pieces', each of which is a union of (lifts to $\RR^2$ of) tiles in $\Gamma$, and moreover, each piece can be translated by a uniquely determined element in the group $\ZZ^2$ of deck transformations such that the newly arranged jigsaw pieces reconstruct $\hex(\sigma_-)$; see Figure~\ref{fig:cuttinghexagon} for an illustration. Now, a unique jigsaw piece contains the tile dual to the vertex $0\in Q_0$, and once we remove this jigsaw piece from each of $\hex(\sigma_\pm)$, we are left with the combinatorial data encoding the quivers $Q(\tau)^{\pm}$. These two quivers do not coincide in general (see Remark~\ref{rem:charjig-(1)}), but the striking fact is that the source vertices of these two quivers are the same (see Corollary~\ref{cor:commonsource}). This key ingredient allows us to mark vertices on line segments in $\Sigma_\theta$ as in Theorem~\ref{thm:main}\two. That our recipe generalises that of Reid~\cite{Reid97} when $Q$ is the McKay quiver of a finite abelian subgroup $G\subset \SL(3,\CC)$ follows from the fact that our jigsaw transformations recover the so-called `$G$-igsaw transformations' of Nakamura~\cite{Nakamura01} in that case (see Corollary~\ref{cor:Gigsaw}).
  
  By combining Theorem~\ref{thm:main} with results of \cite{BCQ15}, we provide the following simple combinatorial description of the objects $\Psi_\theta(S_i)$ on $\mathcal{M}_\theta$ associated to some of the nonzero vertices $i\in Q_0$ (see Proposition~\ref{prop:GRR}). 
  
   \begin{corollary}[Compatibility with Geometric Reid's recipe]
   \label{cor:compatibility}
  The recipe from Theorem~\ref{thm:main} is such that:
  \begin{enumerate}
\item[\one]  if a vertex $i$ marks at least one interior lattice point in $\Sigma_\theta$, then $\Psi_\theta(S_i)\cong L_i^{-1}\otimes \mathcal{O}_{Z_i}$ where $Z_i$ is the (connected) union of all torus-invariant divisors $D_\rho$ indexed by lattice points $\rho$ marked with $i$;
 \item[\two] if a vertex $i$ marks a unique interior line segment $\tau$ in $\Sigma_\theta$, then $\Psi_\theta(S_i)\cong L_i^{-1}\otimes \mathcal{O}_{C_\tau}$ where $C_\tau$ is the torus-invariant curve defined by $\tau$; 
 \end{enumerate}
  \end{corollary}
  
  Corollary~\ref{cor:compatibility} makes no mention of the vertices $i$ marking two or more line segments; instead, we conjecture in such cases that the support of $\Psi_\theta(S_i)$ is the union of all torus-invariant divisors $D_\rho$ defined by lattice points $\rho$ in $\Sigma_\theta$ attached to two or more line segments marked with vertex $i$ (see Conjecture~\ref{conj:CHT} for a stronger statement). In fact, we anticipate in such cases that the object $\Psi_\theta(S_i)$ can be described explicitly in terms of the lattice points at the endpoints of the line segments singled out by  Theorem~\ref{thm:main}\two\ in a manner similar to that given by Cautis--Craw--Logvinenko~\cite[Theorem~1.2]{CCL17} for the toric fan of $\ghilb$. The
  key technical tool in that work, namely the construction of CT-subdivisions in the fan, required as input the original combinatorial recipe from \cite{Reid97, Craw05}. Thus, Theorem~\ref{thm:main}\two\ paves the way for a dimer model generalisation.
  
  \smallskip
 
 We have focused on the geometric interpretation of Reid's recipe that describes the objects $\Psi_\theta(S_i)$, but the application that Reid originally had in mind was to encode a minimal set of relations in $\Pic(\ghilb)$ between the tautological line bundles $L_i$ for $i\neq 0$ (see \cite[Theorem~6.1]{Craw05}). Using Theorem~\ref{thm:main}, we formulate a conjecture presenting a similar set of minimal relations in $\Pic(\mathcal{M}_\theta)$ for any consistent dimer model with a chosen vertex $0\in Q_0$ (see Conjecture~\ref{conj:rels}). A proof of this conjecture should lead to the following:
 \begin{itemize}
     \item the construction of a $\ZZ$-basis of the cohomology $H^*(\mathcal{M}_\theta,\ZZ)$ following the approach from \cite[Section~7]{Craw05}. This would provide the appropriate analogue of the original geometric construction of the McKay correspondence by Gonzalez-Sprinberg--Verdier~\cite{GSV83}.
     \item a description of the GIT chamber containing the 0-generated stability parameter $\theta$ defining $\mathcal{M}_\theta$, in a manner similar to that given recently by Wormleighton~\cite{Wormleighton20} for the GIT chamber of $\ghilb$. 
 \end{itemize}

 As a word of warning on terminology, we use three different combinatorial objects determined by points and lines: the dimer model $\Gamma$; the quiver $Q$; and the triangulation of a polygon that determines the fan $\Sigma_\theta$. To avoid confusion, we use different terminology and notation in each case that we summarise here:
\begin{table}[!ht]
\centering
\begin{tabular}{c|c c c}
 & 0-dimensional & 1-dimensional & 2-dimensional\\ \hline
Dimer model $\Dimer$ & nodes ($\n$) & edges ($\e$) & tiles ($\mathfrak{t}$)\\
Quiver $Q$ & vertices ($i$) & arrows ($a$) & \\
Triangulation $\Sigma_\theta$ & lattice points ($\rho$) & line segments ($\tau$) & lattice triangles ($\sigma$)
\end{tabular}
\end{table}


\subsection*{Acknowledgements}
We thank Raf Bocklandt for sharing with us the computer-generated examples upon which the results of this paper are based. Thanks to Timothy Logvinenko and Alastair King for examining the PhD thesis of the third author that contains many of our results. Thanks also to the anonymous referee for several helpful comments.

\section{Background}

\subsection{Dimer models}
\label{sec:dimers}
 Let $\Dimer$ be a dimer model in the real two-torus $\mathbb{T}$, \emph{i.\,e.}\ $\Dimer$ is a polygonal cell decomposition of $\mathbb{T}$, where elements of the set of $i$-cells $\Dimer_i$ are called \emph{nodes}, \emph{edge} and \emph{tiles} when $i=0,1,2$ respectively, that satisfies the following properties: first, the set $\Dimer_0$ decomposes as the disjoint union of a set of \emph{black} nodes and a set of \emph{white} nodes, such that every edge $\mathfrak{e}\in \Dimer_1$ joins a black node to a white node; and secondly, every tile is a simply-connected convex polygon. We may assume that $\Gamma$ contains no bivalent nodes.

 Dually, we obtain a quiver $Q=(Q_0, Q_1)$ embedded in $\mathbb{T}$, where each vertex $i\in Q_0$ lies in the interior of a tile of $\Dimer$, and each arrow $a\in Q_1$ crosses a unique edge $\mathfrak{e}_a\in \Dimer_1$ in such a way that the white endnode of $\mathfrak{e}_a$ lies to the right of $a$. Let $\head(a), \tail(a)\in Q_0$ denote the head and tail of an arrow $a\in Q_1$. Let $Q_2$ denote the set of connected components of the complement of this quiver in $\mathbb{T}$, and we refer to elements of $Q_2$ as \emph{faces} of $Q$. Each face $f\in Q_2$ contains a unique node of $\Dimer$, and the arrows in the boundary of $f$ form a cycle in $Q$ that is clockwise (resp.\ anticlockwise) when the dual node is white (resp.\ black). Every arrow $a\in Q_1$ therefore appears in the cycle traversing the boundary of precisely one white face and one black face, and we may write these cycles in the form $ap_a^+$ and $ap_a^-$ respectively, where $p_a^{\pm}$ are paths in $Q$ with tail at $\head(a)$ and head at $\tail(a)$.  The \emph{Jacobian algebra} of the dimer model $\Dimer$ is defined to be the quotient
 \begin{equation}
     \label{eqn:Jacobian}
 A:= \kk Q / \langle p_a^+-p_a^- \mid a\in Q_1\rangle
 \end{equation}
 of the path algebra $\kk Q$ of the quiver $Q$ by the ideal of relations $\langle p_a^+-p_a^- \mid a\in Q_1\rangle$.
  
  A \emph{perfect matching} of $\Dimer$ is a set of edges $\Pi\subset \Dimer_1$ such that for each node $\mathfrak{n}\in \Dimer_0$, there is a unique edge $\mathfrak{e}\in \Pi$ for which $\mathfrak{n}$ is an endnode of $\mathfrak{e}$. Given a pair of perfect matchings $\Pi, \Pi'$, the locus $\RR^2\setminus (\Pi\cup \Pi')$ is a union of connected components. The \emph{height function} $h_{\Pi,\Pi'}\colon \RR^2\to \RR$, which is only well-defined up to the choice of an additive constant, is locally-constant on  $\RR^2\setminus (\Pi\cup \Pi')$ and is determined as follows: the value of $h_{\Pi,\Pi'}$ increases  (resp.\ decreases) by 1 when a path passing between connected components crosses either an edge $\mathfrak{e}\in \Pi$ with a black (resp.\ white) node to the right, or an edge $\mathfrak{e}\in \Pi'$ with the white (resp.\ black) node to the right. The ambiguity in the choice of additive constant is removed by defining functions 
\[
h_x(\Pi, \Pi'):= h_{\Pi,\Pi'}\big(p + (1, 0)\big) - h_{\Pi,\Pi'}(p)\quad\text{and}\quad h_y(\Pi, \Pi'):= h_{\Pi,\Pi'}\big(p + (0, 1)\big) - h_{\Pi,\Pi'}(p)
\]
 for any point $p\in \RR^2\setminus (\Pi\cup\Pi')$, and the \emph{height change} of $\Pi$ with respect to $\Pi'$ is defined to be
\begin{equation}
\label{eqn:heightchange}
h(\Pi, \Pi') = \big(h_x(\Pi, \Pi'), h_y(\Pi, \Pi')\big)\in H^1(\TT,\ZZ).
\end{equation}
 Fix a reference perfect matching $\Pi_0$, and define the \emph{characteristic polygon} of $\Gamma$ to be the convex polygon
 \[
 \Delta(\Gamma) := \conv\big\{h(\Pi,\Pi_0)\in H^1(\TT,\RR) \mid \Pi\text{ is a perfect matching of }\Gamma\big\}
 \]
 obtained as the convex hull of the lattice points in $H^1(\TT,\RR)\cong \RR^2$. Given perfect matchings $\Pi_i$ and $\Pi_j$, the line segment joining the lattice points is $h(\Pi_i,\Pi_j):=h(\Pi_i,\Pi_0)-h(\Pi_j,\Pi_0)$. If we choose a different reference perfect matching then $\Delta(\Gamma)$ undergoes a translation. For the rank three lattice $N:=H^1(\TT,\ZZ)\oplus \ZZ$, let $\sigma_0\subset N\otimes_\ZZ \mathbb{Q}$ denote the cone over the lattice polygon $\Delta(\Gamma)\times \{1\}$. Then for the dual lattice $M:=\Hom(N,\ZZ)$, the toric threefold $X:=\Spec \kk[\sigma_0^\vee\cap M]$ determined by the characteristic polygon $\Delta(\Gamma)$ is Gorenstein.
 

\begin{example}
\label{exa:LongHex}
Consider the dimer model shown in black in Figure~\ref{fig:LongHex}. The dual quiver $Q$ is drawn in grey in the same picture. This dimer model admits 60 perfect matchings, but for any reference matching $\Pi_0$ there are only 10 distinct height changes $h(\Pi,\Pi_0)$. The characteristic polygon $\Delta(\Gamma)$ is shown in Figure~\ref{fig:LongHexPolygon}.

\begin{figure}[!ht]
   \centering
      \subfigure[]{
\begin{tikzpicture} [thick,scale=0.4, every node/.style={scale=1}] 
\begin{scope}\clip (0pt,0pt) rectangle (400pt,400pt);
\draw [ggrey,-stealthnew,arrowhead=6pt,shorten >=5pt] (26pt,26pt) to node [rectangle,draw,fill=white,sloped,inner sep=1pt] {{\tiny 1}} (108pt,-56pt); 
\draw [ggrey,-stealthnew,arrowhead=6pt,shorten >=5pt] (26pt,426pt) to node [rectangle,draw,fill=white,sloped,inner sep=1pt] {{\tiny 1}} (108pt,344pt);  
\draw [ggrey,-stealthnew,arrowhead=6pt, shorten >=5pt] (26pt,26pt) to node [rectangle,draw,fill=white,sloped,inner sep=1pt] {{\tiny 2}} (67pt,185pt); 
\draw [ggrey,-stealthnew,arrowhead=6pt,shorten >=5pt] (426pt,426pt) to node [rectangle,draw,fill=white,sloped,inner sep=1pt] {{\tiny 3}} (267pt,385pt); 
\draw [ggrey,-stealthnew,arrowhead=6pt,shorten >=5pt] (26pt,26pt) to node [rectangle,draw,fill=white,sloped,inner sep=1pt] {{\tiny 3}} (-133pt,-15pt); 
\draw [ggrey,-stealthnew,arrowhead=6pt,shorten >=5pt] (426pt,26pt) to node [rectangle,draw,fill=white,sloped,inner sep=1pt] {{\tiny 3}} (267pt,-15pt); 
\draw [ggrey,-stealthnew,arrowhead=6pt,shorten >=5pt] (-50pt,-98pt) to node [rectangle,draw,fill=white,sloped,inner sep=1pt] {{\tiny 4}} (26pt,26pt); 
\draw [ggrey,-stealthnew,arrowhead=6pt,shorten >=5pt] (350pt,302pt) to node [rectangle,draw,fill=white,sloped,inner sep=1pt] {{\tiny 4}} (426pt,426pt); 
\draw [ggrey,-stealthnew,arrowhead=6pt,shorten >=5pt] (-50pt,302pt) to node [rectangle,draw,fill=white,sloped,inner sep=1pt] {{\tiny 4}} (26pt,426pt);
\draw [ggrey,-stealthnew,arrowhead=6pt,shorten >=5pt] (350pt,302pt) to node [rectangle,draw,fill=white,sloped,inner sep=1pt] {{\tiny 5}} (273pt,179pt); 
\draw [ggrey,-stealthnew,arrowhead=6pt,shorten >=5pt] (273pt,579pt) to node [rectangle,draw,fill=white,sloped,inner sep=1pt] {{\tiny 6}} (108pt,344pt); 
\draw [ggrey,-stealthnew,arrowhead=6pt,shorten >=5pt] (273pt,179pt) to node [rectangle,draw,fill=white,sloped,inner sep=1pt] {{\tiny 6}} (108pt,-56pt); 
\draw [ggrey,-stealthnew,arrowhead=6pt,shorten >=5pt] (-127pt,179pt) to node [rectangle,draw,fill=white,sloped,inner sep=1pt] {{\tiny 7}} (108pt,344pt); 
\draw [ggrey,-stealthnew,arrowhead=6pt,shorten >=5pt] (273pt,179pt) to node [rectangle,draw,fill=white,sloped,inner sep=1pt] {{\tiny 7}} (508pt,344pt); 
\draw [ggrey,-stealthnew,arrowhead=6pt,shorten >=5pt] (273pt,179pt) to node [rectangle,draw,fill=white,sloped,inner sep=1pt] {{\tiny 8}} (191pt,261pt); 
\draw [ggrey,-stealthnew,arrowhead=6pt,shorten >=5pt] (308pt,144pt) to node [rectangle,draw,fill=white,sloped,inner sep=1pt,pos=0.4] {{\tiny 9}} (273pt,179pt); 
\draw [ggrey,-stealthnew,arrowhead=6pt,shorten >=5pt] (308pt,144pt) to node [rectangle,draw,fill=white,sloped,inner sep=1pt] {{\tiny 10}} (344pt,108pt); 
\draw [ggrey,-stealthnew,arrowhead=6pt,shorten >=5pt] (508pt,344pt) to node [rectangle,draw,fill=white,sloped,inner sep=1pt] {{\tiny 11}} (350pt,302pt); 
\draw [ggrey,-stealthnew,arrowhead=6pt,shorten >=5pt] (108pt,344pt) to node [rectangle,draw,fill=white,sloped,inner sep=1pt] {{\tiny 11}} (-50pt,302pt);  
\draw [ggrey,-stealthnew,arrowhead=6pt,shorten >=5pt] (108pt,344pt) to node [rectangle,draw,fill=white,sloped,inner sep=1pt] {{\tiny 12}} (-92pt,144pt); 
\draw [ggrey,-stealthnew,arrowhead=6pt,shorten >=5pt] (508pt,344pt) to node [rectangle,draw,fill=white,sloped,inner sep=1pt] {{\tiny 12}} (308pt,144pt); 
\draw [ggrey,-stealthnew,arrowhead=6pt,shorten >=5pt] (108pt,-56pt) to node [rectangle,draw,fill=white,sloped,inner sep=1pt] {{\tiny 13}} (308pt,144pt); 
\draw [ggrey,-stealthnew,arrowhead=6pt,shorten >=5pt] (108pt,344pt) to node [rectangle,draw,fill=white,sloped,inner sep=1pt] {{\tiny 13}} (308pt,544pt); 
\draw [ggrey,-stealthnew,arrowhead=6pt,shorten >=5pt] (108pt,344pt) to node [rectangle,draw,fill=white,sloped,inner sep=1pt] {{\tiny 14}} (67pt,185pt); 
\draw [ggrey,-stealthnew,arrowhead=6pt,shorten >=5pt] (108pt,-56pt) to node [rectangle,draw,fill=white,sloped,inner sep=1pt] {{\tiny 15}} (150pt,102pt); 
\draw [ggrey,-stealthnew,arrowhead=6pt,shorten >=5pt] (108pt,344pt) to node [rectangle,draw,fill=white,sloped,inner sep=1pt] {{\tiny 15}} (150pt,502pt); 
\draw [ggrey,-stealthnew,arrowhead=6pt,shorten >=5pt] (108pt,344pt) to node [rectangle,draw,fill=white,sloped,inner sep=1pt] {{\tiny 16}} (267pt,385pt); 
\draw [ggrey,-stealthnew,arrowhead=6pt,shorten >=5pt] (-56pt,108pt) to node [rectangle,draw,fill=white,sloped,inner sep=1pt] {{\tiny 17}} (26pt,26pt); 
\draw [ggrey,-stealthnew,arrowhead=6pt,shorten >=5pt] (344pt,108pt) to node [rectangle,draw,fill=white,sloped,inner sep=1pt] {{\tiny 17}} (426pt,26pt); 
\draw [ggrey,-stealthnew,arrowhead=6pt,shorten >=5pt] (344pt,508pt) to node [rectangle,draw,fill=white,sloped,inner sep=1pt] {{\tiny 18}} (108pt,344pt); 
\draw [ggrey,-stealthnew,arrowhead=6pt,shorten >=5pt] (344pt,108pt) to node [rectangle,draw,fill=white,sloped,inner sep=1pt] {{\tiny 18}} (108pt,-56pt); 
\draw [ggrey,-stealthnew,arrowhead=6pt,shorten >=5pt] (-56pt,108pt) to node [rectangle,draw,fill=white,sloped,inner sep=1pt] {{\tiny 19}} (108pt,344pt); 
\draw [ggrey,-stealthnew,arrowhead=6pt,shorten >=5pt] (344pt,108pt) to node [rectangle,draw,fill=white,sloped,inner sep=1pt] {{\tiny 19}} (508pt,344pt); 
\draw [ggrey,-stealthnew,arrowhead=6pt,shorten >=5pt] (467pt,185pt) to node [rectangle,draw,fill=white,sloped,inner sep=1pt] {{\tiny 20}} (344pt,108pt); 
\draw [ggrey,-stealthnew,arrowhead=6pt,shorten >=5pt] (67pt,185pt) to node [rectangle,draw,fill=white,sloped,inner sep=1pt] {{\tiny 20}} (-56pt,108pt); 
\draw [ggrey,-stealthnew,arrowhead=6pt,shorten >=5pt] (67pt,185pt) to node [rectangle,draw,fill=white,sloped,inner sep=1pt] {{\tiny 21}} (191pt,261pt); 
\draw [ggrey,-stealthnew,arrowhead=6pt,shorten >=5pt] (150pt,102pt) to node [rectangle,draw,fill=white,sloped,inner sep=1pt] {{\tiny 22}} (26pt,26pt); 
\draw [ggrey,-stealthnew,arrowhead=6pt,shorten >=5pt] (150pt,102pt) to node [rectangle,draw,fill=white,sloped,inner sep=1pt] {{\tiny 23}} (273pt,179pt); 
\draw [ggrey,-stealthnew,arrowhead=6pt,shorten >=5pt] (267pt,-15pt) to node [rectangle,draw,fill=white,sloped,inner sep=1pt] {{\tiny 24}} (344pt,108pt); 
\draw [ggrey,-stealthnew,arrowhead=6pt,shorten >=5pt] (267pt,385pt) to node [rectangle,draw,fill=white,sloped,inner sep=1pt] {{\tiny 24}} (344pt,508pt); 
\draw [ggrey,-stealthnew,arrowhead=6pt,shorten >=5pt] (267pt,385pt) to node [rectangle,draw,fill=white,sloped,inner sep=1pt] {{\tiny 25}} (191pt,261pt); 
\draw [ggrey,-stealthnew,arrowhead=6pt,shorten >=5pt] (191pt,261pt) to node [rectangle,draw,fill=white,sloped,inner sep=1pt] {{\tiny 26}} (350pt,302pt); 
\draw [ggrey,-stealthnew,arrowhead=6pt,shorten >=5pt] (191pt,261pt) to node [rectangle,draw,fill=white,sloped,inner sep=1pt] {{\tiny 27}} (108pt,344pt); 
\draw [ggrey,-stealthnew,arrowhead=6pt,shorten >=5pt] (191pt,261pt) to node [rectangle,draw,fill=white,sloped,inner sep=1pt] {{\tiny 28}} (150pt,102pt); 
\node at (26pt,26pt) [circle,draw=ggrey,draw,fill=white,minimum size=10pt,inner sep=1pt,text=ggrey] {\mbox{\tiny $0$}};  
\node at (350pt,302pt) [circle,draw=ggrey,draw,fill=white,minimum size=10pt,inner sep=1pt,text=ggrey] {\mbox{\tiny $1$}}; 
\node at (273pt,179pt) [circle,draw=ggrey,draw,fill=white,minimum size=10pt,inner sep=1pt,text=ggrey] {\mbox{\tiny $2$}}; 
\node at (308pt,144pt) [circle,draw=ggrey,draw,fill=white,minimum size=10pt,inner sep=1pt,text=ggrey] {\mbox{\tiny $3$}};  
\node at (108pt,344pt) [circle,draw=ggrey,draw,fill=white,minimum size=10pt,inner sep=1pt,text=ggrey] {\mbox{\tiny $4$}}; 
\node at (344pt,108pt) [circle,draw=ggrey,draw,fill=white,minimum size=10pt,inner sep=1pt,text=ggrey] {\mbox{\tiny $5$}}; 
\node at (67pt,185pt) [circle,draw=ggrey,draw,fill=white,minimum size=10pt,inner sep=1pt,text=ggrey] {\mbox{\tiny $6$}}; 
\node at (150pt,102pt) [circle,draw=ggrey,draw,fill=white,minimum size=10pt,inner sep=1pt,text=ggrey] {\mbox{\tiny $7$}}; 
\node at (267pt,385pt) [circle,draw=ggrey,draw,fill=white,minimum size=10pt,inner sep=1pt,text=ggrey] {\mbox{\tiny $8$}}; 
\node at (191pt,261pt) [circle,draw=ggrey,draw,fill=white,minimum size=10pt,inner sep=1pt,text=ggrey] {\mbox{\tiny $9$}}; 
\draw [very thick] (95pt,24pt) -- (28pt,-43pt); 
\draw [very thick] (95pt,424pt) -- (28pt,357pt); 
\draw [very thick] (12pt,106pt) -- (109pt,143pt); 
\draw [very thick] (309pt,343pt) -- (348pt,440pt); 
\draw [very thick] (309pt,-57pt) -- (348pt,40pt); 
\draw [very thick] (309pt,343pt) -- (428pt,357pt); 
\draw [very thick] (-91pt,343pt) -- (28pt,357pt); 
\draw [very thick] (377pt,275pt) -- (271pt,247pt); 
\draw [very thick] (177pt,75pt) -- (230pt,89pt); 
\draw [very thick] (377pt,275pt) -- (363pt,222pt); 
\draw [very thick] (205pt,181pt) -- (271pt,247pt); 
\draw [very thick] (230pt,89pt) -- (363pt,222pt); 
\draw [very thick] (253pt,65pt) -- (387pt,199pt); 
\draw [very thick] (377pt,275pt) -- (428pt,357pt); 
\draw [very thick] (-23pt,275pt) -- (28pt,357pt); 
\draw [very thick] (387pt,199pt) -- (363pt,222pt); 
\draw [very thick] (253pt,65pt) -- (230pt,89pt); 
\draw [very thick] (40pt,212pt) -- (122pt,263pt); 
\draw [very thick] (95pt,24pt) -- (177pt,75pt); 
\draw [very thick] (189pt,330pt) -- (240pt,412pt); 
\draw [very thick] (189pt,-70pt) -- (240pt,12pt); 
\draw [very thick] (348pt,40pt) -- (412pt,106pt); 
\draw [very thick] (-52pt,40pt) -- (12pt,106pt); 
\draw [very thick] (253pt,65pt) -- (240pt,12pt); 
\draw [very thick] (387pt,199pt) -- (440pt,212pt); 
\draw [very thick] (-13pt,199pt) -- (40pt,212pt); 
\draw [very thick] (40pt,212pt) -- (12pt,106pt); 
\draw [very thick] (109pt,143pt) -- (122pt,263pt); 
\draw [very thick] (109pt,143pt) -- (95pt,24pt); 
\draw [very thick] (177pt,75pt) -- (205pt,181pt); 
\draw [very thick] (240pt,12pt) -- (348pt,40pt); 
\draw [very thick] (189pt,330pt) -- (309pt,343pt); 
\draw [very thick] (271pt,247pt) -- (309pt,343pt); 
\draw [very thick] (189pt,330pt) -- (122pt,263pt); 
\draw [very thick] (205pt,181pt) -- (109pt,143pt); 
\node at (309pt,343pt) [circle,draw,fill=black,minimum size=6pt,inner sep=1pt]{}; 
\node at (28pt,357pt) [circle,draw,fill=white,minimum size=6pt,inner sep=1pt]{}; 
\node at (95pt,24pt) [circle,draw,fill=black,minimum size=6pt,inner sep=1pt]{}; 
\node at (348pt,40pt) [circle,draw,fill=white,minimum size=6pt,inner sep=1pt]{}; 
\node at (12pt,106pt) [circle,draw,fill=black,minimum size=6pt,inner sep=1pt]{}; 
\node at (109pt,143pt) [circle,draw,fill=white,minimum size=6pt,inner sep=1pt]{}; 
\node at (377pt,275pt) [circle,draw,fill=black,minimum size=6pt,inner sep=1pt]{}; 
\node at (271pt,247pt) [circle,draw,fill=white,minimum size=6pt,inner sep=1pt]{}; 
\node at (205pt,181pt) [circle,draw,fill=black,minimum size=6pt,inner sep=1pt]{}; 
\node at (230pt,89pt) [circle,draw,fill=black,minimum size=6pt,inner sep=1pt]{}; 
\node at (363pt,222pt) [circle,draw,fill=white,minimum size=6pt,inner sep=1pt]{}; 
\node at (177pt,75pt) [circle,draw,fill=white,minimum size=6pt,inner sep=1pt]{}; 
\node at (253pt,65pt) [circle,draw,fill=white,minimum size=6pt,inner sep=1pt]{}; 
\node at (387pt,199pt) [circle,draw,fill=black,minimum size=6pt,inner sep=1pt]{}; 
\node at (40pt,212pt) [circle,draw,fill=white,minimum size=6pt,inner sep=1pt]{}; 
\node at (240pt,12pt) [circle,draw,fill=black,minimum size=6pt,inner sep=1pt]{}; 
\node at (122pt,263pt) [circle,draw,fill=black,minimum size=6pt,inner sep=1pt]{}; 
\node at (189pt,330pt) [circle,draw,fill=white,minimum size=6pt,inner sep=1pt]{}; 
\draw[very thick, dashed] (0pt,0pt) rectangle (400pt,400pt);
 \end{scope}
 \end{tikzpicture}
 \label{fig:LongHex}
 }
      \qquad 
      \qquad
      \subfigure[]{
\begin{tikzpicture}[baseline={(0,-1.5)}] 
\draw (0,0) -- (0,1) -- (-1,2) -- (-2,3) -- (-3,3) -- (-3,2) -- (-2,1) -- (-1,0) -- (0,0); 
\draw (0,0) node[circle,draw,fill=white,minimum size=10pt,inner sep=1pt] {{\tiny1}};
\draw (0,1) node[circle,draw,fill=white,minimum size=10pt,inner sep=1pt] {{\tiny2}};
\draw (-2,3) node[circle,draw,fill=white,minimum size=10pt,inner sep=1pt] {{\tiny3}};
\draw (-3,3) node[circle,draw,fill=white,minimum size=10pt,inner sep=1pt] {{\tiny4}};
\draw (-3,2) node[circle,draw,fill=white,minimum size=10pt,inner sep=1pt] {{\tiny5}};
\draw (-1,0) node[circle,draw,fill=white,minimum size=10pt,inner sep=1pt] {{\tiny6}};
\draw (-2,1) node[circle,draw,fill=white,minimum size=10pt,inner sep=1pt] {{\tiny7}};
\draw (-2,2) node[circle,draw,fill=white,minimum size=10pt,inner sep=1pt] {{\tiny8}};
\draw (-1,1) node[circle,draw,fill=white,minimum size=10pt,inner sep=1pt] {{\tiny9}};
\draw (-1,2) node[circle,draw,fill=white,minimum size=10pt,inner sep=1pt] {{\tiny10}};
\end{tikzpicture} 
\label{fig:LongHexPolygon}
}
          \caption{(a) A dimer model $\Gamma$ and its dual quiver $Q$; (b) The characteristic polygon $\Delta(\Gamma)$.}
          \label{fig:LHDimerandFan}
  \end{figure}
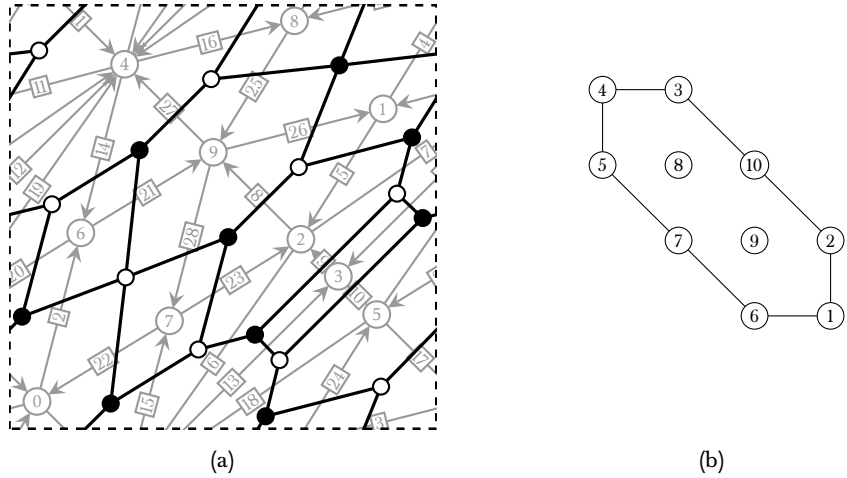
  
 \end{example}
 \subsection{Consistency}
 A dimer model $\Gamma$ determines a graph in $\mathbb{T}$ with nodes $\Gamma_0$ and edges $\Gamma_1$. A \emph{zig-zag path} on $\Gamma$ is a path in this graph that turns maximally right at every white node and maximally left at every black node. Each zig-zag path $\gamma$ is necessarily periodic, so each determines a class $[\gamma]\in H_1(\TT,\ZZ)$. Each edge $\mathfrak{e}\in \Gamma_1$ determines two zig-zag paths: one traversing the edge in each direction. Following Ishii--Ueda~\cite[Definition~3.5]{IshiiUeda11}, we say that $\Gamma$ is \emph{consistent} if: there is no homologically trivial zig-zag path; no zig-zag path intersects itself on the universal cover; and no pair of zig-zag paths intersect each other on the universal cover in the same direction more than once. For alternative definitions of consistency and for a comparison between them, see \emph{ibid.}\ or Bocklandt~\cite{Bocklandt12}; for a single paragraph that summarises the situation, see \cite[Remark~2.1]{BCQ15}. 
 
 Combining the work of Broomhead~\cite[Lemma~5.6]{Broomhead12} with results by Ishii--Ueda~\cite[Proposition~6.3]{IshiiUeda08} and Craw--Quintero-V\'elez~\cite[Theorem~3.15]{CQV11} leads to the following result:
 
 \begin{proposition}
 Let $\Gamma$ be a consistent dimer model. The centre $Z(A)$ of the Jacobian algebra of $\Gamma$ is isomorphic to the semigroup algebra $\kk[\sigma_0^\vee\cap M]$. In particular, the Gorenstein toric threefold determined by the characteristic polygon satisfies $X\cong \Spec Z(A)$.  
 \end{proposition}

  \subsection{Moduli of quiver representations}
 A \emph{representation of the quiver} $Q$ is a collection $\{V_i\mid i\in Q_0\}$ of finite dimensional $\kk$-vector spaces and a collection $\{v_a \colon V_{\tail(a)} \to V_{\head(a)}\mid a\in Q_1\}$ of $\kk$-linear maps. A representation $V$ of $Q$ is said to satisfy the relations $\{p_a^+-p_a^-\mid a\in Q_1\}$ from \eqref{eqn:Jacobian} if and only if the corresponding linear combination of $\kk$-linear maps from $V_{\head(a)}$ to $V_{\tail(a)}$ is zero. A morphism $f \colon V \to V^\prime$ of representations of $Q$ is a collection $\{f_i \colon V_i \to V^\prime_i \mid i \in Q_0 \}$ of $\kk$-linear maps such that $v^\prime_a f_{\tail(a)} = f_{\head(a)}v_a$ for all $a \in Q_1$. The category of representations of $Q$ satisfying the relations from \eqref{eqn:Jacobian} is equivalent to the abelian category $\modA$ of finite-dimensional left modules over the Jacobian algebra $A$. The dimension vector of a representation $V$ is the vector $\dim(V) \in \NN^{Q_0}$ whose $i$th component is $\dim (V_i)$ for $i\in Q_0$. Each vertex $i \in Q_0$ defines a simple object $S_i=\kk e_i$ in $\modA$ called the \emph{vertex simple} for $i\in Q_0$; as a representation of $Q$, this has $V_i=\kk$ and $V_j=0$ for $j\neq i$, where the maps $v_a$ are zero for all $a\in Q_1$.
 
 We are interested primarily in representations of dimension vector $\vv:=(1,1,\dots, 1)\in \NN^{Q_0}$. Consider the rational vector space
\[
\Theta = \Big\{\theta \in \Hom\big(\ZZ^{Q_0},\mathbb{Q}\big) \mid \theta(\vv) = 0\Big\}.
\]
Given a representation $V$ of $Q$, define $\theta(V):=\theta(\dim V)$ for $\theta\in \Theta$. A representation $V$ is \emph{$\theta$-semistable} if $\theta(V)=0$ and every subrepresentation $V^\prime\subset V$ satisfies $\theta(V^\prime)\geq 0$, and it is \emph{$\theta$-stable} if these inequalities are strict for every nonzero, proper subrepresentation $V^\prime\subset V$. These notions apply  to representations satisfying the relations from \eqref{eqn:Jacobian}, so the equivalence of abelian categories mentioned above gives us the notion of $\theta$-(semi)stability for $A$-modules. For any $\theta\in \Theta$, King~\cite[Proposition~5.2]{King94} constructs the coarse moduli space $\overline{\mathcal{M}_\theta}$ of S-equivalence classes of $\theta$-semistable $A$-modules of dimension vector $\vv$ using GIT. 
We say that $\theta\in \Theta$ is \emph{generic} if every $\theta$-semistable $A$-module of dimension vector $v$ is $\theta$-stable, in which case $\overline{\mathcal{M}_\theta}$ coincides with the fine moduli space $\mathcal{M}_\theta$ of isomorphism classes of $\theta$-stable $A$-modules of dimension vector $\vv$. The universal family of $\theta$-stable $A$-modules on $\mathcal{M}_\theta$ is a tautological locally free sheaf 
\[
 T=\bigoplus_{i\in Q_0} L_i,
\] 
where $L_i$ has rank one for all $i\in Q_0$, together with a tautological $\kk$-algebra homomorphism 
\begin{equation}
\label{eqn:tauthomo}
\phi\colon A\rightarrow \End_{\mathscr{O}_{\mathcal{M}_\theta}}(T).    
\end{equation}
 We fix once and for all a vertex of the quiver that we denote $0 \in Q_0$, and we normalise the tautological bundle by fixing $L_0\cong \mathscr{O}_{\mathcal{M}_\theta}$. For each closed point $y\in \mathcal{M}_\theta$, the fibre $T_y=T\otimes \mathcal{O}_y$ of the tautological bundle $T$ over $y$ is a $\theta$-stable $A$-module of dimension vector $\vv$, where the $A$-module structure is obtained by restriction from the tautological maps $\phi(a)\colon L_{\tail(a)}\to L_{\head(a)}$ for $a\in Q_1$.
 
 \subsection{Toric VGIT for consistent dimer models}
  Assume from now on that $\Gamma$ is a consistent dimer model in $\mathbb{T}$. In this case, the moduli spaces $\mathcal{M}_\theta$ can be constructed using toric geometry. To describe the relevant toric GIT quotients, note that the homology $H_*(\mathbb{T},\ZZ)$ is computed by the complex 
 \[
 \ZZ^{Q_2}\stackrel{\partial_2}{\longrightarrow} \ZZ^{Q_1}\stackrel{\partial_1}{\longrightarrow} \ZZ^{Q_0}
 \]
 with maps given by $\partial_2(f) = \Sigma_{a\subseteq \partial f} a$ for $f\in Q_2$ and $\partial_1(a) = \head(a)-\tail(a)$ for $a\in Q_1$. Since $Q$ is connected, the sublattice $B:=\Image(\partial_1)\subset \ZZ^{Q_0}$ has corank one. Mozgovoy--Reineke~\cite[Lemma~3.3]{MozgovoyReineke10} show that 
 \[
 \Lambda:=\ZZ^{Q_1}/\langle \partial_2(f)-\partial_2(f^\prime) \mid f, f^\prime\in Q_2\rangle
 \]
 is a free abelian group (when $\Gamma$ admits a perfect matching which is the case since $\Gamma$ is consistent); an explicit list of generators of $\Lambda$ is given in \cite[Lemma~3.14]{CQV11}. The incidence map  $\partial_1\colon \ZZ^{Q_1}\to B$ of the quiver $Q$ factors through the quotient map  $\wt\colon \ZZ^{Q_1}\to \Lambda$, giving rise to a commutative diagram
  \begin{equation}
   \label{eqn:fgdiagram}
\begin{tikzcd}
  & & \ZZ^{Q_1} \ar[d,swap,"{\wt}"] \ar[dr,"{\partial_1}"]& \\
 0 \ar[r]& M\ar[r]&  \Lambda \ar[r,"d"] & B \ar[r]& 0
 \end{tikzcd}
\end{equation}
 of lattices. 
 It follows that the map of semigroup algebras $\kk[\ZZ^{Q_1}]\to \kk[\Lambda]$ induced by the quotient map $\wt$ is compatible with the $B$-gradings induced by $\partial_1$ and $d$. In geometric terms, this means the inclusion of algebraic tori $\Spec \kk[\Lambda]\rightarrow \Spec \kk[\ZZ^{Q_1}]$ induced by $\wt$ is equivariant with respect to the action of the algebraic torus $T_B:=\Hom(B,\kk^\times)$ whose weights are encoded by $\partial_1$ and $d$.  If we write $\Lambda^+$ for the image under $\wt$ of the subsemigroup $\NN^{Q_1}\subset \ZZ^{Q_1}$, then the induced closed immersion 
 \begin{equation}
     \label{eqn:V}
 V:= \Spec \kk[\Lambda^+]\hookrightarrow \mathbb{C}^{Q_1}:= \Spec \kk[\NN^{Q_1}]
 \end{equation}
 is equivariant with respect to the $T_B$-action. For any character $\theta\in B = \Hom(T_B,\kk^\times)$, write 
 \[
 V\git_\theta T_B:= \Proj \left(\bigoplus_{j\geq 0} \kk[\Lambda^+]_{j\theta}\right)
 \]
 for the GIT quotient, where $\kk[\Lambda^+]_{j\theta}$ denotes the vector space of $j\theta$-semi-invariant elements of $\kk[\Lambda^+]$. Since  $\Gamma$ is consistent,  the study of the moduli spaces $\mathcal{M}_\theta$ by Ishii--Ueda~\cite{IshiiUeda08} can be interpreted directly as variation of GIT quotient for the toric varieties $V\git_\theta T_B$ as in Mozgovoy~\cite[Section~4]{Mozgovoy09} or  Craw--Quintero-V\'{e}lez~\cite{CQV11}; for a one paragraph summary of the argument, see \cite[Remark~2.4]{BCQ15}.  
 
  \begin{proposition}
  Assume that $\Gamma$ is consistent. For any generic $\theta\in \Theta$, there is a commutative diagram
   \begin{equation}
   \label{eqn:crepantres}
\begin{tikzcd}
 \mathcal{M}_\theta  \ar[d] \ar[r,"{\sim}"] & V\git_\theta T_B \ar[d]\\
 X  \ar[r,"\sim"] & V\git_0 T_B
 \end{tikzcd}
\end{equation}
  of toric varieties in which the horizontal arrows are isomorphisms, the vertical arrows are projective, crepant resolutions, and where the right-hand arrow is obtained by variation of GIT quotient. 
  \end{proposition}
  
  \subsection{The (dual) tautological bundle is a tilting bundle}
 To state an important result linking the geometry of the moduli spaces $\mathcal{M}_\theta$ with the Jacobian algebra $A$ of a consistent dimer model, let $D^b(\coh(\mathcal{M}_{\theta}))$ denote the bounded derived category of coherent sheaves on $\mathcal{M}_\theta$ and $D^b(\modA)$ the bounded derived category of finitely-generated left $A$-modules.
 
 \begin{theorem}[Ishii--Ueda~\cite{IshiiUeda15}]
 \label{thm:tilting}
 Let $\Gamma$ be a consistent dimer model. For generic $\theta\in \Theta$, let $T$ denote the tautological bundle on $\mathcal{M}_\theta$. The tautological $\kk$-algebra homomorphism $\phi\colon A\to \End(T)$ is an isomorphism, and the functor
 \begin{equation}
    \label{eqn:Phitheta}
\Phi_\theta(-) = \mathbf{R}\!\Hom_{\mathscr{O}_{\mathcal{M}_{\theta}}}(T^\vee,-)\colon D^b(\coh(\mathcal{M}_{\theta})) \longrightarrow D^b(\modA)
\end{equation}
induced by the dual bundle $T^\vee:=\mathcal{H}om(T,\mathcal{O}_{\mathcal{M}_\theta})$ is an exact equivalence of triangulated categories.
 \end{theorem}
 
 In fact, \cite{IshiiUeda15} show that $\mathbf{R}\!\Hom(T,-)$ is an equivalence, \emph{i.\,e.}\ that $T$ itself is a tilting bundle. However, the fact that $T^\vee$ is tilting follows as in \cite[Lemma~3.2]{BCQ15}, and the functor $\Phi_\theta$ from \eqref{eqn:Phitheta} has some especially nice properties. For example, for each point $y\in \mathcal{M}_\theta$, the skyscraper sheaf $\mathcal{O}_y$ is sent by $\Phi_\theta$ to the $A$-module
 \begin{equation}
     \label{eqn:derivedOy}
  \Phi_\theta(\mathcal{O}_y) = \mathbf{R}\Gamma\circ  \mathbf{R}\mathcal{H}om_{\mathscr{O}_{\mathcal{M}_{\theta}}}(\mathcal{O}_Y,T\otimes \mathcal{O}_y) = \mathbf{R}\Gamma(T\otimes \mathcal{O}_y)= \Gamma(T\otimes \mathcal{O}_y) \cong T_y
 \end{equation}
 obtained as the fibre of the tautological bundle $T$ over $y$ (here and here alone, $\Gamma$ denotes global sections). 
 
   Theorem~\ref{thm:tilting} also facilitates our understanding of the tautological line bundles $\{L_i \mid i\in Q_0\}$ on $\mathcal{M}_\theta$. Indeed, for each $i, j\in Q_0$, the tautological isomorphism $\phi\colon A\to \End(T)$ induces an isomorphism between the space $e_jAe_i$ spanned by classes of paths in $Q$ from $i$ to $j$ and the $\kk$-vector space $\Hom(L_i,L_j)$. Under this isomorphism, each $a\in Q_1$ is assigned to an effective torus-invariant divisor $\div(a)$ obtained as the divisor of zeroes of a section of  $L_{\head(a)}\otimes L_{\tail(a)}^{-1}$.  Explicitly, for each ray $\rho\in \Sigma_\theta(1)$ in the toric fan of $\mathcal{M}_\theta$, if we write $\Pi_\rho$ and $D_\rho$ for the corresponding $\theta$-stable perfect matching in $\Gamma$ and torus-invariant divisor in $\mathcal{M}_\theta$ respectively, then \cite[Theorem~4.2]{BenderMozgovoy09} and \cite[Lemma~4.1]{IshiiUeda15} give that $\div(a)$
  is the divisor of zeroes of the section 
   \begin{equation}
       \label{eqn:xa}
   t^{\div(a)}:= \prod_{\{\rho\in \Sigma_\theta(1) \mid \mathfrak{e}_a\in \Pi_\rho\}} t_\rho\in H^0(L_{\head(a)}\otimes L_{\tail(a)}^{-1}),
   \end{equation}
where $\mathfrak{e}_a$ is the edge in the dimer model dual to the arrow $a$; here, we write the section in the variables of the Cox ring $\kk[t_\rho\mid \rho\in \Sigma_\theta(1)]$ of $\mathcal{M}_\theta$. For $i=0$, if we choose any path $p$ in the quiver $Q$ from $0$ to $j\in Q_0$ and write $\div(p)$ for the sum of the divisors $\div(a)$ indexed by arrows $a$ in the support of $p$, then $\div(p)$ is the divisor of zeroes of the section $t^{\div(p)}:=\prod_{a\in \supp(p)} t^{\div(a)}\in H^0(L_j)$ and hence $   L_j\cong \mathcal{O}_{\mathcal{M}_\theta}\big(\div(p)\big)$.  
 
 \begin{example}
 Consider again the consistent dimer model $\Gamma$ from Example~\ref{exa:LongHex}. The stability parameter $\theta=(\theta_i)\in \Theta$ satisfying $\theta_i>0$ for all $i>0$ is generic. After listing all 10 of the $\theta$-stable perfect matchings (see \cite[Figure~1.5]{TapiaAmador15}), it is straightforward to label each arrow in $Q$ with the section from Equation \eqref{eqn:xa} as shown in Figure~\ref{fig:LHDivisorsTriangulation}(a). The triangulation of the characteristic polygon $\Delta(\Gamma)$ that determines the toric fan $\Sigma_\theta$ of $\mathcal{M}_\theta$ is shown in Figure~\ref{fig:LHDivisorsTriangulation}(b). 
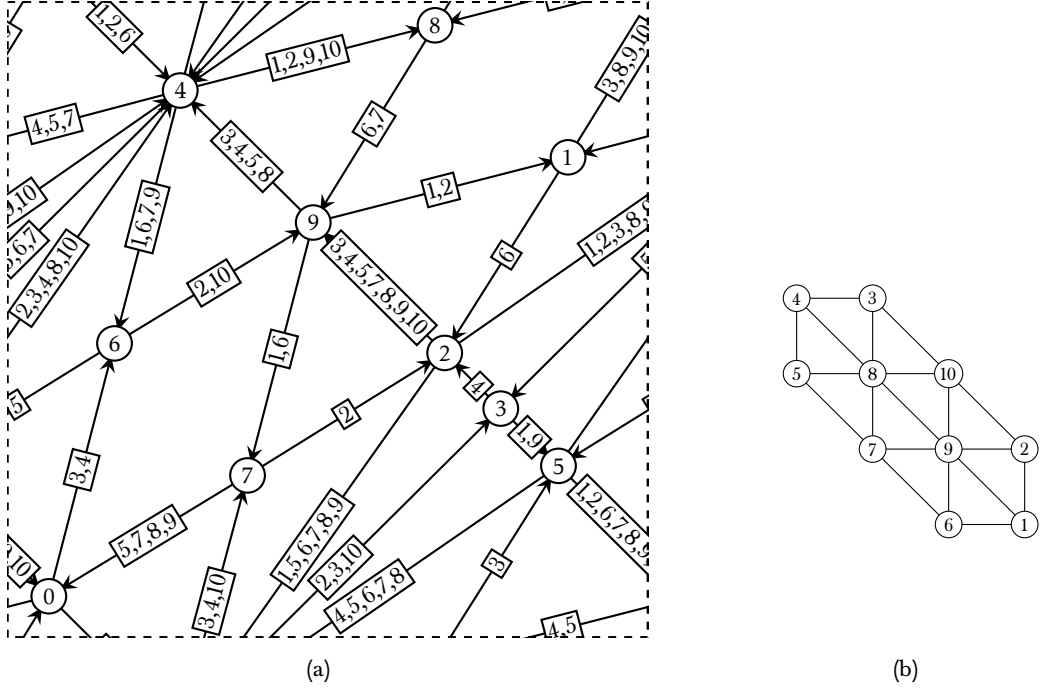
\begin{figure}[!ht]
   \centering
      \subfigure[]{
 \begin{tikzpicture} [thick,scale=0.6, every node/.style={scale=1.3}] 
 
\begin{scope}\clip (0pt,0pt) rectangle (400pt,400pt);
\draw [-stealthnew,arrowhead=6pt,shorten >=5pt] (26pt,26pt) to node [rectangle,draw,fill=white,sloped,inner sep=1pt] {{\tiny 1,2,6}} (108pt,-56pt); 
\draw [-stealthnew,arrowhead=6pt,shorten >=5pt] (26pt,426pt) to node [rectangle,draw,fill=white,sloped,inner sep=1pt] {{\tiny 1,2,6}} (108pt,344pt);  
\draw [-stealthnew,arrowhead=6pt, shorten >=5pt] (26pt,26pt) to node [rectangle,draw,fill=white,sloped,inner sep=1pt] {{\tiny 3,4}} (67pt,185pt); 
\draw [-stealthnew,arrowhead=6pt,shorten >=5pt] (426pt,426pt) to node [rectangle,draw,fill=white,sloped,inner sep=1pt] {{\tiny 4,5}} (267pt,385pt); 
\draw [-stealthnew,arrowhead=6pt,shorten >=5pt] (26pt,26pt) to node [rectangle,draw,fill=white,sloped,inner sep=1pt] {{\tiny 4,5}} (-133pt,-15pt); 
\draw [-stealthnew,arrowhead=6pt,shorten >=5pt] (426pt,26pt) to node [rectangle,draw,fill=white,sloped,inner sep=1pt] {{\tiny 4,5}} (267pt,-15pt); 
\draw [-stealthnew,arrowhead=6pt,shorten >=5pt] (-50pt,-98pt) to node [rectangle,draw,fill=white,sloped,inner sep=1pt] {{\tiny 3,8,9,10}} (26pt,26pt); 
\draw [-stealthnew,arrowhead=6pt,shorten >=5pt] (350pt,302pt) to node [rectangle,draw,fill=white,sloped,inner sep=1pt] {{\tiny 3,8,9,10}} (426pt,426pt); 
\draw [-stealthnew,arrowhead=6pt,shorten >=5pt] (-50pt,302pt) to node [rectangle,draw,fill=white,sloped,inner sep=1pt] {{\tiny 3,8,9,10}} (26pt,426pt);
\draw [-stealthnew,arrowhead=6pt,shorten >=5pt] (350pt,302pt) to node [rectangle,draw,fill=white,sloped,inner sep=1pt] {{\tiny 6}} (273pt,179pt); 
\draw [-stealthnew,arrowhead=6pt,shorten >=5pt] (273pt,579pt) to node [rectangle,draw,fill=white,sloped,inner sep=1pt] {{\tiny 1,5,6,7,8,9}} (108pt,344pt); 
\draw [-stealthnew,arrowhead=6pt,shorten >=5pt] (273pt,179pt) to node [rectangle,draw,fill=white,sloped,inner sep=1pt] {{\tiny 1,5,6,7,8,9}} (108pt,-56pt); 
\draw [-stealthnew,arrowhead=6pt,shorten >=5pt] (-127pt,179pt) to node [rectangle,draw,fill=white,sloped,inner sep=1pt] {{\tiny 1,2,3,8,9,10}} (108pt,344pt); 
\draw [-stealthnew,arrowhead=6pt,shorten >=5pt] (273pt,179pt) to node [rectangle,draw,fill=white,sloped,inner sep=1pt] {{\tiny 1,2,3,8,9,10}} (508pt,344pt); 
\draw [-stealthnew,arrowhead=6pt,shorten >=5pt] (273pt,179pt) to node [rectangle,draw,fill=white,sloped,inner sep=1pt] {{\tiny 3,4,5,7,8,9,10}} (191pt,261pt); 
\draw [-stealthnew,arrowhead=6pt,shorten >=5pt] (308pt,144pt) to node [rectangle,draw,fill=white,sloped,inner sep=1pt,pos=0.4] {{\tiny 4}} (273pt,179pt); 
\draw [-stealthnew,arrowhead=6pt,shorten >=5pt] (308pt,144pt) to node [rectangle,draw,fill=white,sloped,inner sep=1pt] {{\tiny 1,9}} (344pt,108pt); 
\draw [-stealthnew,arrowhead=6pt,shorten >=5pt] (508pt,344pt) to node [rectangle,draw,fill=white,sloped,inner sep=1pt] {{\tiny 4,5,7}} (350pt,302pt); 
\draw [-stealthnew,arrowhead=6pt,shorten >=5pt] (108pt,344pt) to node [rectangle,draw,fill=white,sloped,inner sep=1pt] {{\tiny 4,5,7}} (-50pt,302pt);  
\draw [-stealthnew,arrowhead=6pt,shorten >=5pt] (108pt,344pt) to node [rectangle,draw,fill=white,sloped,inner sep=1pt] {{\tiny 5,6,7}} (-92pt,144pt); 
\draw [-stealthnew,arrowhead=6pt,shorten >=5pt] (508pt,344pt) to node [rectangle,draw,fill=white,sloped,inner sep=1pt] {{\tiny 5,6,7}} (308pt,144pt); 
\draw [-stealthnew,arrowhead=6pt,shorten >=5pt] (108pt,-56pt) to node [rectangle,draw,fill=white,sloped,inner sep=1pt] {{\tiny 2,3,10}} (308pt,144pt); 
\draw [-stealthnew,arrowhead=6pt,shorten >=5pt] (108pt,344pt) to node [rectangle,draw,fill=white,sloped,inner sep=1pt] {{\tiny 2,3,10}} (308pt,544pt); 
\draw [-stealthnew,arrowhead=6pt,shorten >=5pt] (108pt,344pt) to node [rectangle,draw,fill=white,sloped,inner sep=1pt] {{\tiny 1,6,7,9}} (67pt,185pt); 
\draw [-stealthnew,arrowhead=6pt,shorten >=5pt] (108pt,-56pt) to node [rectangle,draw,fill=white,sloped,inner sep=1pt] {{\tiny 3,4,10}} (150pt,102pt); 
\draw [-stealthnew,arrowhead=6pt,shorten >=5pt] (108pt,344pt) to node [rectangle,draw,fill=white,sloped,inner sep=1pt] {{\tiny 3,4,10}} (150pt,502pt); 
\draw [-stealthnew,arrowhead=6pt,shorten >=5pt] (108pt,344pt) to node [rectangle,draw,fill=white,sloped,inner sep=1pt] {{\tiny 1,2,9,10}} (267pt,385pt); 
\draw [-stealthnew,arrowhead=6pt,shorten >=5pt] (-56pt,108pt) to node [rectangle,draw,fill=white,sloped,inner sep=1pt] {{\tiny 1,2,6,7,8,9,10}} (26pt,26pt); 
\draw [-stealthnew,arrowhead=6pt,shorten >=5pt] (344pt,108pt) to node [rectangle,draw,fill=white,sloped,inner sep=1pt] {{\tiny 1,2,6,7,8,9,10}} (426pt,26pt); 
\draw [-stealthnew,arrowhead=6pt,shorten >=5pt] (344pt,508pt) to node [rectangle,draw,fill=white,sloped,inner sep=1pt] {{\tiny 4,5,6,7,8}} (108pt,344pt); 
\draw [-stealthnew,arrowhead=6pt,shorten >=5pt] (344pt,108pt) to node [rectangle,draw,fill=white,sloped,inner sep=1pt] {{\tiny 4,5,6,7,8}} (108pt,-56pt); 
\draw [-stealthnew,arrowhead=6pt,shorten >=5pt] (-56pt,108pt) to node [rectangle,draw,fill=white,sloped,inner sep=1pt] {{\tiny 2,3,4,8,10}} (108pt,344pt); 
\draw [-stealthnew,arrowhead=6pt,shorten >=5pt] (344pt,108pt) to node [rectangle,draw,fill=white,sloped,inner sep=1pt] {{\tiny 2,3,4,8,10}} (508pt,344pt); 
\draw [-stealthnew,arrowhead=6pt,shorten >=5pt] (467pt,185pt) to node [rectangle,draw,fill=white,sloped,inner sep=1pt] {{\tiny 5}} (344pt,108pt); 
\draw [-stealthnew,arrowhead=6pt,shorten >=5pt] (67pt,185pt) to node [rectangle,draw,fill=white,sloped,inner sep=1pt] {{\tiny 5}} (-56pt,108pt); 
\draw [-stealthnew,arrowhead=6pt,shorten >=5pt] (67pt,185pt) to node [rectangle,draw,fill=white,sloped,inner sep=1pt] {{\tiny 2,10}} (191pt,261pt); 
\draw [-stealthnew,arrowhead=6pt,shorten >=5pt] (150pt,102pt) to node [rectangle,draw,fill=white,sloped,inner sep=1pt] {{\tiny 5,7,8,9}} (26pt,26pt); 
\draw [-stealthnew,arrowhead=6pt,shorten >=5pt] (150pt,102pt) to node [rectangle,draw,fill=white,sloped,inner sep=1pt] {{\tiny 2}} (273pt,179pt); 
\draw [-stealthnew,arrowhead=6pt,shorten >=5pt] (267pt,-15pt) to node [rectangle,draw,fill=white,sloped,inner sep=1pt] {{\tiny 3}} (344pt,108pt); 
\draw [-stealthnew,arrowhead=6pt,shorten >=5pt] (267pt,385pt) to node [rectangle,draw,fill=white,sloped,inner sep=1pt] {{\tiny 3}} (344pt,508pt); 
\draw [-stealthnew,arrowhead=6pt,shorten >=5pt] (267pt,385pt) to node [rectangle,draw,fill=white,sloped,inner sep=1pt] {{\tiny 6,7}} (191pt,261pt); 
\draw [-stealthnew,arrowhead=6pt,shorten >=5pt] (191pt,261pt) to node [rectangle,draw,fill=white,sloped,inner sep=1pt] {{\tiny 1,2}} (350pt,302pt); 
\draw [-stealthnew,arrowhead=6pt,shorten >=5pt] (191pt,261pt) to node [rectangle,draw,fill=white,sloped,inner sep=1pt] {{\tiny 3,4,5,8}} (108pt,344pt); 
\draw [-stealthnew,arrowhead=6pt,shorten >=5pt] (191pt,261pt) to node [rectangle,draw,fill=white,sloped,inner sep=1pt] {{\tiny 1,6}} (150pt,102pt); 
\node at (26pt,26pt) [circle,draw,fill=white,minimum size=10pt,inner sep=1pt] {\mbox{\tiny $0$}};  
\node at (350pt,302pt) [circle,draw,fill=white,minimum size=10pt,inner sep=1pt] {\mbox{\tiny $1$}}; 
\node at (273pt,179pt) [circle,draw,fill=white,minimum size=10pt,inner sep=1pt] {\mbox{\tiny $2$}}; 
\node at (308pt,144pt) [circle,draw,fill=white,minimum size=10pt,inner sep=1pt] {\mbox{\tiny $3$}};  
\node at (108pt,344pt) [circle,draw,fill=white,minimum size=10pt,inner sep=1pt] {\mbox{\tiny $4$}}; 
\node at (344pt,108pt) [circle,draw,fill=white,minimum size=10pt,inner sep=1pt] {\mbox{\tiny $5$}}; 
\node at (67pt,185pt) [circle,draw,fill=white,minimum size=10pt,inner sep=1pt] {\mbox{\tiny $6$}}; 
\node at (150pt,102pt) [circle,draw,fill=white,minimum size=10pt,inner sep=1pt] {\mbox{\tiny $7$}}; 
\node at (267pt,385pt) [circle,draw,fill=white,minimum size=10pt,inner sep=1pt] {\mbox{\tiny $8$}}; 
\node at (191pt,261pt) [circle,draw,fill=white,minimum size=10pt,inner sep=1pt] {\mbox{\tiny $9$}}; 
\draw[very thick, dashed] (0pt,0pt) rectangle (400pt,400pt);
\end{scope}
\end{tikzpicture}}
      \qquad 
      \qquad
      \subfigure[]{
\begin{tikzpicture}[baseline={(0,-1.5)}] 
\draw (0,0) -- (0,1) -- (-1,2) -- (-2,3) -- (-3,3) -- (-3,2) -- (-2,1) -- (-1,0) -- (0,0); 
\draw (0,0) -- (-1,1) -- (-2,2) -- (-3,3); 
\draw (0,1) -- (-1,1) -- (-2,1); 
\draw (-1,2) -- (-2,2) -- (-3,2); 
\draw (-2,3) -- (-2,2) -- (-2,1); 
\draw (-1,2) -- (-1,1) -- (-1,0); 
\draw (0,0) node[circle,draw,fill=white,minimum size=10pt,inner sep=1pt] {{\tiny1}};
\draw (0,1) node[circle,draw,fill=white,minimum size=10pt,inner sep=1pt] {{\tiny2}};
\draw (-2,3) node[circle,draw,fill=white,minimum size=10pt,inner sep=1pt] {{\tiny3}};
\draw (-3,3) node[circle,draw,fill=white,minimum size=10pt,inner sep=1pt] {{\tiny4}};
\draw (-3,2) node[circle,draw,fill=white,minimum size=10pt,inner sep=1pt] {{\tiny5}};
\draw (-1,0) node[circle,draw,fill=white,minimum size=10pt,inner sep=1pt] {{\tiny6}};
\draw (-2,1) node[circle,draw,fill=white,minimum size=10pt,inner sep=1pt] {{\tiny7}};
\draw (-2,2) node[circle,draw,fill=white,minimum size=10pt,inner sep=1pt] {{\tiny8}};
\draw (-1,1) node[circle,draw,fill=white,minimum size=10pt,inner sep=1pt] {{\tiny9}};
\draw (-1,2) node[circle,draw,fill=white,minimum size=10pt,inner sep=1pt] {{\tiny10}};
\end{tikzpicture} 
}
          \caption{(a) Divisors labelling arrows in $Q$; (b) Triangulation of $\Delta(\Gamma)$ defining a fan $\Sigma_\theta$.}
          \label{fig:LHDivisorsTriangulation}
  \end{figure}

To illustrate how to compute with the tautological line bundles, let $\sigma_+$ and $\sigma_-$ denote the 3-dimensional cones in $\Sigma_\theta$ generated by the rays $\{\rho_7,\rho_8,\rho_9\}$ and $\{\rho_8, \rho_9,\rho_{10}\}$ respectively, and consider $\tau:=\sigma_+\cap \sigma_-\in \Sigma_\theta(2)$. Let $p_+$ denote the path in $Q$ from vertex 0 to vertex 9 labelled with the section
\[
t^{\div(p_+)}=t_2t_3t_4t_{10}\in H^0(U_{\sigma_+},L_9).
\]
The primitive vector $m\in M$ orthogonal to $\tau$ that is positive on $\sigma_+$ determines the Laurent monomial  $t^m = t_5t_6t_7/t_2t_3t_{10}$ in the variables of the Cox ring, and 
   \[
t^m\cdot t^{\div(p_+)} = t_4t_5t_6t_7\in H^0(U_{\sigma_-},L_9).
   \]
  The path $p_-$ from vertex 0 to 9 labelled with $t^{\div(p_-)}= t_4t_5t_6t_7$ is the generating section of $H^0(U_{\sigma_-},L_9)$, and since $t^m$ appears with exponent 1 in the equation $t^{\div(p_-)} = t^m\cdot t^{\div(p_+)}$, it follows that $L_9$ has degree one on the torus-invariant rational curve $C_\tau$ in $\mathcal{M}_\theta$ defined by the cone $\tau$. 
  \end{example}
  
 \subsection{Classical Reid's recipe}
  \label{sec:RR}
 The first example in the literature of a consistent dimer model $\Gamma$ in a real 2-torus is the hexagonal tiling for a finite abelian subgroup $G\subset \SL(3,\kk)$ studied by Reid~\cite{Reid97}; a more formal construction in the context of dimer models was given by Ueda--Yamazaki~\cite[Section~5]{UY11}. In this case:
 \begin{itemize}
     \item the quiver $Q$ is the McKay quiver of $G$ with vertex set $Q_0 = \Irr(G)$ given by the set of isomorphism classes of irreducible representations of $G$; 
      \item the Jacobian algebra $A=\kk[x,y,z]\rtimes G$ is the skew group algebra of $G$;
     \item the characteristic polygon $\Delta(\Gamma)$ is the junior simplex and the Gorenstein toric threefold is $X\cong \mathbb{C}^3/G$.
 \end{itemize}
 Let $0\in Q_0$ denote the vertex corresponding to the trivial representation of $G$. For the generic stability condition $\theta\in \Theta$ satisfying $\theta_i>0$ for $i\neq 0$, the fine moduli space $\mathcal{M}_\theta$ is isomorphic to the $G$-Hilbert scheme $\ghilb$ that parametrises $G$-invariant subschemes $Z\subset \mathbb{C}^3$ such that $H^0(\mathcal{O}_Z)$ is isomorphic as a $G$-module to $\kk[G]$. The fan $\Sigma$ of the toric variety $\ghilb$ can be constructed as in  Craw--Reid~\cite{CR02}.
 
  \emph{Reid's recipe} was introduced in several examples by Reid~\cite{Reid97} and proved in general by Craw~\cite{Craw05}. The recipe marks interior line segments and lattice points of the triangulation $\Sigma$ with certain nontrivial irreducible representations $i\in Q_0\setminus \{0\}$. To state the result, let  $S_i:=\kk e_i$ denote the vertex simple $A$-module where $e_i\in A$ is the idempotent corresponding to the trivial path at vertex $i$.
    
\begin{recipe}[Classical version of Reid's recipe]\label{rec:RR}
\leavevmode
\begin{enumerate}
\item Let $\tau\in \Sigma(2)$ be an interior line segment and let $m=(m_1, m_2, m_3)\in M$ be the primitive vector in either direction normal to the hyperplane spanned by $\tau$.  The line segment $\tau$ is marked with the vertex $i\in Q_0$ such that the numerator and denominator of the $G$-invariant Laurent monomial $x^{m_1}y^{m_2}z^{m_3}\in \kk[x^{\pm 1}, y^{\pm 1}, z^{\pm 1}]$ lie in the  $i$-character space.
\item Let $\rho\in \Sigma(1)$ be an interior lattice point. A vertex $i\in Q_0$ marks $\rho$ if and only if $S_i$ lies in the socle of every $G$-cluster defined by a torus-invariant point of the divisor $D_\rho$ in $\ghilb$.
\end{enumerate}
\end{recipe}
   
   Surprisingly, every nonzero vertex marks either a unique lattice point or a connected chain of edges that branches at most once \cite[Corollary~4.5]{Craw05}. As an application, a complete set of minimal relations in $\Pic(\ghilb)$ between the tautological lines bundles $\{L_i\mid i\in Q_0\}$ of $\ghilb$ can be given \cite[Theorem~6.1]{Craw05}. The description of the marking of lattice points presented in Recipe~\ref{rec:RR}(2) above is equivalent to the recipe from the original construction, see Craw--Ishii~\cite[Proposition~9.1]{CrawIshii04}.   
      
  Logvinenko~\cite{Logvinenko10} opened up the study of a geometric version of Reid's recipe when he proved a pair of conjectures from Cautis--Logvinenko~\cite{CautisLogvinenko09} in the special case when $\mathbb{C}^3/G$ has a single isolated singularity. To state these results, consider $\mathcal{M}_\theta=\ghilb$ and the functor $\Phi_\theta$ from \eqref{eqn:Phitheta}. Write the quasi-inverse to $\Phi_\theta$ as
\begin{equation}
\label{eqn:Psi}
\Psi_\theta(-):= T^\vee\ltensor_{A} (-) \;\;\colon D^b(\modA) \longrightarrow D^b\big(\coh(\mathcal{M}_\theta)\big)
\end{equation}
rather than $-\ltensor_{A^{\textrm{op}}} T^\vee$. It is immediate that $\Psi_\theta$ sends the indecomposable projective $A$-module $Ae_i$ to the line bundle $L_i^{-1}$ for each $i\in Q_0$. Cautis--Logvinenko~\cite{CautisLogvinenko09} studied the image under $\Psi_\theta$ of the vertex simple $A$-modules $S_i$, and showed that the object $\Psi_\theta(S_i)$, which a priori is a complex of coherent sheaves,  is in fact quasi-isomorphic to the shift (by 0 or 1) of a coherent sheaf. Logvinenko~\cite{Logvinenko10} went further by establishing the following striking relationship between Reid's recipe and the (support of the) object $\Psi_\theta(S_i)$:

\begin{theorem}[Logvinenko~\cite{Logvinenko10}]
\label{thm:Logvinenko}
Let $\mathbb{C}^3/G$ have a single isolated singularity and let $i\in Q_0$ be any non-zero vertex, \emph{i.\,e.}\ $i$ is not the trivial representation. According to Reid's recipe, vertex $i$ marks either:
\begin{enumerate}
 \item[\one] a unique interior lattice point $\rho\in\Sigma(1)$, and $\Psi_\theta(S_i)$ is quasi-isomorphic to $L_i^{-1}\otimes \mathcal{O}_{D_\rho}$;
 \item[\two] a unique interior line segment $\tau\in \Sigma(2)$, and $\Psi_\theta(S_i)$ is quasi-isomorphic to $L_i^{-1}\otimes \mathcal{O}_{C_\tau}$; 
  \item[\three] two or more interior line segments in $\Sigma(2)$, and $\Psi_\theta(S_i)$ is quasi-isomorphic to an object $\mathcal{F}[1]$, where $\mathcal{F}$ is a coherent sheaf whose support  is the union of all torus-invariant divisors $D_\rho$ such that two line segments $\tau$ in $\Sigma$ containing $\rho$ are marked by vertex $i$.  
\end{enumerate}
\end{theorem}

Cautis--Craw--Logvinenko~\cite{CCL17} later described the mysterious sheaves $\mathcal{F}$ that appear in Theorem~\ref{thm:Logvinenko}\three\ in terms of Reid's recipe and they generalised Theorem~\ref{thm:Logvinenko} to every finite abelian subgroup $G\subset\SL(3,\kk)$, \emph{i.\,e.}\ $\mathbb{C}^3/G$ need not have a single isolated singularity. This complete description of the objects $\Psi_\theta(S_i)$ is called \emph{Derived Reid's recipe}.  Even for this derived category statement, the classical Recipe~\ref{rec:RR} forms an important part of the result and it is essential in defining the notion of a CT-subdivision (see \cite[Definition~4.2]{CCL17}) that provides the key technical tool for the proofs.

\section{Nakamura's jigsaw transformations}
This section studies fundamental hexagons and provides a dimer model analogue of the `$G$-igsaw transformations' introduced originally by Nakamura~\cite{Nakamura01} for a finite abelian subgroup $G\subset \SL(3,\kk)$.

\subsection{Distinguished $\theta$-stable $A$-modules}
 For any generic stability parameter $\theta\in\Theta$, we now describe how to associate a $\theta$-stable $A$-module $M_\sigma$ to each cone in the fan $\Sigma_\theta$ of the toric variety $\mathcal{M}_\theta$.
 
 First, for any ray $\rho\in\Sigma_\theta(1)$, let $O_\rho\subset \mathcal{M}_\theta$ denote the corresponding two-dimensional torus-orbit. For any point $y\in O_\rho$, let $(v_a)\in \mathbb{C}^{Q_1}$ be any quiver representation that determines the $\theta$-stable $A$-module $T_y$ obtained as the fibre of $T$ over the point $y$. Ishii--Ueda~\cite[Lemma~6.1]{IshiiUeda08} prove that the set of edges 
\begin{equation}
\label{eqn:pm}
\Pi_\rho:=\{\e_a\in \Gamma_1 \mid v_a=0\}
\end{equation}
 in $\Gamma$ is a perfect matching that is independent of the choice of point $y\in O_\rho$ and the choice of representative $(v_a)$ of $T_y$, and moreover, every $\theta$-stable perfect matching of $\Dimer$ arises in this way. 
 
  Using these perfect matchings, we associate a $\theta$-stable $A$-module to each cone of the fan $\Sigma_\theta$ as follows. For $\sigma\in\Sigma_\theta$, write $\sigma(1):=\{\rho\in\Sigma_\theta(1)\mid\rho\subseteq\sigma\}$, and for each $a\in Q_1$, define scalars
\begin{equation}
     \label{eqn:AmoduleOfCone}
v_a:=\left\{
\begin{array}{cl}
0 & \text{if } \mathfrak{e}_a\in\bigcup_{\rho\in\sigma(1)} \Pi_\rho\\
1 & \text{otherwise}
\end{array} \right..
\end{equation}
  For example, the zero cone $\sigma$ contains no rays, so $v_a=1$ for all $a\in Q_1$ in that case. For each cone $\sigma\in \Sigma_\theta$, the point $(v_a)\in \mathbb{C}^{Q_1}$ is a representation of the quiver $Q$ of dimension vector $\vv$ that depends only on $\sigma$. Let $M_\sigma$ denote the corresponding $\kk Q$-module of dimension vector $v$. In fact we can say more about $M_\sigma$:

\begin{lemma} 
\label{lem:conepoint}
For $\sigma\in \Sigma_\theta$, there exists a unique point $y\in O_\sigma\subseteq \mathcal{M}_\theta$ such that $M_\sigma\cong T_y$. In particular, $M_\sigma$ is a $\theta$-stable $A$-module of dimension vector $v$.
\end{lemma}
\begin{proof}
This is essentially a restatement of results from Ishii--Ueda~\cite{IshiiUeda08} and Mozgovoy~\cite{Mozgovoy09}. Indeed, for $\sigma\in \Sigma_\theta$, we claim that the point $(v_a)\in \mathbb{C}^{Q_1}$ from \eqref{eqn:AmoduleOfCone} is a $\theta$-stable point of the subscheme $V\subseteq \mathbb{C}^{Q_1}$ from \eqref{eqn:V} and, moreover, the point $y:=[(v_a)]\in \mathcal{M}_\theta$ lies in $O_\sigma$. The proof is a case-by-case analysis according to the dimension of $\sigma$ as follows. For $0\in \Sigma_\theta$, we have $v_a=1$ for all $a\in Q_1$, so $(v_a)\in V\cap (\kk^\times)^{Q_1}$. Every point in $(\kk^\times)^{Q_1}$ is $\theta$-stable for all $\theta\in \Theta$, so $y=[(v_a)]$ satisfies $y\in\bigl(V\cap(\kk^\times)^{Q_1}\bigl)/(\kk^\times)^{Q_0}=O_0\subset \mathcal{M}_\theta$ as required. For a cone $\sigma\in \Sigma_\theta(d)$ of dimension $d>0$, the result was established by \cite[Lemma~6.2]{IshiiUeda08} when $d=1$, by \cite[Corollary~4.19]{Mozgovoy09} when $d=2$ and by both \cite[Section~4]{IshiiUeda08} and \cite[Corollary~4.18]{Mozgovoy09} when $d=3$.
\end{proof}

 \subsection{Fundamental hexagons}
 \label{sec:fundahex}
We now recall the combinatorial description of the torus-invariant $\theta$-stable $A$-modules $M_\sigma$ for $\sigma\in \Sigma_\theta(3)$ by Ishii--Ueda ~\cite[Lemma~4.5]{IshiiUeda08}. Let $\pi\colon\RR^2\rightarrow \mathbb{T}$ denote the universal cover of the two-dimensional torus. The preimage of $\Dimer$ defines a $\ZZ^2$-periodic regular cell decomposition $\widetilde \Dimer$ of $\RR^2$. Similarly, the preimage of the quiver $Q$ dual to $\Dimer$ is a $\ZZ^2$-periodic quiver $\widetilde Q$ in $\RR^2$, such that each vertex $\widetilde{i}$ of $\widetilde Q$ lies inside a unique polygonal tile of $\widetilde \Dimer$.
 
 For a cone $\sigma\in\Sigma_\theta(3)$, let $\rho_0,\rho_1,\rho_2\in\Sigma_\theta(1)$ denote the rays in $\sigma$, and write $\Pi_0, \Pi_1, \Pi_2$ for the corresponding $\theta$-stable perfect matchings. Let $Q^\sigma$ denote the subquiver of $Q$ with vertex set $Q_0$ and with arrow set comprising those $a\in Q_1$ for which the corresponding scalar from \eqref{eqn:AmoduleOfCone} satisfies $v_a \neq 0$; in terms of perfect matchings, this is the set of arrows $a\in Q_1$ such that $\mathfrak{e}_a\not\in\bigcup_{0\leq i\leq 2} \Pi_i$. Fix a vertex $i\in Q_0$. Then for any choice of $\widetilde{i} \in \pi^{-1}(i)$, there is a uniquely defined subquiver $\widetilde{Q^\sigma}$ of $\widetilde Q$ that has $\widetilde{i}$ as a vertex such that $\pi$ identifies the quivers $\widetilde{Q^\sigma}$ and $Q^\sigma$ \cite[Lemma~4.1]{IshiiUeda08}.

\begin{definition} 
\label{def:hexHoneycomb}
For $\sigma\in\Sigma_\theta(3)$, the \emph{fundamental hexagon} $\hex(\sigma)$ is the subset of $\RR^2$ covered by the tiles of $\widetilde{\Dimer}$ dual to the vertices of $\widetilde{Q^{\sigma}}$. Let $\tiling(\sigma)$ denote the $\ZZ^2$-periodic graph obtained as the union of all $\ZZ^2$-translates of the boundary of $\hex(\sigma)$.
\end{definition}

 Note that $\hex(\sigma)$ encodes the same information as the quiver $Q^\sigma$, so it allows one to reconstruct the $A$-module $M_\sigma$ for $\sigma\in \Sigma_\theta(3)$. Also, since the boundary of $\hex(\sigma)$ is a subset of $\tiling(\sigma)$, it is sometimes convenient to regard an edge of the boundary of $\hex(\sigma)$ as an edge of the dimer model $\Dimer$ in the 2-torus $\mathbb{T}$. 
 
\begin{example}
\label{exa:lhfundhex}
Consider the dimer model from Example~\ref{exa:LongHex}, and let $\sigma\in\Sigma_\theta(3)$ be the cone generated by the rays $\rho_8, \rho_9, \rho_{10}$ in the fan $\Sigma_\theta$ shown in Figure~\ref{fig:LHDivisorsTriangulation}(b). The dashed arrows in green in Figure~\ref{fig:8910} illustrate one lift $\widetilde{Q^\sigma}$ to $\RR^2$ of the quiver $Q^\sigma$, while the dashed edges in red illustrate the edges in the boundary of $\hex(\sigma)$. The set $\bigcup_{0\leq \rho\leq 2} \Pi_\rho$ is the union of all edges shown in black and all edges shown in dashed red; note that $\tiling(\sigma)$ is obtained from this set by removing each connected component comprising only a single edge.
\begin{figure}[!ht]
\centering
\begin{tikzpicture} [thick,scale=0.31, every node/.style={scale=1}]
\begin{scope}\clip (0pt,0pt) rectangle (400pt,400pt);

\draw [ggrey,-stealthnew,arrowhead=6pt,shorten >=5pt] (26pt,26pt) to node [rectangle,draw,fill=white,sloped,inner sep=1pt] {{\tiny 1}} (108pt,-56pt); 
\draw [ggrey,-stealthnew,arrowhead=6pt,shorten >=5pt] (26pt,426pt) to node [rectangle,draw,fill=white,sloped,inner sep=1pt] {{\tiny 1}} (108pt,344pt);  
\draw [ggrey,-stealthnew,arrowhead=6pt, shorten >=5pt] (26pt,26pt) to node [rectangle,draw,fill=white,sloped,inner sep=1pt] {{\tiny 2}} (67pt,185pt); 
\draw [ggrey,-stealthnew,arrowhead=6pt,shorten >=5pt] (426pt,426pt) to node [rectangle,draw,fill=white,sloped,inner sep=1pt] {{\tiny 3}} (267pt,385pt); 
\draw [ggrey,-stealthnew,arrowhead=6pt,shorten >=5pt] (26pt,26pt) to node [rectangle,draw,fill=white,sloped,inner sep=1pt] {{\tiny 3}} (-133pt,-15pt); 
\draw [dashed, ggreen,-stealthnew,arrowhead=6pt,shorten >=5pt] (426pt,26pt) to node [rectangle,solid,draw,fill=white,sloped,inner sep=1pt] {{\tiny 3}} (267pt,-15pt); 
\draw [ggrey,-stealthnew,arrowhead=6pt,shorten >=5pt] (350pt,302pt) to node [rectangle,draw,fill=white,sloped,inner sep=1pt] {{\tiny 5}} (273pt,179pt); 
\draw [ggrey,-stealthnew,arrowhead=6pt,shorten >=5pt] (308pt,144pt) to node [rectangle,draw,fill=white,sloped,inner sep=1pt] {{\tiny 9}} (273pt,179pt);  
\draw [ggrey,-stealthnew,arrowhead=6pt,shorten >=5pt] (508pt,344pt) to node [rectangle,draw,fill=white,sloped,inner sep=1pt] {{\tiny 11}} (350pt,302pt); 
\draw [ggrey,-stealthnew,arrowhead=6pt,shorten >=5pt] (108pt,344pt) to node [rectangle,draw,fill=white,sloped,inner sep=1pt] {{\tiny 11}} (-50pt,302pt);  
\draw [ggrey,-stealthnew,arrowhead=6pt,shorten >=5pt] (108pt,344pt) to node [rectangle,draw,fill=white,sloped,inner sep=1pt] {{\tiny 12}} (-92pt,144pt); 
\draw [ggrey,-stealthnew,arrowhead=6pt,shorten >=5pt] (508pt,344pt) to node [rectangle,draw,fill=white,sloped,inner sep=1pt] {{\tiny 12}} (308pt,144pt);  
\draw [dashed, ggreen,-stealthnew,arrowhead=6pt,shorten >=5pt] (467pt,185pt) to node [rectangle,solid,draw,fill=white,sloped,inner sep=1pt] {{\tiny 20}} (344pt,108pt); 
\draw [ggrey,-stealthnew,arrowhead=6pt,shorten >=5pt] (67pt,185pt) to node [rectangle,draw,fill=white,sloped,inner sep=1pt] {{\tiny 20}} (-56pt,108pt); 
\draw [ggrey,-stealthnew,arrowhead=6pt,shorten >=5pt] (150pt,102pt) to node [rectangle,draw,fill=white,sloped,inner sep=1pt] {{\tiny 23}} (273pt,179pt); 
\draw [dashed, ggreen,-stealthnew,arrowhead=6pt,shorten >=5pt] (267pt,-15pt) to node [rectangle,solid,draw,fill=white,sloped,inner sep=1pt] {{\tiny 24}} (344pt,108pt); 
\draw [ggrey,-stealthnew,arrowhead=6pt,shorten >=5pt] (267pt,385pt) to node [rectangle,draw,fill=white,sloped,inner sep=1pt] {{\tiny 24}} (344pt,508pt); 
\draw [ggrey,-stealthnew,arrowhead=6pt,shorten >=5pt] (267pt,385pt) to node [rectangle,draw,fill=white,sloped,inner sep=1pt] {{\tiny 25}} (191pt,261pt); 
\draw [ggrey,-stealthnew,arrowhead=6pt,shorten >=5pt] (191pt,261pt) to node [rectangle,draw,fill=white,sloped,inner sep=1pt] {{\tiny 26}} (350pt,302pt); ; 
\draw [ggrey,-stealthnew,arrowhead=6pt,shorten >=5pt] (191pt,261pt) to node [rectangle,draw,fill=white,sloped,inner sep=1pt] {{\tiny 28}} (150pt,102pt); 
\node at (26pt,26pt) [circle,draw=ggrey,draw,fill=white,minimum size=10pt,inner sep=1pt,text=ggrey] {\mbox{\tiny $0$}};  
\node at (350pt,302pt) [circle,draw=ggrey,draw,fill=white,minimum size=10pt,inner sep=1pt,text=ggrey] {\mbox{\tiny $1$}}; 
\node at (273pt,179pt) [circle,draw=ggrey,draw,fill=white,minimum size=10pt,inner sep=1pt,text=ggrey] {\mbox{\tiny $2$}}; 
\node at (308pt,144pt) [circle,draw=ggrey,draw,fill=white,minimum size=10pt,inner sep=1pt,text=ggrey] {\mbox{\tiny $3$}};  
\node at (108pt,344pt) [circle,draw=ggrey,draw,fill=white,minimum size=10pt,inner sep=1pt,text=ggrey] {\mbox{\tiny $4$}}; 
\node at (344pt,108pt) [circle,draw=black,draw,fill=white,minimum size=10pt,inner sep=1pt,text=black] {\mbox{\tiny $5$}}; 
\node at (67pt,185pt) [circle,draw=ggrey,draw,fill=white,minimum size=10pt,inner sep=1pt,text=ggrey] {\mbox{\tiny $6$}}; 
\node at (150pt,102pt) [circle,draw=ggrey,draw,fill=white,minimum size=10pt,inner sep=1pt,text=ggrey] {\mbox{\tiny $7$}}; 
\node at (267pt,385pt) [circle,draw=ggrey,draw,fill=white,minimum size=10pt,inner sep=1pt,text=ggrey] {\mbox{\tiny $8$}}; 
\node at (191pt,261pt) [circle,draw=ggrey,draw,fill=white,minimum size=10pt,inner sep=1pt,text=ggrey] {\mbox{\tiny $9$}}; 
\draw [very thick] (309pt,343pt) -- (428pt,357pt); 
\draw [very thick] (-91pt,343pt) -- (28pt,357pt); 
\draw [very thick] (177pt,75pt) -- (230pt,89pt); 
\draw [very thick] (377pt,275pt) -- (363pt,222pt); 
\draw [very thick] (205pt,181pt) -- (271pt,247pt); 
\draw [dashed, dbblue,very thick] (253pt,65pt) -- (387pt,199pt); 
\draw [very thick] (253pt,65pt) -- (230pt,89pt); 
\draw [very thick] (40pt,212pt) -- (122pt,263pt); 
\draw [very thick] (95pt,24pt) -- (177pt,75pt); 
\draw [very thick] (189pt,330pt) -- (240pt,412pt); 
\draw [dashed, dbblue, very thick] (189pt,-70pt) -- (240pt,12pt); 
\draw [very thick] (348pt,40pt) -- (412pt,106pt); 
\draw [very thick] (-52pt,40pt) -- (12pt,106pt); 
\draw [dashed, dbblue, very thick] (253pt,65pt) -- (240pt,12pt); 
\draw [dashed, dbblue, very thick] (387pt,199pt) -- (440pt,212pt); 
\draw [very thick] (-13pt,199pt) -- (40pt,212pt); 
\draw [very thick] (109pt,143pt) -- (122pt,263pt); 
\draw [very thick] (109pt,143pt) -- (95pt,24pt); 
\draw [very thick] (189pt,330pt) -- (122pt,263pt); 

\node at (309pt,343pt) [circle,draw,fill=black,minimum size=6pt,inner sep=1pt]{}; 
\node at (28pt,357pt) [circle,draw,fill=white,minimum size=6pt,inner sep=1pt]{}; 
\node at (95pt,24pt) [circle,draw,fill=black,minimum size=6pt,inner sep=1pt]{}; 
\node at (348pt,40pt) [circle,draw,fill=white,minimum size=6pt,inner sep=1pt]{}; 
\node at (12pt,106pt) [circle,draw,fill=black,minimum size=6pt,inner sep=1pt]{}; 
\node at (109pt,143pt) [circle,draw,fill=white,minimum size=6pt,inner sep=1pt]{}; 
\node at (377pt,275pt) [circle,draw,fill=black,minimum size=6pt,inner sep=1pt]{}; 
\node at (271pt,247pt) [circle,draw,fill=white,minimum size=6pt,inner sep=1pt]{}; 
\node at (205pt,181pt) [circle,draw,fill=black,minimum size=6pt,inner sep=1pt]{}; 
\node at (230pt,89pt) [circle,draw,fill=black,minimum size=6pt,inner sep=1pt]{}; 
\node at (363pt,222pt) [circle,draw,fill=white,minimum size=6pt,inner sep=1pt]{}; 
\node at (177pt,75pt) [circle,draw,fill=white,minimum size=6pt,inner sep=1pt]{}; 
\node at (253pt,65pt) [circle,draw,fill=white,minimum size=6pt,inner sep=1pt]{}; 
\node at (387pt,199pt) [circle,draw,fill=black,minimum size=6pt,inner sep=1pt]{}; 
\node at (40pt,212pt) [circle,draw,fill=white,minimum size=6pt,inner sep=1pt]{}; 
\node at (240pt,12pt) [circle,draw,fill=black,minimum size=6pt,inner sep=1pt]{}; 
\node at (122pt,263pt) [circle,draw,fill=black,minimum size=6pt,inner sep=1pt]{}; 
\node at (189pt,330pt) [circle,draw,fill=white,minimum size=6pt,inner sep=1pt]{}; 
\draw[very thick, dashed] (0pt,0pt) rectangle (400pt,400pt);
\end{scope}\end{tikzpicture}
\hspace{-.225cm}
\begin{tikzpicture} [thick,scale=0.31, every node/.style={scale=1}]
\begin{scope}\clip (0pt,0pt) rectangle (400pt,400pt);
\draw [dashed, ggreen,-stealthnew,arrowhead=6pt,shorten >=5pt] (26pt,26pt) to node [rectangle,solid,draw,fill=white,sloped,inner sep=1pt] {{\tiny 1}} (108pt,-56pt); 
\draw [ggrey,-stealthnew,arrowhead=6pt,shorten >=5pt] (26pt,426pt) to node [rectangle,draw,fill=white,sloped,inner sep=1pt] {{\tiny 1}} (108pt,344pt);  
\draw [dashed, ggreen,-stealthnew,arrowhead=6pt, shorten >=5pt] (26pt,26pt) to node [rectangle,solid,draw,fill=white,sloped,inner sep=1pt] {{\tiny 2}} (67pt,185pt); 
\draw [ggrey,-stealthnew,arrowhead=6pt,shorten >=5pt] (426pt,426pt) to node [rectangle,draw,fill=white,sloped,inner sep=1pt] {{\tiny 3}} (267pt,385pt); 
\draw [dashed, ggreen,-stealthnew,arrowhead=6pt,shorten >=5pt] (26pt,26pt) to node [rectangle,solid,draw,fill=white,sloped,inner sep=1pt] {{\tiny 3}} (-133pt,-15pt); 
\draw [ggrey,-stealthnew,arrowhead=6pt,shorten >=5pt] (426pt,26pt) to node [rectangle,draw,fill=white,sloped,inner sep=1pt] {{\tiny 3}} (267pt,-15pt); 
\draw [ggrey,-stealthnew,arrowhead=6pt,shorten >=5pt] (350pt,302pt) to node [rectangle,draw,fill=white,sloped,inner sep=1pt] {{\tiny 5}} (273pt,179pt); 
\draw [ggrey,-stealthnew,arrowhead=6pt,shorten >=5pt] (308pt,144pt) to node [rectangle,draw,fill=white,sloped,inner sep=1pt] {{\tiny 9}} (273pt,179pt);  
\draw [ggrey,-stealthnew,arrowhead=6pt,shorten >=5pt] (508pt,344pt) to node [rectangle,draw,fill=white,sloped,inner sep=1pt] {{\tiny 11}} (350pt,302pt); 
\draw [ggrey,-stealthnew,arrowhead=6pt,shorten >=5pt] (108pt,344pt) to node [rectangle,draw,fill=white,sloped,inner sep=1pt] {{\tiny 11}} (-50pt,302pt);  
\draw [ggrey,-stealthnew,arrowhead=6pt,shorten >=5pt] (108pt,344pt) to node [rectangle,draw,fill=white,sloped,inner sep=1pt] {{\tiny 12}} (-92pt,144pt); 
\draw [ggrey,-stealthnew,arrowhead=6pt,shorten >=5pt] (508pt,344pt) to node [rectangle,draw,fill=white,sloped,inner sep=1pt] {{\tiny 12}} (308pt,144pt);  
\draw [ggrey,-stealthnew,arrowhead=6pt,shorten >=5pt] (467pt,185pt) to node [rectangle,draw,fill=white,sloped,inner sep=1pt] {{\tiny 20}} (344pt,108pt); 
\draw [dashed, ggreen,-stealthnew,arrowhead=6pt,shorten >=5pt] (67pt,185pt) to node [rectangle,solid,draw,fill=white,sloped,inner sep=1pt] {{\tiny 20}} (-56pt,108pt); 
\draw [ggrey,-stealthnew,arrowhead=6pt,shorten >=5pt] (150pt,102pt) to node [rectangle,draw,fill=white,sloped,inner sep=1pt] {{\tiny 23}} (273pt,179pt); 
\draw [ggrey,-stealthnew,arrowhead=6pt,shorten >=5pt] (267pt,-15pt) to node [rectangle,draw,fill=white,sloped,inner sep=1pt] {{\tiny 24}} (344pt,108pt); 
\draw [ggrey,-stealthnew,arrowhead=6pt,shorten >=5pt] (267pt,385pt) to node [rectangle,draw,fill=white,sloped,inner sep=1pt] {{\tiny 24}} (344pt,508pt); 
\draw [ggrey,-stealthnew,arrowhead=6pt,shorten >=5pt] (267pt,385pt) to node [rectangle,draw,fill=white,sloped,inner sep=1pt] {{\tiny 25}} (191pt,261pt); 
\draw [ggrey,-stealthnew,arrowhead=6pt,shorten >=5pt] (191pt,261pt) to node [rectangle,draw,fill=white,sloped,inner sep=1pt] {{\tiny 26}} (350pt,302pt); ; 
\draw [ggrey,-stealthnew,arrowhead=6pt,shorten >=5pt] (191pt,261pt) to node [rectangle,draw,fill=white,sloped,inner sep=1pt] {{\tiny 28}} (150pt,102pt); 
\node at (26pt,26pt) [circle,draw=black,draw,fill=white,minimum size=10pt,inner sep=1pt,text=black] {\mbox{\tiny $0$}};  
\node at (350pt,302pt) [circle,draw=ggrey,draw,fill=white,minimum size=10pt,inner sep=1pt,text=ggrey] {\mbox{\tiny $1$}}; 
\node at (273pt,179pt) [circle,draw=ggrey,draw,fill=white,minimum size=10pt,inner sep=1pt,text=ggrey] {\mbox{\tiny $2$}}; 
\node at (308pt,144pt) [circle,draw=ggrey,draw,fill=white,minimum size=10pt,inner sep=1pt,text=ggrey] {\mbox{\tiny $3$}};  
\node at (108pt,344pt) [circle,draw=ggrey,draw,fill=white,minimum size=10pt,inner sep=1pt,text=ggrey] {\mbox{\tiny $4$}}; 
\node at (344pt,108pt) [circle,draw=ggrey,draw,fill=white,minimum size=10pt,inner sep=1pt,text=ggrey] {\mbox{\tiny $5$}}; 
\node at (67pt,185pt) [circle,draw=black,draw,fill=white,minimum size=10pt,inner sep=1pt,text=black] {\mbox{\tiny $6$}}; 
\node at (150pt,102pt) [circle,draw=ggrey,draw,fill=white,minimum size=10pt,inner sep=1pt,text=ggrey] {\mbox{\tiny $7$}}; 
\node at (267pt,385pt) [circle,draw=ggrey,draw,fill=white,minimum size=10pt,inner sep=1pt,text=ggrey] {\mbox{\tiny $8$}}; 
\node at (191pt,261pt) [circle,draw=ggrey,draw,fill=white,minimum size=10pt,inner sep=1pt,text=ggrey] {\mbox{\tiny $9$}}; 
\draw [very thick] (309pt,343pt) -- (428pt,357pt); 
\draw [very thick] (-91pt,343pt) -- (28pt,357pt); 
\draw [dashed, dbblue, very thick] (177pt,75pt) -- (230pt,89pt); 
\draw [very thick] (377pt,275pt) -- (363pt,222pt); 
\draw [very thick] (205pt,181pt) -- (271pt,247pt); 
\draw [very thick] (253pt,65pt) -- (387pt,199pt); 
\draw [dashed, dbblue, very thick] (253pt,65pt) -- (230pt,89pt); 
\draw [dashed, dbblue, very thick] (40pt,212pt) -- (122pt,263pt); 
\draw [dashed, dbblue, very thick] (95pt,24pt) -- (177pt,75pt); 
\draw [very thick] (189pt,330pt) -- (240pt,412pt); 
\draw [dashed, dbblue, very thick] (189pt,-70pt) -- (240pt,12pt); 
\draw [very thick] (348pt,40pt) -- (412pt,106pt); 
\draw [very thick] (-52pt,40pt) -- (12pt,106pt); 
\draw [dashed, dbblue, very thick] (253pt,65pt) -- (240pt,12pt); 
\draw [very thick] (387pt,199pt) -- (440pt,212pt); 
\draw [dashed, dbblue, very thick] (-13pt,199pt) -- (40pt,212pt); 
\draw [dashed, dbblue, very thick] (109pt,143pt) -- (122pt,263pt); 
\draw [dashed, dbblue, very thick] (109pt,143pt) -- (95pt,24pt); 
\draw [very thick] (189pt,330pt) -- (122pt,263pt); 

\node at (309pt,343pt) [circle,draw,fill=black,minimum size=6pt,inner sep=1pt]{}; 
\node at (28pt,357pt) [circle,draw,fill=white,minimum size=6pt,inner sep=1pt]{}; 
\node at (95pt,24pt) [circle,draw,fill=black,minimum size=6pt,inner sep=1pt]{}; 
\node at (348pt,40pt) [circle,draw,fill=white,minimum size=6pt,inner sep=1pt]{}; 
\node at (12pt,106pt) [circle,draw,fill=black,minimum size=6pt,inner sep=1pt]{}; 
\node at (109pt,143pt) [circle,draw,fill=white,minimum size=6pt,inner sep=1pt]{}; 
\node at (377pt,275pt) [circle,draw,fill=black,minimum size=6pt,inner sep=1pt]{}; 
\node at (271pt,247pt) [circle,draw,fill=white,minimum size=6pt,inner sep=1pt]{}; 
\node at (205pt,181pt) [circle,draw,fill=black,minimum size=6pt,inner sep=1pt]{}; 
\node at (230pt,89pt) [circle,draw,fill=black,minimum size=6pt,inner sep=1pt]{}; 
\node at (363pt,222pt) [circle,draw,fill=white,minimum size=6pt,inner sep=1pt]{}; 
\node at (177pt,75pt) [circle,draw,fill=white,minimum size=6pt,inner sep=1pt]{}; 
\node at (253pt,65pt) [circle,draw,fill=white,minimum size=6pt,inner sep=1pt]{}; 
\node at (387pt,199pt) [circle,draw,fill=black,minimum size=6pt,inner sep=1pt]{}; 
\node at (40pt,212pt) [circle,draw,fill=white,minimum size=6pt,inner sep=1pt]{}; 
\node at (240pt,12pt) [circle,draw,fill=black,minimum size=6pt,inner sep=1pt]{}; 
\node at (122pt,263pt) [circle,draw,fill=black,minimum size=6pt,inner sep=1pt]{}; 
\node at (189pt,330pt) [circle,draw,fill=white,minimum size=6pt,inner sep=1pt]{}; 
\draw[very thick, dashed] (0pt,0pt) rectangle (400pt,400pt);
\end{scope}\end{tikzpicture}\\
\begin{tikzpicture} [thick,scale=0.31, every node/.style={scale=1}]
\begin{scope}\clip (0pt,0pt) rectangle (400pt,400pt);
\draw [ggrey,-stealthnew,arrowhead=6pt,shorten >=5pt] (26pt,26pt) to node [rectangle,draw,fill=white,sloped,inner sep=1pt] {{\tiny 1}} (108pt,-56pt); 
\draw [ggrey,-stealthnew,arrowhead=6pt,shorten >=5pt] (26pt,426pt) to node [rectangle,draw,fill=white,sloped,inner sep=1pt] {{\tiny 1}} (108pt,344pt);  
\draw [ggrey,-stealthnew,arrowhead=6pt, shorten >=5pt] (26pt,26pt) to node [rectangle,draw,fill=white,sloped,inner sep=1pt] {{\tiny 2}} (67pt,185pt); 
\draw [dashed, ggreen,-stealthnew,arrowhead=6pt,shorten >=5pt] (426pt,426pt) to node [rectangle,solid,draw,fill=white,sloped,inner sep=1pt] {{\tiny 3}} (267pt,385pt); 
\draw [ggrey,-stealthnew,arrowhead=6pt,shorten >=5pt] (26pt,26pt) to node [rectangle,draw,fill=white,sloped,inner sep=1pt] {{\tiny 3}} (-133pt,-15pt); 
\draw [ggrey,-stealthnew,arrowhead=6pt,shorten >=5pt] (426pt,26pt) to node [rectangle,draw,fill=white,sloped,inner sep=1pt] {{\tiny 3}} (267pt,-15pt); 
\draw [dashed, ggreen,-stealthnew,arrowhead=6pt,shorten >=5pt] (350pt,302pt) to node [rectangle,solid,draw,fill=white,sloped,inner sep=1pt] {{\tiny 5}} (273pt,179pt); 
\draw [dashed, ggreen,-stealthnew,arrowhead=6pt,shorten >=5pt] (308pt,144pt) to node [rectangle,solid,draw,fill=white,sloped,inner sep=1pt] {{\tiny 9}} (273pt,179pt);  
\draw [dashed, ggreen,-stealthnew,arrowhead=6pt,shorten >=5pt] (508pt,344pt) to node [rectangle,solid,draw,fill=white,sloped,inner sep=1pt] {{\tiny 11}} (350pt,302pt); 
\draw [ggrey,-stealthnew,arrowhead=6pt,shorten >=5pt] (108pt,344pt) to node [rectangle,draw,fill=white,sloped,inner sep=1pt] {{\tiny 11}} (-50pt,302pt);  
\draw [ggrey,-stealthnew,arrowhead=6pt,shorten >=5pt] (108pt,344pt) to node [rectangle,draw,fill=white,sloped,inner sep=1pt] {{\tiny 12}} (-92pt,144pt); 
\draw [dashed, ggreen,-stealthnew,arrowhead=6pt,shorten >=5pt] (508pt,344pt) to node [rectangle,solid,draw,fill=white,sloped,inner sep=1pt] {{\tiny 12}} (308pt,144pt);  
\draw [ggrey,-stealthnew,arrowhead=6pt,shorten >=5pt] (467pt,185pt) to node [rectangle,draw,fill=white,sloped,inner sep=1pt] {{\tiny 20}} (344pt,108pt); 
\draw [ggrey,-stealthnew,arrowhead=6pt,shorten >=5pt] (67pt,185pt) to node [rectangle,draw,fill=white,sloped,inner sep=1pt] {{\tiny 20}} (-56pt,108pt); 
\draw [dashed, ggreen,-stealthnew,arrowhead=6pt,shorten >=5pt] (150pt,102pt) to node [rectangle,solid,draw,fill=white,sloped,inner sep=1pt] {{\tiny 23}} (273pt,179pt); 
\draw [ggrey,-stealthnew,arrowhead=6pt,shorten >=5pt] (267pt,-15pt) to node [rectangle,draw,fill=white,sloped,inner sep=1pt] {{\tiny 24}} (344pt,108pt); 
\draw [dashed, ggreen,-stealthnew,arrowhead=6pt,shorten >=5pt] (267pt,385pt) to node [rectangle,solid,draw,fill=white,sloped,inner sep=1pt] {{\tiny 24}} (344pt,508pt); 
\draw [dashed, ggreen,-stealthnew,arrowhead=6pt,shorten >=5pt] (267pt,385pt) to node [rectangle,solid,draw,fill=white,sloped,inner sep=1pt] {{\tiny 25}} (191pt,261pt); 
\draw [dashed, ggreen,-stealthnew,arrowhead=6pt,shorten >=5pt] (191pt,261pt) to node [rectangle,solid,draw,fill=white,sloped,inner sep=1pt] {{\tiny 26}} (350pt,302pt); ; 
\draw [dashed, ggreen,-stealthnew,arrowhead=6pt,shorten >=5pt] (191pt,261pt) to node [rectangle,solid, draw,fill=white,sloped,inner sep=1pt] {{\tiny 28}} (150pt,102pt); 
\node at (26pt,26pt) [circle,draw=ggrey,draw,fill=white,minimum size=10pt,inner sep=1pt,text=ggrey] {\mbox{\tiny $0$}};  
\node at (350pt,302pt) [circle,draw=black,draw,fill=white,minimum size=10pt,inner sep=1pt,text=black] {\mbox{\tiny $1$}}; 
\node at (273pt,179pt) [circle,draw=black,draw,fill=white,minimum size=10pt,inner sep=1pt,text=black] {\mbox{\tiny $2$}}; 
\node at (308pt,144pt) [circle,draw=black,draw,fill=white,minimum size=10pt,inner sep=1pt,text=black] {\mbox{\tiny $3$}};  
\node at (108pt,344pt) [circle,draw=ggrey,draw,fill=white,minimum size=10pt,inner sep=1pt,text=ggrey] {\mbox{\tiny $4$}}; 
\node at (344pt,108pt) [circle,draw=ggrey,draw,fill=white,minimum size=10pt,inner sep=1pt,text=ggrey] {\mbox{\tiny $5$}}; 
\node at (67pt,185pt) [circle,draw=ggrey,draw,fill=white,minimum size=10pt,inner sep=1pt,text=ggrey] {\mbox{\tiny $6$}}; 
\node at (150pt,102pt) [circle,draw=black,draw,fill=white,minimum size=10pt,inner sep=1pt,text=black] {\mbox{\tiny $7$}}; 
\node at (267pt,385pt) [circle,draw=black,draw,fill=white,minimum size=10pt,inner sep=1pt,text=black] {\mbox{\tiny $8$}}; 
\node at (191pt,261pt) [circle,draw=black,draw,fill=white,minimum size=10pt,inner sep=1pt,text=black] {\mbox{\tiny $9$}}; 

\draw [very thick] (309pt,343pt) -- (428pt,357pt); 
\draw [very thick] (-91pt,343pt) -- (28pt,357pt); 
\draw [dashed, dbblue, very thick] (177pt,75pt) -- (230pt,89pt); 
\draw [very thick] (377pt,275pt) -- (363pt,222pt); 
\draw [very thick] (205pt,181pt) -- (271pt,247pt); 
\draw [dashed, dbblue, very thick] (253pt,65pt) -- (387pt,199pt); 
\draw [dashed, dbblue, very thick] (253pt,65pt) -- (230pt,89pt); 
\draw [very thick] (40pt,212pt) -- (122pt,263pt); 
\draw [dashed, dbblue, very thick] (95pt,24pt) -- (177pt,75pt); 
\draw [dashed, dbblue, very thick] (189pt,330pt) -- (240pt,412pt); 
\draw [very thick] (189pt,-70pt) -- (240pt,12pt); 
\draw [very thick] (348pt,40pt) -- (412pt,106pt); 
\draw [very thick] (-52pt,40pt) -- (12pt,106pt); 
\draw [very thick] (253pt,65pt) -- (240pt,12pt); 
\draw [dashed, dbblue, very thick] (387pt,199pt) -- (440pt,212pt); 
\draw [very thick] (-13pt,199pt) -- (40pt,212pt); 
\draw [dashed, dbblue, very thick] (109pt,143pt) -- (122pt,263pt); 
\draw [dashed, dbblue, very thick] (109pt,143pt) -- (95pt,24pt); 
\draw [dashed, dbblue, very thick] (189pt,330pt) -- (122pt,263pt); 

\node at (309pt,343pt) [circle,draw,fill=black,minimum size=6pt,inner sep=1pt]{}; 
\node at (28pt,357pt) [circle,draw,fill=white,minimum size=6pt,inner sep=1pt]{}; 
\node at (95pt,24pt) [circle,draw,fill=black,minimum size=6pt,inner sep=1pt]{}; 
\node at (348pt,40pt) [circle,draw,fill=white,minimum size=6pt,inner sep=1pt]{}; 
\node at (12pt,106pt) [circle,draw,fill=black,minimum size=6pt,inner sep=1pt]{}; 
\node at (109pt,143pt) [circle,draw,fill=white,minimum size=6pt,inner sep=1pt]{}; 
\node at (377pt,275pt) [circle,draw,fill=black,minimum size=6pt,inner sep=1pt]{}; 
\node at (271pt,247pt) [circle,draw,fill=white,minimum size=6pt,inner sep=1pt]{}; 
\node at (205pt,181pt) [circle,draw,fill=black,minimum size=6pt,inner sep=1pt]{}; 
\node at (230pt,89pt) [circle,draw,fill=black,minimum size=6pt,inner sep=1pt]{}; 
\node at (363pt,222pt) [circle,draw,fill=white,minimum size=6pt,inner sep=1pt]{}; 
\node at (177pt,75pt) [circle,draw,fill=white,minimum size=6pt,inner sep=1pt]{}; 
\node at (253pt,65pt) [circle,draw,fill=white,minimum size=6pt,inner sep=1pt]{}; 
\node at (387pt,199pt) [circle,draw,fill=black,minimum size=6pt,inner sep=1pt]{}; 
\node at (40pt,212pt) [circle,draw,fill=white,minimum size=6pt,inner sep=1pt]{}; 
\node at (240pt,12pt) [circle,draw,fill=black,minimum size=6pt,inner sep=1pt]{}; 
\node at (122pt,263pt) [circle,draw,fill=black,minimum size=6pt,inner sep=1pt]{}; 
\node at (189pt,330pt) [circle,draw,fill=white,minimum size=6pt,inner sep=1pt]{}; 
\draw[very thick, dashed] (0pt,0pt) rectangle (400pt,400pt);
\end{scope}\end{tikzpicture}
\hspace{-.225cm}
\begin{tikzpicture} [thick,scale=0.31, every node/.style={scale=1}]
\begin{scope}\clip (0pt,0pt) rectangle (400pt,400pt);
\draw [ggrey,-stealthnew,arrowhead=6pt,shorten >=5pt] (26pt,26pt) to node [rectangle,draw,fill=white,sloped,inner sep=1pt] {{\tiny 1}} (108pt,-56pt); 
\draw [dashed, ggreen,-stealthnew,arrowhead=6pt,shorten >=5pt] (26pt,426pt) to node [rectangle,solid, draw,fill=white,sloped,inner sep=1pt] {{\tiny 1}} (108pt,344pt);  
\draw [ggrey,-stealthnew,arrowhead=6pt, shorten >=5pt] (26pt,26pt) to node [rectangle,draw,fill=white,sloped,inner sep=1pt] {{\tiny 2}} (67pt,185pt); 
\draw [ggrey,-stealthnew,arrowhead=6pt,shorten >=5pt] (426pt,426pt) to node [rectangle,draw,fill=white,sloped,inner sep=1pt] {{\tiny 3}} (267pt,385pt); 
\draw [ggrey,-stealthnew,arrowhead=6pt,shorten >=5pt] (26pt,26pt) to node [rectangle,draw,fill=white,sloped,inner sep=1pt] {{\tiny 3}} (-133pt,-15pt); 
\draw [ggrey,-stealthnew,arrowhead=6pt,shorten >=5pt] (426pt,26pt) to node [rectangle,draw,fill=white,sloped,inner sep=1pt] {{\tiny 3}} (267pt,-15pt); 
\draw [ggrey,-stealthnew,arrowhead=6pt,shorten >=5pt] (350pt,302pt) to node [rectangle,draw,fill=white,sloped,inner sep=1pt] {{\tiny 5}} (273pt,179pt); 
\draw [ggrey,-stealthnew,arrowhead=6pt,shorten >=5pt] (308pt,144pt) to node [rectangle,draw,fill=white,sloped,inner sep=1pt] {{\tiny 9}} (273pt,179pt);  
\draw [ggrey,-stealthnew,arrowhead=6pt,shorten >=5pt] (508pt,344pt) to node [rectangle,draw,fill=white,sloped,inner sep=1pt] {{\tiny 11}} (350pt,302pt); 
\draw [dashed, ggreen,-stealthnew,arrowhead=6pt,shorten >=5pt] (108pt,344pt) to node [rectangle,solid,draw,fill=white,sloped,inner sep=1pt] {{\tiny 11}} (-50pt,302pt);  
\draw [dashed, ggreen,-stealthnew,arrowhead=6pt,shorten >=5pt] (108pt,344pt) to node [rectangle,solid,draw,fill=white,sloped,inner sep=1pt] {{\tiny 12}} (-92pt,144pt); 
\draw [ggrey,-stealthnew,arrowhead=6pt,shorten >=5pt] (508pt,344pt) to node [rectangle,draw,fill=white,sloped,inner sep=1pt] {{\tiny 12}} (308pt,144pt);  
\draw [ggrey,-stealthnew,arrowhead=6pt,shorten >=5pt] (467pt,185pt) to node [rectangle,draw,fill=white,sloped,inner sep=1pt] {{\tiny 20}} (344pt,108pt); 
\draw [ggrey,-stealthnew,arrowhead=6pt,shorten >=5pt] (67pt,185pt) to node [rectangle,draw,fill=white,sloped,inner sep=1pt] {{\tiny 20}} (-56pt,108pt); 
\draw [ggrey,-stealthnew,arrowhead=6pt,shorten >=5pt] (150pt,102pt) to node [rectangle,draw,fill=white,sloped,inner sep=1pt] {{\tiny 23}} (273pt,179pt); 
\draw [ggrey,-stealthnew,arrowhead=6pt,shorten >=5pt] (267pt,-15pt) to node [rectangle,draw,fill=white,sloped,inner sep=1pt] {{\tiny 24}} (344pt,108pt); 
\draw [ggrey,-stealthnew,arrowhead=6pt,shorten >=5pt] (267pt,385pt) to node [rectangle,draw,fill=white,sloped,inner sep=1pt] {{\tiny 24}} (344pt,508pt); 
\draw [ggrey,-stealthnew,arrowhead=6pt,shorten >=5pt] (267pt,385pt) to node [rectangle,draw,fill=white,sloped,inner sep=1pt] {{\tiny 25}} (191pt,261pt); 
\draw [ggrey,-stealthnew,arrowhead=6pt,shorten >=5pt] (191pt,261pt) to node [rectangle,draw,fill=white,sloped,inner sep=1pt] {{\tiny 26}} (350pt,302pt); ; 
\draw [ggrey,-stealthnew,arrowhead=6pt,shorten >=5pt] (191pt,261pt) to node [rectangle,draw,fill=white,sloped,inner sep=1pt] {{\tiny 28}} (150pt,102pt); 
\node at (26pt,26pt) [circle,draw=ggrey,draw,fill=white,minimum size=10pt,inner sep=1pt,text=ggrey] {\mbox{\tiny $0$}};  
\node at (350pt,302pt) [circle,draw=ggrey,draw,fill=white,minimum size=10pt,inner sep=1pt,text=ggrey] {\mbox{\tiny $1$}}; 
\node at (273pt,179pt) [circle,draw=ggrey,draw,fill=white,minimum size=10pt,inner sep=1pt,text=ggrey] {\mbox{\tiny $2$}}; 
\node at (308pt,144pt) [circle,draw=ggrey,draw,fill=white,minimum size=10pt,inner sep=1pt,text=ggrey] {\mbox{\tiny $3$}};  
\node at (108pt,344pt) [circle,draw,draw,fill=white,minimum size=10pt,inner sep=1pt,text=black] {\mbox{\tiny $4$}}; 
\node at (344pt,108pt) [circle,draw=ggrey,draw,fill=white,minimum size=10pt,inner sep=1pt,text=ggrey] {\mbox{\tiny $5$}}; 
\node at (67pt,185pt) [circle,draw=ggrey,draw,fill=white,minimum size=10pt,inner sep=1pt,text=ggrey] {\mbox{\tiny $6$}}; 
\node at (150pt,102pt) [circle,draw=ggrey,draw,fill=white,minimum size=10pt,inner sep=1pt,text=ggrey] {\mbox{\tiny $7$}}; 
\node at (267pt,385pt) [circle,draw=ggrey,draw,fill=white,minimum size=10pt,inner sep=1pt,text=ggrey] {\mbox{\tiny $8$}}; 
\node at (191pt,261pt) [circle,draw=ggrey,draw,fill=white,minimum size=10pt,inner sep=1pt,text=ggrey] {\mbox{\tiny $9$}}; 
\draw [very thick] (309pt,343pt) -- (428pt,357pt); 
\draw [very thick] (-91pt,343pt) -- (28pt,357pt); 
\draw [very thick] (177pt,75pt) -- (230pt,89pt); 
\draw [very thick] (377pt,275pt) -- (363pt,222pt); 
\draw [very thick] (205pt,181pt) -- (271pt,247pt); 
\draw [very thick] (253pt,65pt) -- (387pt,199pt); 
\draw [very thick] (253pt,65pt) -- (230pt,89pt); 
\draw [dashed, dbblue, very thick] (40pt,212pt) -- (122pt,263pt); 
\draw [very thick] (95pt,24pt) -- (177pt,75pt); 
\draw [dashed, dbblue, very thick] (189pt,330pt) -- (240pt,412pt); 
\draw [very thick] (189pt,-70pt) -- (240pt,12pt); 
\draw [very thick] (348pt,40pt) -- (412pt,106pt); 
\draw [very thick] (-52pt,40pt) -- (12pt,106pt); 
\draw [very thick] (253pt,65pt) -- (240pt,12pt); 
\draw [very thick] (387pt,199pt) -- (440pt,212pt); 
\draw [dashed, dbblue, very thick] (-13pt,199pt) -- (40pt,212pt); 
\draw [very thick] (109pt,143pt) -- (122pt,263pt); 
\draw [very thick] (109pt,143pt) -- (95pt,24pt); 
\draw [dashed, dbblue, very thick] (189pt,330pt) -- (122pt,263pt); 

\node at (309pt,343pt) [circle,draw,fill=black,minimum size=6pt,inner sep=1pt]{}; 
\node at (28pt,357pt) [circle,draw,fill=white,minimum size=6pt,inner sep=1pt]{}; 
\node at (95pt,24pt) [circle,draw,fill=black,minimum size=6pt,inner sep=1pt]{}; 
\node at (348pt,40pt) [circle,draw,fill=white,minimum size=6pt,inner sep=1pt]{}; 
\node at (12pt,106pt) [circle,draw,fill=black,minimum size=6pt,inner sep=1pt]{}; 
\node at (109pt,143pt) [circle,draw,fill=white,minimum size=6pt,inner sep=1pt]{}; 
\node at (377pt,275pt) [circle,draw,fill=black,minimum size=6pt,inner sep=1pt]{}; 
\node at (271pt,247pt) [circle,draw,fill=white,minimum size=6pt,inner sep=1pt]{}; 
\node at (205pt,181pt) [circle,draw,fill=black,minimum size=6pt,inner sep=1pt]{}; 
\node at (230pt,89pt) [circle,draw,fill=black,minimum size=6pt,inner sep=1pt]{}; 
\node at (363pt,222pt) [circle,draw,fill=white,minimum size=6pt,inner sep=1pt]{}; 
\node at (177pt,75pt) [circle,draw,fill=white,minimum size=6pt,inner sep=1pt]{}; 
\node at (253pt,65pt) [circle,draw,fill=white,minimum size=6pt,inner sep=1pt]{}; 
\node at (387pt,199pt) [circle,draw,fill=black,minimum size=6pt,inner sep=1pt]{}; 
\node at (40pt,212pt) [circle,draw,fill=white,minimum size=6pt,inner sep=1pt]{}; 
\node at (240pt,12pt) [circle,draw,fill=black,minimum size=6pt,inner sep=1pt]{}; 
\node at (122pt,263pt) [circle,draw,fill=black,minimum size=6pt,inner sep=1pt]{}; 
\node at (189pt,330pt) [circle,draw,fill=white,minimum size=6pt,inner sep=1pt]{}; 
\draw[very thick, dashed] (0pt,0pt) rectangle (400pt,400pt);
\end{scope}\end{tikzpicture}
\caption{The red dashed lines form the boundary of the fundamental hexagon.}
\label{fig:8910}
\end{figure}
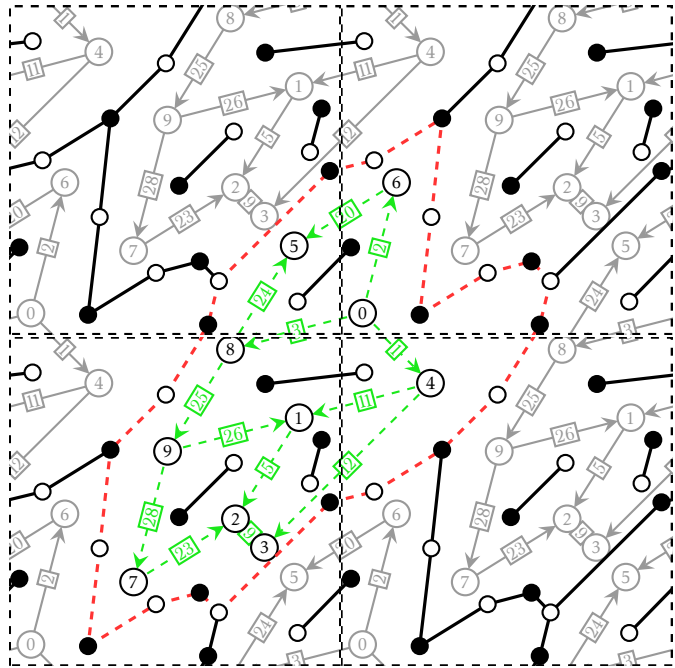
\end{example}

Since the tiles in $\hex(\sigma)$ are obtained by lifting each tile of $\Gamma$ in $\TT$, we see that $\hex(\sigma)$ forms a fundamental domain for the action of $\ZZ^2$ on $\RR^2$, and the boundary of $\hex(\sigma)$ coincides with $\tiling(\sigma)$ when both are viewed as subsets of $\mathbb{T}$. To fully justify the terminology from Definition~\ref{def:hexHoneycomb} and to show how the boundary of  $\hex(\sigma)$ can be calculated directly from perfect matchings, we bring together results of Ishii--Ueda and Mozgovoy in the following statement. 

\begin{proposition} 
\label{prop:IUvalency}
 For $\sigma\in\Sigma_\theta(3)$, the boundary of $\hex(\sigma)$ contains precisely six trivalent nodes of $\tiling(\sigma)$ that, when listed cyclically,  alternate between white nodes of one $\ZZ^2$-orbit and black nodes in a second $\ZZ^2$-orbit; all other nodes of $\tiling(\sigma)$ have valency two. Moreover, $\tiling(\sigma)$ is the unique connected component of the locus $\bigcup_{0\leq \rho\leq 2} \Pi_\rho$ in $\mathbb{T}$ comprising more than a single edge.
 \end{proposition}
 \begin{proof}
 The first statement follows from \cite[Lemma~4.4]{IshiiUeda08} and the proof of \cite[Lemma~4.5]{IshiiUeda08}. For the second statement, Mozgovoy~\cite[Corollary~4.18]{Mozgovoy09} shows that each edge $\mathfrak{e}\in\bigcup_{0\leq \rho\leq 2} \Pi_\rho$ lies either:
\begin{enumerate}
\item[\one] strictly in the interior of $\hex(\sigma)$, in which case $\mathfrak{e}\in \Pi_0\cap \Pi_1\cap\Pi_2$. Such edges do not touch any other edge of $\bigcup_{0\leq \rho\leq 2} \Pi_\rho$, because no two edges of the same perfect matching touch; or
\item[\two] \label{lem:perfectmatchings2} on the boundary of $\hex(\sigma)$, where $\mathfrak{e}$ is an edge in a chain linking adjacent trivalent points of $\tiling(\sigma)$. These six chains are identified pairwise by the quotient map $\pi$. Each chain comprises an odd number of edges, and for $0\leq \rho\leq 2$, the edges in each chain belong alternately to either a single perfect matching $\Pi_\rho$ or to a pair of perfect matchings $\Pi_{\rho\pm 1}$, where addition of indices is taken modulo $3$ as shown in Figure~\ref{fig:fundhexpm}.
\end{enumerate}
 Therefore edges of $\tiling(\sigma)$ are precisely those of type \two, and these edges form a connected component of $\bigcup_{0\leq \rho\leq 2} \Pi_\rho$. Every other edge in $\bigcup_{0\leq \rho\leq 2} \Pi_\rho$ is of type \one, and every such edge lies in a connected component comprising only that single edge. This completes the proof.  \end{proof}

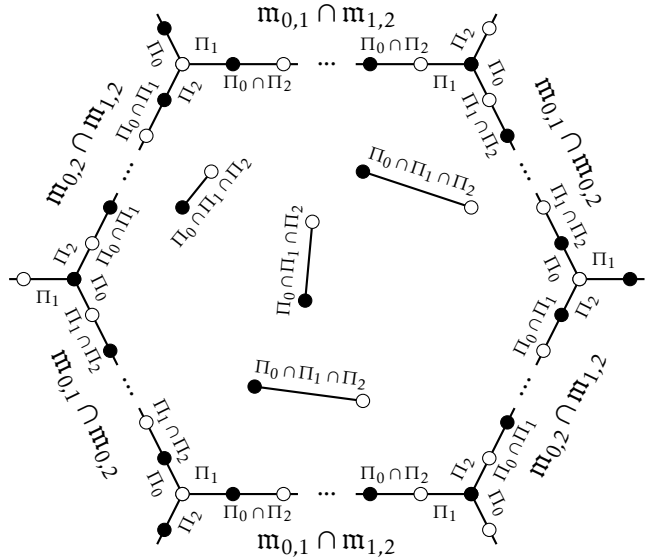
\begin{figure}[!ht]
\centering
\begin{tikzpicture} [scale=0.95, every node/.style={scale=1}]
\draw [thick] (-3.5,0) -- node[pos=0.5,sloped,below] {\tiny $\Pi_1$}(-4.2,0) -- (-4.4,0);
\draw [thick] (3.5,0) -- node[pos=0.5,sloped,above] {\tiny $\Pi_1$} (4.2,0) -- (4.4,0);
\draw [thick] (2,-3) -- node[pos=0.5,sloped,above] {\tiny $\Pi_0$}(2.25,-3.5) -- (2.35,-3.7);
\draw [thick] (-2,-3) -- node[pos=0.5,sloped,below] {\tiny $\Pi_2$}(-2.25,-3.5) --  (-2.35,-3.7);
\draw [thick] (2,3) -- node[pos=0.5,sloped,above] {\tiny $\Pi_2$}(2.25,3.5) -- (2.35,3.7);
\draw [thick] (-2,3) -- node[pos=0.5,sloped,below] {\tiny $\Pi_0$}(-2.25,3.5) --  (-2.35,3.7);

\draw (-4.2,0) node[circle,draw,fill=white,minimum size=5pt,inner sep=1pt] {{}};
\draw (4.2,0) node[circle,draw,fill=black,minimum size=5pt,inner sep=1pt] {{}};
\draw (-2.25,-3.5) node[circle,draw,fill=black,minimum size=5pt,inner sep=1pt] {{}};
\draw (-2.25,3.5) node[circle,draw,fill=black,minimum size=5pt,inner sep=1pt] {{}};
\draw (2.25,-3.5) node[circle,draw,fill=white,minimum size=5pt,inner sep=1pt] {{}};
\draw (2.25,3.5) node[circle,draw,fill=white,minimum size=5pt,inner sep=1pt] {{}};

\draw [thick] (-2,1) -- (-1.6,1.5) node[pos=0.5,sloped,below] {\tiny $\Pi_0\cap\Pi_1\cap\Pi_2$};
\draw [thick] (-1,-1.5) -- (.5,-1.7) node[pos=0.5,sloped,above] {\tiny $\Pi_0\cap\Pi_1\cap\Pi_2$};
\draw [thick] (.5,1.5) -- (2,1) node[pos=0.5,sloped,above] {\tiny $\Pi_0\cap\Pi_1\cap\Pi_2$};
\draw [thick] (-.3,-.3) -- (-.2,0.8)  node[pos=0.5,sloped,above] {\tiny $\Pi_0\cap\Pi_1\cap\Pi_2$};
\draw (-2,1) node[circle,draw,fill=black,minimum size=5pt,inner sep=1pt] {{}};
\draw (-1.6,1.5) node[circle,draw,fill=white,minimum size=5pt,inner sep=1pt] {{}};
\draw (-1,-1.5) node[circle,draw,fill=black,minimum size=5pt,inner sep=1pt] {{}};
\draw (0.5,-1.7) node[circle,draw,fill=white,minimum size=5pt,inner sep=1pt] {{}};
\draw (.5,1.5) node[circle,draw,fill=black,minimum size=5pt,inner sep=1pt] {{}};
\draw (2,1) node[circle,draw,fill=white,minimum size=5pt,inner sep=1pt] {{}};
\draw (-.3,-.3) node[circle,draw,fill=black,minimum size=5pt,inner sep=1pt] {{}};
\draw (-.2,0.8) node[circle,draw,fill=white,minimum size=5pt,inner sep=1pt] {{}};

\draw [thick] (-2,3) -- node[pos=0.5,sloped,above] {\tiny $\Pi_1$} (-1.3,3) -- node[pos=0.5,sloped,below] {\tiny $\Pi_0\cap\Pi_2$}(-0.6,3) -- (-0.4,3);
\draw [thick] (2,3) -- node[pos=0.5,sloped,below] {\tiny $\Pi_1$} (1.3,3) -- node[pos=0.5,sloped,above] {\tiny $\Pi_0\cap\Pi_2$} (0.6,3) -- (0.4,3); 
\node at (0,3) {...};
\node at (0,3.65) {$\m_{0,1}\cap\m_{1,2}$};
\draw (-2,3) node[circle,draw,fill=white,minimum size=5pt,inner sep=1pt] {{}};
\draw (-1.3,3) node[circle,draw,fill=black,minimum size=5pt,inner sep=1pt] {{}};
\draw (-0.6,3) node[circle,draw,fill=white,minimum size=5pt,inner sep=1pt] {{}};
\draw (2,3) node[circle,draw,fill=black,text=white,minimum size=5pt,inner sep=1pt] {{}};
\draw (1.3,3) node[circle,draw,fill=white,minimum size=5pt,inner sep=1pt] {{}};
\draw (0.6,3) node[circle,draw,fill=black,minimum size=5pt,inner sep=1pt] {{}};

\draw [thick] (-2,3) -- node[pos=0.5,sloped,below] {\tiny $\Pi_2$}(-2.25,2.5) -- node[pos=0.5,sloped,above] {\tiny $\Pi_0\cap\Pi_1$} (-2.5,2) --(-2.6,1.8);
\draw [thick] (-2.9,1.2) -- (-3,1) -- node[pos=0.5,sloped,below] {\tiny $\Pi_0\cap\Pi_1$} (-3.25,0.5) -- node[pos=0.5,sloped,above] {\tiny $\Pi_2$} (-3.5,0); 
\node at (-2.7,1.6) {.};
\node at (-2.75,1.5) {.};
\node at (-2.8,1.4) {.};
\node [rotate=64] at (-3.35,1.9) {$\m_{0,2}\cap\m_{1,2}$};
\draw (-2,3) node[circle,draw,fill=white,minimum size=5pt,inner sep=1pt] {{}};
\draw (-2.25,2.5) node[circle,draw,fill=black,minimum size=5pt,inner sep=1pt] {{}};
\draw (-2.5,2) node[circle,draw,fill=white,minimum size=5pt,inner sep=1pt] {{}};
\draw (-3,1) node[circle,draw,fill=black,minimum size=5pt,inner sep=1pt] {{}};
\draw (-3.25,0.5) node[circle,draw,fill=white,minimum size=5pt,inner sep=1pt] {{}};
\draw (-3.5,0) node[circle,draw,fill=black,text=white,minimum size=5pt,inner sep=1pt] {{}};

\draw [thick] (-2,-3) -- node[pos=0.5,sloped,below] {\tiny $\Pi_0$} (-2.25,-2.5) -- node[pos=0.5,sloped,above] {\tiny $\Pi_1\cap\Pi_2$} (-2.5,-2)-- (-2.6,-1.8);
\draw [thick] (-2.9,-1.2) -- (-3,-1) -- node[pos=0.5,sloped,below] {\tiny $\Pi_1\cap\Pi_2$}(-3.25, -0.5) -- node[pos=0.5,sloped,above] {\tiny $\Pi_0$}(-3.5,0); 
\node at (-2.7,-1.6) {.};
\node at (-2.75,-1.5) {.};
\node at (-2.8,-1.4) {.};
\node [rotate=-64] at (-3.35,-1.9) {$\m_{0,1}\cap\m_{0,2}$};
\draw (-2,-3) node[circle,draw,fill=white,minimum size=5pt,inner sep=1pt] {{}};
\draw (-2.25,-2.5) node[circle,draw,fill=black,minimum size=5pt,inner sep=1pt] {{}};
\draw (-2.5,-2) node[circle,draw,fill=white,minimum size=5pt,inner sep=1pt] {{}};
\draw (-3,-1) node[circle,draw,fill=black,minimum size=5pt,inner sep=1pt] {{}};
\draw (-3.25,-0.5) node[circle,draw,fill=white,minimum size=5pt,inner sep=1pt] {{}};
\draw (-3.5,0) node[circle,draw,fill=black,text=white,minimum size=5pt,inner sep=1pt] {{}};

\draw [thick] (-2,-3) -- node[pos=0.5,sloped,above] {\tiny $\Pi_1$}(-1.3, -3) -- node[pos=0.5,sloped,below] {\tiny $\Pi_0\cap\Pi_2$}(-0.6,-3) -- (-0.4,-3);
\draw [thick] (2,-3) -- node[pos=0.5,sloped,below] {\tiny $\Pi_1$}(1.3,-3) -- node[pos=0.5,sloped,above] {\tiny $\Pi_0\cap\Pi_2$}(0.6,-3) -- (0.4,-3); 
\node at (0,-3) {...};
\node at (0,-3.7) {$\m_{0,1}\cap\m_{1,2}$};
\draw (-2,-3) node[circle,draw,fill=white,minimum size=5pt,inner sep=1pt] {{}};
\draw (-1.3,-3) node[circle,draw,fill=black,minimum size=5pt,inner sep=1pt] {{}};
\draw (-0.6,-3) node[circle,draw,fill=white,minimum size=5pt,inner sep=1pt] {{}};
\draw (2,-3) node[circle,draw,fill=black,text=white,minimum size=5pt,inner sep=1pt] {{}};
\draw (1.3,-3) node[circle,draw,fill=white,minimum size=5pt,inner sep=1pt] {{}};
\draw (0.6,-3) node[circle,draw,fill=black,minimum size=5pt,inner sep=1pt] {{}};

\draw [thick] (2,-3) -- node[pos=0.5,sloped,above] {\tiny $\Pi_2$}(2.25, -2.5) -- node[pos=0.5,sloped,below] {\tiny $\Pi_0\cap\Pi_1$}(2.5, -2) -- (2.6,-1.8);
\draw [thick] (2.9,-1.2) -- (3,-1) -- node[pos=0.5,sloped,above] {\tiny $\Pi_0\cap\Pi_1$}(3.25, -0.5) -- node[pos=0.5,sloped,below] {\tiny $\Pi_2$}(3.5,0); 
\node at (2.7,-1.6) {.};
\node at (2.75,-1.5) {.};
\node at (2.8,-1.4) {.};
\node [rotate=64] at (3.35,-1.9) {$\m_{0,2}\cap\m_{1,2}$};
\draw (2,-3) node[circle,draw,fill=black,text=white,minimum size=5pt,inner sep=1pt] {{}};
\draw (2.25,-2.5) node[circle,draw,fill=white,minimum size=5pt,inner sep=1pt] {{}};
\draw (2.5,-2) node[circle,draw,fill=black,minimum size=5pt,inner sep=1pt] {{}};
\draw (3,-1) node[circle,draw,fill=white,minimum size=5pt,inner sep=1pt] {{}};
\draw (3.25,-0.5) node[circle,draw,fill=black,minimum size=5pt,inner sep=1pt] {{}};
\draw (3.5,0) node[circle,draw,fill=white,minimum size=5pt,inner sep=1pt] {{}};

\draw [thick] (2,3) -- node[pos=0.5,sloped,above] {\tiny $\Pi_0$}(2.25,2.5) -- node[pos=0.5,sloped,below] {\tiny $\Pi_1\cap\Pi_2$}(2.5, 2) -- (2.6,1.8);
\draw [thick] (2.9,1.2) -- (3,1) -- node[pos=0.5,sloped,above] {\tiny $\Pi_1\cap\Pi_2$}(3.25, 0.5) -- node[pos=0.5,sloped,below] {\tiny $\Pi_0$}(3.5,0); 
\node at (2.7,1.6) {.};
\node at (2.75,1.5) {.};
\node at (2.8,1.4) {.};
\node [rotate=-64] at (3.35,1.9) {$\m_{0,1}\cap\m_{0,2}$};
\draw (2,3) node[circle,draw,fill=black,text=white,minimum size=5pt,inner sep=1pt] {{}};
\draw (2.25,2.5) node[circle,draw,fill=white,minimum size=5pt,inner sep=1pt] {{}};
\draw (2.5,2) node[circle,draw,fill=black,minimum size=5pt,inner sep=1pt] {{}};
\draw (3,1) node[circle,draw,fill=white,minimum size=5pt,inner sep=1pt] {{}};
\draw (3.25,0.5) node[circle,draw,fill=black,minimum size=5pt,inner sep=1pt] {{}};
\draw (3.5,0) node[circle,draw,fill=white,minimum size=5pt,inner sep=1pt] {{}};

\end{tikzpicture}
\caption{The perfect matchings and meandering walks (see Section~\ref{sec:meandering}) for $\hex(\sigma)$.}
\label{fig:fundhexpm}
\end{figure}

\subsection{Meandering walks}
\label{sec:meandering}
 We now generalise the notion of a zig-zag path in $\Gamma$ by associating a walk in a consistent dimer model $\Dimer$ to any line segment $\tau\in \Sigma_\theta(2)$, where $\Sigma_\theta$ is the toric fan of $\mathcal{M}_\theta$ for any generic stability parameter $\theta\in \Theta$. In fact, the content of this section requires only that $\Gamma$ is non-degenerate \cite{IshiiUeda08}.

\begin{definition}
\label{def:meanderingwalk}
Let $\Pi_i$, $\Pi_j$ be $\theta$-stable perfect matchings of $\Dimer$, and let $\rho_i, \rho_j\in \Sigma_\theta(1)$ be the corresponding rays. The \emph{symmetric difference} of $\Pi_i$ and $\Pi_j$ is the set $\Pi_i\ominus\Pi_j:=(\Pi_i\cup\Pi_j)\setminus (\Pi_i\cap\Pi_j)$. For $\tau\in\Sigma_\theta(2)$, the \emph{meandering walk} of $\tau$ is the set $\mathfrak{m}_\tau:=\mathfrak{m}_{i,j} = \Pi_i\ominus\Pi_{j}$, where $\rho_i, \rho_j$ are the ray generators of $\tau$.  
\end{definition}

 To explain the terminology, we'll see that the edges in a meandering walk form a cycle that does not always turn maximally right at white nodes and maximally left at black nodes, but rather, it meanders. Meandering walks provide a common generalisation of the notions of a zig-zag path as in Broomhead~\cite{Broomhead12} and a $\sigma$-Strand as in Logvinenko~\cite[Definition~6.51]{Logvinenko04}.

 The proof of Proposition~\ref{prop:IUvalency}, and specifically statement \two\ of that proof, implies that if $\rho_i$ and $\rho_j$ are the endpoints of a side $\tau$ in a triangle $\sigma\in\Sigma_\theta(3)$, then the edges of the chains that form two adjacent sides of the boundary of $\hex(\sigma)$ belong alternately to the corresponding perfect matchings $\Pi_i$ and $\Pi_j$. Therefore, the edges in a meandering walk $\m_\tau$ form a cycle in $\Dimer$, and conversely, the edges of the chains that form any two adjacent sides of $\hex(\sigma)$ form the edges in a meandering walk $\mathfrak{m}_\tau$ for some two-dimensional cone $\tau\subset \sigma$. We illustrate this in Figure~\ref{fig:fundhexpm} where, 
for example, the meandering walk $\mathfrak{m}_{1,2}$ traverses edges that lie alternately in perfect matchings $\Pi_1$ and $\Pi_2$. Note that the chain of edges along any single side of the boundary of $\hex(\sigma)$ form the intersection of a pair of meandering walks.

 Traversing the edges in a meandering walk $\mathfrak{m}_\tau$ around the torus defines a homology class $[\m_\tau]\in H_1(\TT,\ZZ)$. Changing the direction in which we traverse the edges would change the sign of this class, but this sign is irrelevant for the next result. 
 
\begin{lemma}
\label{lem:mwnormal}
Let $\tau\in\Sigma_\theta(2)$. For either choice of direction on the meandering walk $\m_\tau$, the resulting homology class $[\m_\tau]\in H_1(\TT,\ZZ)$ is orthogonal to the line segment $\tau$ in the characteristic polygon $\Delta(\Dimer)\subset H^1(\mathbb{T},\mathbb{R})$ with respect to (the extension over $\RR$ of) the perfect pairing $H_1(\TT,\ZZ)\otimes H^1(\TT,\ZZ)\to \ZZ$. 
\end{lemma}
\begin{proof}
 Let $\rho_i, \rho_j\in\Sigma_\theta(1)$ be the ray generators of $\tau$, and choose $\sigma\in\Sigma_\theta(3)$ satisfying $\tau\subset \sigma$. Traversing the edges of $\m_\tau$ in $\mathbb{T}$ defines a cycle obtained as the image under $\pi\colon \RR^2\to \mathbb{T}$ of a path that follows two adjacent sides of the fundamental hexagon $\hex(\sigma)$. If we choose coordinates on $\RR^2$ such that this path travels from $(0,0)$ to $(0,1)$, then tracing the cycle $\mathfrak{m}_\tau\subset \mathbb{T}$ in this direction defines the class $(0,1)\in H_1(\TT,\ZZ)$. Figure~\ref{fig:mtaunormal} depicts part of $\tiling(\sigma)$ in the universal cover: the thicker edges of $\tiling(\sigma)$ show lifts of $\m_\tau$ via the universal cover $\pi$; the dashed lines delineate a fundamental domain for the action of $\ZZ^2$ on $\RR^2$.  
\begin{figure}[!ht]
\centering
\begin{tikzpicture} [thick,scale=0.35, every node/.style={scale=1}] 
\node at (-100pt,-30pt) [rectangle,draw=white,fill=white,minimum size=6pt,inner sep=1pt]{\scriptsize $(0,0)$};
\node at (-100pt,205pt) [rectangle,draw=white,fill=white,minimum size=6pt,inner sep=1pt]{\scriptsize $(0,1)$};
\node at (50pt,-120pt) [rectangle,draw=white,fill=white,minimum size=6pt,inner sep=1pt]{\scriptsize $(1,0)$};
\draw [thick] (0pt,0pt) -- (50pt,90pt) -- (150pt,90pt) -- (200pt,0pt) -- (150pt,-90pt) -- (50pt,-90pt) -- (0pt, 0pt);
\draw [thick] (0pt,0pt) -- (-100pt,0pt) -- (-150pt,90pt) -- (-100pt,180pt) -- (0pt,180pt) -- (50pt,90pt);
\draw[very thick, dashed] (-100pt,180pt) -- (50pt,90pt) -- (50pt,-90pt) -- (-100pt,0pt) -- (-100pt,180pt);
\draw [-stealthnew, line width = 0.07cm,shorten >=3pt] (-100pt,0pt) -- (-150pt,90pt) -- (-100pt,180pt);
\draw [-stealthnew, line width = 0.07cm,shorten >=3pt] (50pt,-90pt) -- (0pt,0pt) -- (50pt,90pt);
 \draw [line width = 0.07cm] (50pt,90pt) -- (0pt,180pt);
\draw [-stealthnew, line width = 0.07cm,shorten >=3pt] (150pt,-90pt) -- (200pt,0pt);
\draw [line width = 0.07cm] (200pt,0pt) -- (150pt,90pt);
\node at (0pt,0pt) [circle,draw,fill=black,minimum size=6pt,inner sep=1pt]{};;
\node at (150pt,90pt) [circle,draw,fill=black,minimum size=6pt,inner sep=1pt]{};
\node at (150pt,-90pt) [circle,draw,fill=black,minimum size=6pt,inner sep=1pt]{};
\node at (-150pt,90pt) [circle,draw,fill=black,minimum size=6pt,inner sep=1pt]{};;
\node at (0pt,180pt) [circle,draw,fill=black,minimum size=6pt,inner sep=1pt]{};

\node at (50pt,90pt) [circle,draw,fill=white,minimum size=6pt,inner sep=1pt]{};
\node at (200pt,0pt) [circle,draw,fill=white,minimum size=6pt,inner sep=1pt]{};
\node at (50pt,-90pt) [circle,draw,fill=white,minimum size=6pt,inner sep=1pt]{};
\node at (-100pt,0pt) [circle,draw,fill=white,minimum size=6pt,inner sep=1pt]{};
\node at (-100pt,180pt) [circle,draw,fill=white,minimum size=6pt,inner sep=1pt]{};
\end{tikzpicture}
\caption{Lifts of the meandering walk $\m_\tau$ with homology class $[\m_\tau]=(0,1)$.}
\label{fig:mtaunormal}
\end{figure}
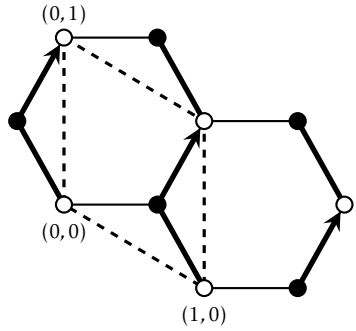

 For the $\theta$-stable perfect matchings $\Pi_i$ and $\Pi_j$ at the lattice points $\rho_i$ and $\rho_j$ respectively, we saw in Section~\ref{sec:dimers} that the height change $h(\Pi_i,\Pi_j)\in \ZZ^2$ is a vector parallel to the line segment $\tau$ in the characteristic polygon $\Delta(\Dimer)$. More invariantly, we regard the height change as a cohomology class $h(\Pi_i,\Pi_j)\in H^1(\TT,\ZZ)$ that we now compare to the homology class of the corresponding meandering walk $\mathfrak{m}_\tau$; here, the walk $\mathfrak{m}_\tau$ up two sides of the hexagon at the extreme left of Figure~\ref{fig:mtaunormal} corresponds to the walk $\mathfrak{m}_{0,2}$ in Figure~\ref{fig:fundhexpm}, so for the purposes of comparison we set $i=0$ and $j=2$.  Every edge in the perfect matchings $\Pi_0$ and $\Pi_2$ is shown in Figure~\ref{fig:fundhexpm}. Comparing this with the top-left hexagon in Figure~\ref{fig:mtaunormal}, we see that it is possible to travel from $p=(\varepsilon, \varepsilon)$ for small $\varepsilon>0$ to $p+(0,1)$ in Figure~\ref{fig:mtaunormal} without crossing a single edge in either $\Pi_0$ or $\Pi_2$, so according to \eqref{eqn:heightchange}, the $y$-coordinate of $h(\Pi_i,\Pi_j)$ equals 0. On the other hand, when travelling from $p$ to $p+(1,0)$, one is forced to cross once the walk $\mathfrak{m}_{0,2}$ whose edges lie alternately in $\Pi_0$ and $\Pi_2$ (as shown in Figure~\ref{fig:fundhexpm}), so the $x$-coordinate of $h(\Pi_i,\Pi_j)$ equals 1. Therefore $h(\Pi_i,\Pi_j) = (1,0)\in \ZZ^2=H^1(\mathbb{T},\ZZ)$, and the natural pairing gives $h(\Pi_i,\Pi_j) \cdot [\mathfrak{m}_{i,j}] = (1,0)\cdot (0,1) = 0$ as required. 
\end{proof}


\begin{remark} 
\label{rmk:zigzags}
 Ishii--Ueda~\cite[Theorem~11.1]{IshiiUeda15} implies that the homology class of each zig-zag in $\Dimer$ is orthogonal to the corresponding line segment in the boundary of $\Delta(\Dimer)$. Lemma~\ref{lem:mwnormal} provides a generalisation of this statement to any line segment in the triangulation of $\Delta(\Dimer)$ determined by the crepant resolution $\tau_\theta\colon \mathcal{M}_\theta\to X$ for any generic $\theta\in \Theta$. Lemma~\ref{lem:mwnormal} also extends the result of Logvinenko~\cite[Proposition~6.57]{Logvinenko04} for $\sigma$-Strands beyond the McKay quiver case.
\end{remark}

\subsection{Generalised Nakamura jigsaw transformations}
For any finite abelian subgroup $G\subset \SL(3,\kk)$, Nakamura~\cite{Nakamura01} introduced an algorithm to construct the $G$-Hilbert scheme where the key step in each iteration of the algorithm was a combinatorial procedure called a `$G$-igsaw transformation' from one torus-invariant $G$-cluster to another. In this section we generalise Nakamura's $G$-igsaw transformations to any generic stability condition $\theta$ and any consistent dimer model $\Gamma$ in $\TT$.

 First we recall Nakamura's $G$-igsaw transformation using the notation from Section~\ref{sec:RR}. Let $\Sigma$ be the fan of $\ghilb$. For each $\sigma_+\in \Sigma(3)$, there is a $G$-invariant monomial ideal $I_+\subset \kk[x,y,z]$ such that the torus-invariant $G$-cluster $\kk[x,y,z]/I_+$ is the fibre of the universal family on $\ghilb$ over the origin in the chart $U_{\sigma_+}$. The set $\mathcal{S}_+$ of monomials in $\kk[x,y,z]\setminus I_+$ provides an eigenbasis for $\kk[x,y,z]/I_+$, \emph{i.\,e.}\ each monomial in $\mathcal{S}_+$ lies in a different character space of the $G$-action. Following Nakamura~\cite{Nakamura01}, we obtain the \emph{$G$-graph} of $\kk[x,y,z]\setminus I_+$ by introducing directed edges between monomials in $\mathcal{S}_+$ to encode the $\kk[x,y,z]$-module structure on $\kk[x,y,z]/I_+$; in particular, each directed edge in the $G$-graph is naturally labelled by a variable $x$, $y$ or $z$. To describe the $G$-igsaw transformations of this monomial $G$-cluster, let $\tau\in \Sigma(2)$ be any interior line segment satisfying $\tau=\sigma_+\cap \sigma_-$ for some $\sigma_-\in \Sigma(3)$ and let $m=(m_1, m_2, m_3)\in M$ be the primitive vector in the normal direction to the hyperplane spanned by $\tau$ such that $\langle m,n\rangle\geq 0$ for all $n\in \sigma_+$.

 \begin{proposition}{\!\emph{(Nakamura}~\cite[Lemma~2.8]{Nakamura01}\emph{).}}
 \label{prop:Nakamura}
 Let $S_-$ be the $G$-graph of the torus-invariant $G$-cluster obtained as the fibre of the universal family on $\ghilb$ over the origin in $U_{\sigma_-}$. For each monomial $x^ay^bz^c\in\mathcal{S}_+$, the unique element of $\mathcal{S}_-$ in the same character space is  obtained by multiplying $x^ay^bz^c$ by the highest power of the $G$-invariant Laurent monomial $x^{m_1}y^{m_2}z^{m_3}$ such that the product lies in $\kk[x,y,z]$, \emph{i.\,e.}\
  \[
 \mathcal{S}_-=\Big\{ x^ay^bz^c(x^{m_1}y^{m_2}z^{m_3})^{d(a,b,c)}\in \kk[x,y,z] \mid x^ay^bz^c\in \mathcal{S}_+\Big\}
 \]
 for $d(a,b,c):=\max\{d\in \NN \mid a+dm_1\geq0, b+dm_2\geq 0, c+dm_3\geq 0\}$.
 \end{proposition}

\begin{remarks}
\leavevmode
\begin{enumerate}
    \item At least one of $m_1, m_2, m_3\in \ZZ$ is negative because $\tau$ is an interior line segment, so $d(a,b,c)$ is well-defined for each $x^ay^bz^c\in\mathcal{S}_+$.
    \item Since $G\subset \SL(3,\kk)$, the result of \cite[Lemma~2.8]{Nakamura01} is correct; compare \cite[Remark~4.13]{CMT07b}. 
    \item Worked examples of jigsaw transformations illustrating Proposition~\ref{prop:Nakamura} appear in \cite[Section~5]{Nakamura01}.
\end{enumerate}
\end{remarks}

 We now work towards the generalisation of this statement. Let $\theta\in \Theta$ be generic and write $\Sigma_\theta$ for the toric fan of $\mathcal{M}_\theta$. Let $\sigma_\pm\in\Sigma_\theta(3)$ be adjacent cones in $\Sigma_\theta$, with $\tau=\sigma_+\cap\sigma_-\in\Sigma_\theta(2)$. Let $\rho_0, \rho_1, \rho_2$ and $\rho_1, \rho_2, \rho_3$ denote the rays in $\sigma_+$ and $\sigma_-$ respectively as shown in Figure~\ref{fig:jigsawcones}, and  for $0\leq i\leq 3$ we write $\Pi_i$ for the perfect matching associated to the ray $\rho_i$ as in \eqref{eqn:pm}.
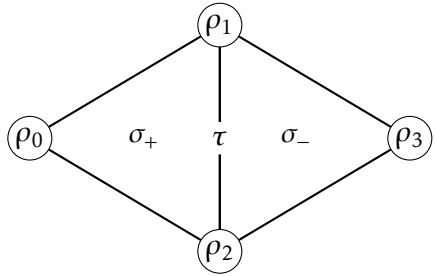
\begin{figure}[!ht]
\centering
\begin{tikzpicture} 
\draw [thick] (-2.5,0) -- (0,1.5) -- node [rectangle,draw=white,fill=white,sloped,inner sep=1pt] {{\tiny $\tau$}} (0,-1.5) -- (-2.5,0);
\draw [thick] (2.5,0) -- (0,1.5) -- node [rectangle,draw=white,fill=white,inner sep=3pt] {{$\tau$}}(0,-1.5) -- (2.5,0);

\draw (-1,0) node[circle,draw=white,fill=white,minimum size=5pt,inner sep=1pt] {{$\sigma_+$}};
\draw (1,0) node[circle,draw=white,fill=white,minimum size=5pt,inner sep=1pt] {{$\sigma_-$}};

\draw (-2.5,0) node[circle,draw,fill=white,minimum size=5pt,inner sep=1pt] {{$\rho_0$}};
\draw (2.5,0) node[circle,draw,fill=white,minimum size=5pt,inner sep=1pt] {{$\rho_3$}};
\draw (0,1.5) node[circle,draw,fill=white,minimum size=5pt,inner sep=1pt] {{$\rho_1$}};
\draw (0,-1.5) node[circle,draw,fill=white,minimum size=5pt,inner sep=1pt] {{$\rho_2$}};
\end{tikzpicture}
\caption{Cones $\sigma_+$ and $\sigma_-$ in the triangulation $\Sigma$.}
\label{fig:jigsawcones}
\end{figure}

\begin{definition} 
\label{def:jigsawpieces}
 Let $\tau\in \Sigma_\theta(2)$. A \emph{jigsaw piece} for $\tau$ is the closure of any connected component of the locus $\mathbb{T}\setminus\bigcup_{0\leq i \leq 3}\Pi_i$. 
\end{definition}

Each jigsaw piece for $\tau$ is the union of a collection of tiles of $\Dimer$. Tiles can be lifted from $\Gamma$ to $\widetilde{\Dimer}$, and we draw the jigsaw pieces in the universal cover. The terminology is chosen to suggest that one can move jigsaw pieces around in $\widetilde{\Dimer}$ by deliberately choosing different lifts.

\begin{lemma} 
\label{lem:cutHex}
Let $\sigma_\pm\in\Sigma_\theta(3)$ satisfy $\tau=\sigma_+\cap\sigma_-$ as in Figure~\ref{fig:jigsawcones} above, and regard
\begin{equation}
\label{eq:cuts-} 
\mathfrak{c}_-:=\mathfrak{m}_{1,3}\cap\mathfrak{m}_{2,3}
\end{equation}
as a subset of the edges in the boundary of $\hex(\sigma_-)$. The closure of any connected component of $\hex(\sigma_+)\setminus \mathfrak{c}_-$ is a jigsaw piece of $\tau$, and every jigsaw piece of $\tau$ arises in this way. In particular, if we cut $\hex(\sigma_+)$ along the edges of $\mathfrak{c}_-$ then we obtain precisely the jigsaw pieces of $\tau$.
\end{lemma}
\begin{proof}
The rays $\rho_1, \rho_2$ lie in $\sigma_+\cap \sigma_-$, so the fundamental hexagons $\hex(\sigma_\pm)$ share four of their six boundary sides, namely those along the lifts of $\mathfrak{m}_{1,2}$ in Figure~\ref{fig:fundhexpm}. We also know $\mathfrak{m}_{0,1}\cap \mathfrak{m}_{0,2}$ traverses the other boundary sides of $\hex(\sigma_+)$, and cutting $\hex(\sigma_+)$ along the set $\mathfrak{c}_-$ requires that we cut along the edges of $\mathfrak{m}_{1,3}\cap\mathfrak{m}_{2,3}$. All together then, the edges that cut out the connected components of $\hex(\sigma_+)\setminus \mathfrak{c}_-$ are those in the set
\begin{equation}
\label{eqn:edgesinallcuts}
\mathfrak{C}:=\mathfrak{m}_{1,2}\cup (\mathfrak{m}_{0,1}\cap \mathfrak{m}_{0,2})\cup (\mathfrak{m}_{1,3}\cap\mathfrak{m}_{2,3})
\end{equation}
 that forms the union of all edges in the boundaries of $\hex(\sigma)$ and $\hex(\sigma_-)$. Regarding $\tiling(\sigma_\pm)$ as subsets of $\mathbb{T}$, we have that
\[
\pi(\mathfrak{C}) = \pi\big(\partial \hex(\sigma_+)\big)\cup \pi\big(\partial \hex(\sigma_-)\big) = \tiling(\sigma_+)\cup\tiling(\sigma_-).
\]
 Proposition~\ref{prop:IUvalency} implies that this subset of $\mathbb{T}$ is the unique connected component of the locus $\bigcup_{0\leq i\leq 3} \Pi_i$ comprising more than a single edge. Since each jigsaw piece for $\tau$ is the closure of a union of tiles in $\Dimer$, we may ignore isolated edges of $\bigcup_{0\leq i\leq 3} \Pi_i$ when computing the jigsaw pieces. Thus, jigsaw pieces for $\tau$ are precisely the (images under $\pi$ of the) closures of the connected components of $\hex(\sigma_+)\setminus \mathfrak{c}_-$.
\end{proof}

\begin{theorem}[Generalised jigsaw transformation]
\label{thm:genjigsaw}
 Let $M_{\sigma_+}$ be the torus-invariant $\theta$-stable $A$-module associated to $\sigma_+\in \Sigma_\theta(3)$. By cutting the fundamental hexagon $\hex(\sigma_+)$ in $\RR^2$ along the edges of  $\mathfrak{c}_-$ and rearranging the resulting jigsaw pieces ({i.\,e.}\ translating each piece by a carefully chosen element of $\ZZ^2$), one obtains $\hex(\sigma_-)$, thereby determining the $\theta$-stable $A$-module $M_{\sigma_-}$ for the adjacent cone $\sigma_-\in \Sigma_\theta(3)$.
\end{theorem}
\begin{proof}
 By lifting the result of Lemma~\ref{lem:cutHex} to the universal cover, we may cut the interior of $\hex(\sigma_+)$ in $\RR^2$ along $\mathfrak{c}_-$ in order to decompose $\hex(\sigma_+)$ into the union of all jigsaw pieces for $\tau$, each a particular lift of the corresponding jigsaw piece in $\mathbb{T}^2$. By symmetry, a similar statement holds for the cut of $\hex(\sigma_-)$ along the chain of edges $\mathfrak{c}_+:=\mathfrak{m}_{0,1}\cap\mathfrak{m}_{0,2}$ in the boundary of $\hex(\sigma_+)$.  The result follows since any two lifts of a jigsaw piece differ only in translation by an element of $\ZZ^2$. 
\end{proof}

To see that Nakamura's $G$-igsaw transformations from Proposition~\ref{prop:Nakamura} can be recovered as a special case of Theorem~\ref{thm:genjigsaw}, we adopt the notation from Section~\ref{sec:RR}. In particular, for the stability parameter $\theta$ satisfying $\theta_i>0$ for each $i\neq 0$, the fine moduli space $\mathcal{M}_\theta$ is isomorphic to $\ghilb$.

\begin{corollary}
\label{cor:Gigsaw}
Let $G\subset \SL(3,\kk)$ be a finite abelian subgroup. Nakamura's $G$-igsaw transformations from Proposition~\ref{prop:Nakamura} are a special case of those from Theorem~\ref{thm:genjigsaw}.
\end{corollary}
\begin{proof}
 Let $\Sigma$ be the fan of $\ghilb$ and $\sigma_\pm\in \Sigma(3)$ adjacent cones satisfying $\tau=\sigma_+\cap \sigma_-$. The isomorphism $\mathcal{M}_\theta\cong \ghilb$ identifies the torus-invariant $\theta$-stable $A$-modules $M_{\sigma_\pm}$ with the torus-invariant $G$-clusters $\kk[x,y,z]/I_{\pm}$ obtained as the fibres of the universal family on $\ghilb$ over the origins in $U_{\sigma_\pm}$.  Under this identification, the quivers $Q^{\sigma_\pm}$ correspond to the $G$-graphs $\mathcal{S}_\pm$ of $\kk[x,y,z]/I_\pm$. Theorem~\ref{thm:genjigsaw} cuts the quiver $Q^{\sigma_+}$ (equivalently, the $A$-module $M_{\sigma_+}$) into pieces and rearranges them to produce $Q^{\sigma_-}$ (equivalently $M_{\sigma_-}$), just as Nakamura's $G$-igsaw transformation cuts $\mathcal{S}_+$  (equivalently $\kk[x,y,z]/I_+$) into pieces and rearranges them to produce $\mathcal{S}_-$ (equivalently $\kk[x,y,z]/I_-$) as in Proposition~\ref{prop:Nakamura}.
\end{proof}

\begin{remark}
\label{rem:Gigsaw}
 Each edge of the $G$-graphs $\mathcal{S}_\pm$ 
 is labelled with one of the variables $x, y, z$ according to the $\kk[x,y,z]$-module structure of $\kk[x,y,z]/I_\pm$. Contrast this with our labelling of arrows in $Q^{\sigma_\pm}$ with monomials $t^{\div(a)}$ in the Cox ring $\kk[t_\rho\mid \rho\in \Sigma(1)]$ following \eqref{eqn:xa}. To compare these labels, identify $x, y, z$ with the variables in the Cox ring indexed by the lattice points $\rho\in \Sigma(1)$ defined by corners of the junior simplex $\Delta(\Gamma)$. If we substitute $t_\rho=1$ in $t^{\div(a)}$ for each lattice point that is not a corner of the polygon $\Delta(\Gamma)$, then we obtain a single variable $x, y$ or $z$. Indeed, labelling each arrow by $t^{\div(a)}$ encodes the tautological isomorphism $\phi\colon A\to \End(T)$, whereas labelling each edge in a $G$-graph by a variable $x,y,z$ encodes the description of $A$ as the endomorphism algebra of the direct sum of the maximal Cohen--Macaulay modules on $\mathbb{C}^3/G$. The process of substituting $t_\rho=1$ for all non-corner lattice points corresponds to pushing forward each tautological line bundle on $\ghilb$ to the corresponding maximal Cohen--Macaulay module on $\mathbb{C}^3/G$.
\end{remark}

\begin{example}
\label{exa:lhjigsaw}
Continuing Examples~\ref{exa:LongHex} and ~\ref{exa:lhfundhex}, consider the cones $\sigma_+$ and $\sigma_-$ in Figure~\ref{fig:LHDivisorsTriangulation} generated by the rays $\{\rho_7,\rho_8,\rho_{9}\}$ and $\{\rho_8,\rho_9,\rho_{10}\}$ respectively.
Figure~\ref{fig:jigsaw} illustrates several copies of the fundamental hexagons $\hex(\sigma_\pm)$, where edges in $\bigcup_{7\leq i\leq 10}\Pi_i$ are coloured as follows: edges in $\mathfrak{m}_{1,2}$, $\mathfrak{c}_+$ and $\mathfrak{c}_-$ are blue, green and red respectively.
\begin{figure}[!ht]
\centering
\begin{tikzpicture} [thick,scale=0.3, every node/.style={scale=1}]
\begin{scope}\clip (0pt,0pt) rectangle (400pt,400pt);
\draw [ggrey,-stealthnew,arrowhead=6pt,shorten >=5pt] (26pt,26pt) to node [rectangle,draw,fill=white,sloped,inner sep=1pt] {{\tiny 1}} (108pt,-56pt); 
\draw [ggrey,-stealthnew,arrowhead=6pt,shorten >=5pt] (26pt,426pt) to node [rectangle,draw,fill=white,sloped,inner sep=1pt] {{\tiny 1}} (108pt,344pt);  
\draw [ggrey,-stealthnew,arrowhead=6pt, shorten >=5pt] (26pt,26pt) to node [rectangle,draw,fill=white,sloped,inner sep=1pt] {{\tiny 2}} (67pt,185pt); 
\draw [ggrey,-stealthnew,arrowhead=6pt,shorten >=5pt] (426pt,426pt) to node [rectangle,draw,fill=white,sloped,inner sep=1pt] {{\tiny 3}} (267pt,385pt); 
\draw [ggrey,-stealthnew,arrowhead=6pt,shorten >=5pt] (26pt,26pt) to node [rectangle,draw,fill=white,sloped,inner sep=1pt] {{\tiny 3}} (-133pt,-15pt); 
\draw [ggrey,-stealthnew,arrowhead=6pt,shorten >=5pt] (426pt,26pt) to node [rectangle,draw,fill=white,sloped,inner sep=1pt] {{\tiny 3}} (267pt,-15pt); 
\draw [ggrey,-stealthnew,arrowhead=6pt,shorten >=5pt] (350pt,302pt) to node [rectangle,draw,fill=white,sloped,inner sep=1pt] {{\tiny 5}} (273pt,179pt); 
\draw [ggrey,-stealthnew,arrowhead=6pt,shorten >=5pt] (308pt,144pt) to node [rectangle,draw,fill=white,sloped,inner sep=1pt] {{\tiny 9}} (273pt,179pt);  
\draw [ggrey,-stealthnew,arrowhead=6pt,shorten >=5pt] (467pt,185pt) to node [rectangle,draw,fill=white,sloped,inner sep=1pt] {{\tiny 20}} (344pt,108pt); 
\draw [ggrey,-stealthnew,arrowhead=6pt,shorten >=5pt] (67pt,185pt) to node [rectangle,draw,fill=white,sloped,inner sep=1pt] {{\tiny 20}} (-56pt,108pt);   
\draw [ggrey,-stealthnew,arrowhead=6pt,shorten >=5pt] (150pt,102pt) to node [rectangle,draw,fill=white,sloped,inner sep=1pt] {{\tiny 23}} (273pt,179pt); 
\draw [ggrey,-stealthnew,arrowhead=6pt,shorten >=5pt] (267pt,-15pt) to node [rectangle,draw,fill=white,sloped,inner sep=1pt] {{\tiny 24}} (344pt,108pt); 
\draw [ggrey,-stealthnew,arrowhead=6pt,shorten >=5pt] (267pt,385pt) to node [rectangle,draw,fill=white,sloped,inner sep=1pt] {{\tiny 24}} (344pt,508pt);  
\draw [ggrey,-stealthnew,arrowhead=6pt,shorten >=5pt] (191pt,261pt) to node [rectangle,draw,fill=white,sloped,inner sep=1pt] {{\tiny 26}} (350pt,302pt);  
\draw [ggrey,-stealthnew,arrowhead=6pt,shorten >=5pt] (191pt,261pt) to node [rectangle,draw,fill=white,sloped,inner sep=1pt] {{\tiny 28}} (150pt,102pt); 
\node at (26pt,26pt) [circle,draw=ggrey,draw,fill=white,minimum size=10pt,inner sep=1pt,text=ggrey] {\mbox{\tiny $0$}};  
\node at (350pt,302pt) [circle,draw=ggrey,draw,fill=white,minimum size=10pt,inner sep=1pt,text=ggrey] {\mbox{\tiny $1$}}; 
\node at (273pt,179pt) [circle,draw=ggrey,draw,fill=white,minimum size=10pt,inner sep=1pt,text=ggrey] {\mbox{\tiny $2$}}; 
\node at (308pt,144pt) [circle,draw=ggrey,draw,fill=white,minimum size=10pt,inner sep=1pt,text=ggrey] {\mbox{\tiny $3$}};  
\node at (108pt,344pt) [circle,draw=ggrey,draw,fill=white,minimum size=10pt,inner sep=1pt,text=ggrey] {\mbox{\tiny $4$}}; 
\node at (344pt,108pt) [circle,draw=ggrey,draw,fill=white,minimum size=10pt,inner sep=1pt,text=ggrey] {\mbox{\tiny $5$}}; 
\node at (67pt,185pt) [circle,draw=ggrey,draw,fill=white,minimum size=10pt,inner sep=1pt,text=ggrey] {\mbox{\tiny $6$}}; 
\node at (150pt,102pt) [circle,draw=ggrey,draw,fill=white,minimum size=10pt,inner sep=1pt,text=ggrey] {\mbox{\tiny $7$}}; 
\node at (267pt,385pt) [circle,draw=ggrey,draw,fill=white,minimum size=10pt,inner sep=1pt,text=ggrey] {\mbox{\tiny $8$}}; 
\node at (191pt,261pt) [circle,draw=ggrey,draw,fill=white,minimum size=10pt,inner sep=1pt,text=ggrey] {\mbox{\tiny $9$}}; 
\draw [ggreen, very thick] (309pt,343pt) -- (428pt,357pt); 
\draw [ggreen, very thick] (-91pt,343pt) -- (28pt,357pt); 
\draw [dbblue, very thick] (177pt,75pt) -- (230pt,89pt); 
\draw [ggreen, very thick] (377pt,275pt) -- (363pt,222pt); 
\draw [very thick] (205pt,181pt) -- (271pt,247pt); 
\draw [bluee, very thick] (253pt,65pt) -- (387pt,199pt); 
\draw [ggreen, very thick] (377pt,275pt) -- (428pt,357pt); 
\draw [ggreen, very thick] (-23pt,275pt) -- (28pt,357pt); 
\draw [ggreen, very thick] (387pt,199pt) -- (363pt,222pt); 
\draw [dbblue, very thick] (253pt,65pt) -- (230pt,89pt); 
\draw [bluee, very thick] (40pt,212pt) -- (122pt,263pt); 
\draw [dbblue, very thick] (95pt,24pt) -- (177pt,75pt); 
\draw [bluee, very thick] (189pt,330pt) -- (240pt,412pt); 
\draw [bluee, very thick] (189pt,-70pt) -- (240pt,12pt); 
\draw [very thick] (348pt,40pt) -- (412pt,106pt); 
\draw [very thick] (-52pt,40pt) -- (12pt,106pt); 
\draw [bluee, very thick] (253pt,65pt) -- (240pt,12pt); 
\draw [bluee, very thick] (387pt,199pt) -- (440pt,212pt); 
\draw [bluee, very thick] (-13pt,199pt) -- (40pt,212pt); 
\draw [dbblue, very thick] (109pt,143pt) -- (122pt,263pt); 
\draw [dbblue, very thick] (109pt,143pt) -- (95pt,24pt); 
\draw [ggreen, very thick] (189pt,330pt) -- (309pt,343pt); 
\draw [bluee, very thick] (189pt,330pt) -- (122pt,263pt); 
\node at (309pt,343pt) [circle,draw=ggreen,fill=ggreen,minimum size=6pt,inner sep=1pt]{}; 
\node at (28pt,357pt) [circle,draw=ggreen,fill=white,minimum size=6pt,inner sep=1pt]{}; 
\node at (95pt,24pt) [circle,draw=dbblue,fill=dbblue,minimum size=6pt,inner sep=1pt]{}; 
\node at (348pt,40pt) [circle,draw,fill=white,minimum size=6pt,inner sep=1pt]{}; 
\node at (12pt,106pt) [circle,draw,fill=black,minimum size=6pt,inner sep=1pt]{}; 
\node at (109pt,143pt) [circle,draw=dbblue,fill=white,minimum size=6pt,inner sep=1pt]{}; 
\node at (377pt,275pt) [circle,draw=ggreen,fill=ggreen,minimum size=6pt,inner sep=1pt]{}; 
\node at (271pt,247pt) [circle,draw,fill=white,minimum size=6pt,inner sep=1pt]{}; 
\node at (205pt,181pt) [circle,draw,fill=black,minimum size=6pt,inner sep=1pt]{}; 
\node at (230pt,89pt) [circle,draw=dbblue,fill=dbblue,minimum size=6pt,inner sep=1pt]{}; 
\node at (363pt,222pt) [circle,draw=ggreen,fill=white,minimum size=6pt,inner sep=1pt]{}; 
\node at (177pt,75pt) [circle,draw=dbblue,fill=white,minimum size=6pt,inner sep=1pt]{}; 
\node at (253pt,65pt) [circle,draw=bluee,fill=white,minimum size=6pt,inner sep=1pt]{}; 
\node at (387pt,199pt) [circle,draw=bluee,fill=bluee,minimum size=6pt,inner sep=1pt]{}; 
\node at (40pt,212pt) [circle,draw=bluee,fill=white,minimum size=6pt,inner sep=1pt]{}; 
\node at (240pt,12pt) [circle,draw=bluee,fill=bluee,minimum size=6pt,inner sep=1pt]{}; 
\node at (122pt,263pt) [circle,draw=bluee,fill=bluee,minimum size=6pt,inner sep=1pt]{}; 
\node at (189pt,330pt) [circle,draw=bluee,fill=white,minimum size=6pt,inner sep=1pt]{}; 
\draw[very thick, dashed] (0pt,0pt) rectangle (400pt,400pt);
\end{scope}\end{tikzpicture}
\hspace{-.225cm}
\begin{tikzpicture} [thick,scale=0.3, every node/.style={scale=1}]
\begin{scope}\clip (0pt,0pt) rectangle (400pt,400pt);
\draw [ggrey,-stealthnew,arrowhead=6pt,shorten >=5pt] (26pt,26pt) to node [rectangle,draw,fill=white,sloped,inner sep=1pt] {{\tiny 1}} (108pt,-56pt); 
\draw [ggrey,-stealthnew,arrowhead=6pt,shorten >=5pt] (26pt,426pt) to node [rectangle,draw,fill=white,sloped,inner sep=1pt] {{\tiny 1}} (108pt,344pt);  
\draw [ggrey,-stealthnew,arrowhead=6pt, shorten >=5pt] (26pt,26pt) to node [rectangle,draw,fill=white,sloped,inner sep=1pt] {{\tiny 2}} (67pt,185pt); 
\draw [ggrey,-stealthnew,arrowhead=6pt,shorten >=5pt] (426pt,426pt) to node [rectangle,draw,fill=white,sloped,inner sep=1pt] {{\tiny 3}} (267pt,385pt); 
\draw [ggrey,-stealthnew,arrowhead=6pt,shorten >=5pt] (26pt,26pt) to node [rectangle,draw,fill=white,sloped,inner sep=1pt] {{\tiny 3}} (-133pt,-15pt); 
\draw [ggrey,-stealthnew,arrowhead=6pt,shorten >=5pt] (426pt,26pt) to node [rectangle,draw,fill=white,sloped,inner sep=1pt] {{\tiny 3}} (267pt,-15pt); 
\draw [ggrey,-stealthnew,arrowhead=6pt,shorten >=5pt] (350pt,302pt) to node [rectangle,draw,fill=white,sloped,inner sep=1pt] {{\tiny 5}} (273pt,179pt); 
\draw [ggrey,-stealthnew,arrowhead=6pt,shorten >=5pt] (308pt,144pt) to node [rectangle,draw,fill=white,sloped,inner sep=1pt] {{\tiny 9}} (273pt,179pt);  
\draw [ggrey,-stealthnew,arrowhead=6pt,shorten >=5pt] (467pt,185pt) to node [rectangle,draw,fill=white,sloped,inner sep=1pt] {{\tiny 20}} (344pt,108pt); 
\draw [ggrey,-stealthnew,arrowhead=6pt,shorten >=5pt] (67pt,185pt) to node [rectangle,draw,fill=white,sloped,inner sep=1pt] {{\tiny 20}} (-56pt,108pt);   
\draw [ggrey,-stealthnew,arrowhead=6pt,shorten >=5pt] (150pt,102pt) to node [rectangle,draw,fill=white,sloped,inner sep=1pt] {{\tiny 23}} (273pt,179pt); 
\draw [ggrey,-stealthnew,arrowhead=6pt,shorten >=5pt] (267pt,-15pt) to node [rectangle,draw,fill=white,sloped,inner sep=1pt] {{\tiny 24}} (344pt,108pt); 
\draw [ggrey,-stealthnew,arrowhead=6pt,shorten >=5pt] (267pt,385pt) to node [rectangle,draw,fill=white,sloped,inner sep=1pt] {{\tiny 24}} (344pt,508pt);  
\draw [ggrey,-stealthnew,arrowhead=6pt,shorten >=5pt] (191pt,261pt) to node [rectangle,draw,fill=white,sloped,inner sep=1pt] {{\tiny 26}} (350pt,302pt);  
\draw [ggrey,-stealthnew,arrowhead=6pt,shorten >=5pt] (191pt,261pt) to node [rectangle,draw,fill=white,sloped,inner sep=1pt] {{\tiny 28}} (150pt,102pt); 
\node at (26pt,26pt) [circle,draw=ggrey,draw,fill=white,minimum size=10pt,inner sep=1pt,text=ggrey] {\mbox{\tiny $0$}};  
\node at (350pt,302pt) [circle,draw=ggrey,draw,fill=white,minimum size=10pt,inner sep=1pt,text=ggrey] {\mbox{\tiny $1$}}; 
\node at (273pt,179pt) [circle,draw=ggrey,draw,fill=white,minimum size=10pt,inner sep=1pt,text=ggrey] {\mbox{\tiny $2$}}; 
\node at (308pt,144pt) [circle,draw=ggrey,draw,fill=white,minimum size=10pt,inner sep=1pt,text=ggrey] {\mbox{\tiny $3$}};  
\node at (108pt,344pt) [circle,draw=ggrey,draw,fill=white,minimum size=10pt,inner sep=1pt,text=ggrey] {\mbox{\tiny $4$}}; 
\node at (344pt,108pt) [circle,draw=ggrey,draw,fill=white,minimum size=10pt,inner sep=1pt,text=ggrey] {\mbox{\tiny $5$}}; 
\node at (67pt,185pt) [circle,draw=ggrey,draw,fill=white,minimum size=10pt,inner sep=1pt,text=ggrey] {\mbox{\tiny $6$}}; 
\node at (150pt,102pt) [circle,draw=ggrey,draw,fill=white,minimum size=10pt,inner sep=1pt,text=ggrey] {\mbox{\tiny $7$}}; 
\node at (267pt,385pt) [circle,draw=ggrey,draw,fill=white,minimum size=10pt,inner sep=1pt,text=ggrey] {\mbox{\tiny $8$}}; 
\node at (191pt,261pt) [circle,draw=ggrey,draw,fill=white,minimum size=10pt,inner sep=1pt,text=ggrey] {\mbox{\tiny $9$}}; 
\draw [ggreen, very thick] (309pt,343pt) -- (428pt,357pt); 
\draw [ggreen, very thick] (-91pt,343pt) -- (28pt,357pt); 
\draw [dbblue, very thick] (177pt,75pt) -- (230pt,89pt); 
\draw [ggreen, very thick] (377pt,275pt) -- (363pt,222pt); 
\draw [very thick] (205pt,181pt) -- (271pt,247pt); 
\draw [bluee, very thick] (253pt,65pt) -- (387pt,199pt); 
\draw [ggreen, very thick] (377pt,275pt) -- (428pt,357pt); 
\draw [ggreen, very thick] (-23pt,275pt) -- (28pt,357pt); 
\draw [ggreen, very thick] (387pt,199pt) -- (363pt,222pt); 
\draw [dbblue, very thick] (253pt,65pt) -- (230pt,89pt); 
\draw [bluee, very thick] (40pt,212pt) -- (122pt,263pt); 
\draw [dbblue, very thick] (95pt,24pt) -- (177pt,75pt); 
\draw [bluee, very thick] (189pt,330pt) -- (240pt,412pt); 
\draw [bluee, very thick] (189pt,-70pt) -- (240pt,12pt); 
\draw [very thick] (348pt,40pt) -- (412pt,106pt); 
\draw [very thick] (-52pt,40pt) -- (12pt,106pt); 
\draw [bluee, very thick] (253pt,65pt) -- (240pt,12pt); 
\draw [bluee, very thick] (387pt,199pt) -- (440pt,212pt); 
\draw [bluee, very thick] (-13pt,199pt) -- (40pt,212pt); 
\draw [dbblue, very thick] (109pt,143pt) -- (122pt,263pt); 
\draw [dbblue, very thick] (109pt,143pt) -- (95pt,24pt); 
\draw [ggreen, very thick] (189pt,330pt) -- (309pt,343pt); 
\draw [bluee, very thick] (189pt,330pt) -- (122pt,263pt); 
\node at (309pt,343pt) [circle,draw=ggreen,fill=ggreen,minimum size=6pt,inner sep=1pt]{}; 
\node at (28pt,357pt) [circle,draw=ggreen,fill=white,minimum size=6pt,inner sep=1pt]{}; 
\node at (95pt,24pt) [circle,draw=dbblue,fill=dbblue,minimum size=6pt,inner sep=1pt]{}; 
\node at (348pt,40pt) [circle,draw,fill=white,minimum size=6pt,inner sep=1pt]{}; 
\node at (12pt,106pt) [circle,draw,fill=black,minimum size=6pt,inner sep=1pt]{}; 
\node at (109pt,143pt) [circle,draw=dbblue,fill=white,minimum size=6pt,inner sep=1pt]{}; 
\node at (377pt,275pt) [circle,draw=ggreen,fill=ggreen,minimum size=6pt,inner sep=1pt]{}; 
\node at (271pt,247pt) [circle,draw,fill=white,minimum size=6pt,inner sep=1pt]{}; 
\node at (205pt,181pt) [circle,draw,fill=black,minimum size=6pt,inner sep=1pt]{}; 
\node at (230pt,89pt) [circle,draw=dbblue,fill=dbblue,minimum size=6pt,inner sep=1pt]{}; 
\node at (363pt,222pt) [circle,draw=ggreen,fill=white,minimum size=6pt,inner sep=1pt]{}; 
\node at (177pt,75pt) [circle,draw=dbblue,fill=white,minimum size=6pt,inner sep=1pt]{}; 
\node at (253pt,65pt) [circle,draw=bluee,fill=white,minimum size=6pt,inner sep=1pt]{}; 
\node at (387pt,199pt) [circle,draw=bluee,fill=bluee,minimum size=6pt,inner sep=1pt]{}; 
\node at (40pt,212pt) [circle,draw=bluee,fill=white,minimum size=6pt,inner sep=1pt]{}; 
\node at (240pt,12pt) [circle,draw=bluee,fill=bluee,minimum size=6pt,inner sep=1pt]{}; 
\node at (122pt,263pt) [circle,draw=bluee,fill=bluee,minimum size=6pt,inner sep=1pt]{}; 
\node at (189pt,330pt) [circle,draw=bluee,fill=white,minimum size=6pt,inner sep=1pt]{}; 

\draw[very thick, dashed] (0pt,0pt) rectangle (400pt,400pt);
\end{scope}\end{tikzpicture}\\
\begin{tikzpicture} [thick,scale=0.3, every node/.style={scale=1}]
\begin{scope}\clip (0pt,0pt) rectangle (400pt,400pt);
\draw [ggrey,-stealthnew,arrowhead=6pt,shorten >=5pt] (26pt,26pt) to node [rectangle,draw,fill=white,sloped,inner sep=1pt] {{\tiny 1}} (108pt,-56pt); 
\draw [ggrey,-stealthnew,arrowhead=6pt,shorten >=5pt] (26pt,426pt) to node [rectangle,draw,fill=white,sloped,inner sep=1pt] {{\tiny 1}} (108pt,344pt);  
\draw [ggrey,-stealthnew,arrowhead=6pt, shorten >=5pt] (26pt,26pt) to node [rectangle,draw,fill=white,sloped,inner sep=1pt] {{\tiny 2}} (67pt,185pt); 
\draw [ggrey,-stealthnew,arrowhead=6pt,shorten >=5pt] (426pt,426pt) to node [rectangle,draw,fill=white,sloped,inner sep=1pt] {{\tiny 3}} (267pt,385pt); 
\draw [ggrey,-stealthnew,arrowhead=6pt,shorten >=5pt] (26pt,26pt) to node [rectangle,draw,fill=white,sloped,inner sep=1pt] {{\tiny 3}} (-133pt,-15pt); 
\draw [ggrey,-stealthnew,arrowhead=6pt,shorten >=5pt] (426pt,26pt) to node [rectangle,draw,fill=white,sloped,inner sep=1pt] {{\tiny 3}} (267pt,-15pt); 
\draw [ggrey,-stealthnew,arrowhead=6pt,shorten >=5pt] (350pt,302pt) to node [rectangle,draw,fill=white,sloped,inner sep=1pt] {{\tiny 5}} (273pt,179pt); 
\draw [ggrey,-stealthnew,arrowhead=6pt,shorten >=5pt] (308pt,144pt) to node [rectangle,draw,fill=white,sloped,inner sep=1pt] {{\tiny 9}} (273pt,179pt);  
\draw [ggrey,-stealthnew,arrowhead=6pt,shorten >=5pt] (467pt,185pt) to node [rectangle,draw,fill=white,sloped,inner sep=1pt] {{\tiny 20}} (344pt,108pt); 
\draw [ggrey,-stealthnew,arrowhead=6pt,shorten >=5pt] (67pt,185pt) to node [rectangle,draw,fill=white,sloped,inner sep=1pt] {{\tiny 20}} (-56pt,108pt);   
\draw [ggrey,-stealthnew,arrowhead=6pt,shorten >=5pt] (150pt,102pt) to node [rectangle,draw,fill=white,sloped,inner sep=1pt] {{\tiny 23}} (273pt,179pt); 
\draw [ggrey,-stealthnew,arrowhead=6pt,shorten >=5pt] (267pt,-15pt) to node [rectangle,draw,fill=white,sloped,inner sep=1pt] {{\tiny 24}} (344pt,108pt); 
\draw [ggrey,-stealthnew,arrowhead=6pt,shorten >=5pt] (267pt,385pt) to node [rectangle,draw,fill=white,sloped,inner sep=1pt] {{\tiny 24}} (344pt,508pt);  
\draw [ggrey,-stealthnew,arrowhead=6pt,shorten >=5pt] (191pt,261pt) to node [rectangle,draw,fill=white,sloped,inner sep=1pt] {{\tiny 26}} (350pt,302pt);  
\draw [ggrey,-stealthnew,arrowhead=6pt,shorten >=5pt] (191pt,261pt) to node [rectangle,draw,fill=white,sloped,inner sep=1pt] {{\tiny 28}} (150pt,102pt); 
\node at (26pt,26pt) [circle,draw=ggrey,draw,fill=white,minimum size=10pt,inner sep=1pt,text=ggrey] {\mbox{\tiny $0$}};  
\node at (350pt,302pt) [circle,draw=ggrey,draw,fill=white,minimum size=10pt,inner sep=1pt,text=ggrey] {\mbox{\tiny $1$}}; 
\node at (273pt,179pt) [circle,draw=ggrey,draw,fill=white,minimum size=10pt,inner sep=1pt,text=ggrey] {\mbox{\tiny $2$}}; 
\node at (308pt,144pt) [circle,draw=ggrey,draw,fill=white,minimum size=10pt,inner sep=1pt,text=ggrey] {\mbox{\tiny $3$}};  
\node at (108pt,344pt) [circle,draw=ggrey,draw,fill=white,minimum size=10pt,inner sep=1pt,text=ggrey] {\mbox{\tiny $4$}}; 
\node at (344pt,108pt) [circle,draw=ggrey,draw,fill=white,minimum size=10pt,inner sep=1pt,text=ggrey] {\mbox{\tiny $5$}}; 
\node at (67pt,185pt) [circle,draw=ggrey,draw,fill=white,minimum size=10pt,inner sep=1pt,text=ggrey] {\mbox{\tiny $6$}}; 
\node at (150pt,102pt) [circle,draw=ggrey,draw,fill=white,minimum size=10pt,inner sep=1pt,text=ggrey] {\mbox{\tiny $7$}}; 
\node at (267pt,385pt) [circle,draw=ggrey,draw,fill=white,minimum size=10pt,inner sep=1pt,text=ggrey] {\mbox{\tiny $8$}}; 
\node at (191pt,261pt) [circle,draw=ggrey,draw,fill=white,minimum size=10pt,inner sep=1pt,text=ggrey] {\mbox{\tiny $9$}}; 
\draw [ggreen, very thick] (309pt,343pt) -- (428pt,357pt); 
\draw [ggreen, very thick] (-91pt,343pt) -- (28pt,357pt); 
\draw [dbblue, very thick] (177pt,75pt) -- (230pt,89pt); 
\draw [ggreen, very thick] (377pt,275pt) -- (363pt,222pt); 
\draw [very thick] (205pt,181pt) -- (271pt,247pt); 
\draw [bluee, very thick] (253pt,65pt) -- (387pt,199pt); 
\draw [ggreen, very thick] (377pt,275pt) -- (428pt,357pt); 
\draw [ggreen, very thick] (-23pt,275pt) -- (28pt,357pt); 
\draw [ggreen, very thick] (387pt,199pt) -- (363pt,222pt); 
\draw [dbblue, very thick] (253pt,65pt) -- (230pt,89pt); 
\draw [bluee, very thick] (40pt,212pt) -- (122pt,263pt); 
\draw [dbblue, very thick] (95pt,24pt) -- (177pt,75pt); 
\draw [bluee, very thick] (189pt,330pt) -- (240pt,412pt); 
\draw [bluee, very thick] (189pt,-70pt) -- (240pt,12pt); 
\draw [very thick] (348pt,40pt) -- (412pt,106pt); 
\draw [very thick] (-52pt,40pt) -- (12pt,106pt); 
\draw [bluee, very thick] (253pt,65pt) -- (240pt,12pt); 
\draw [bluee, very thick] (387pt,199pt) -- (440pt,212pt); 
\draw [bluee, very thick] (-13pt,199pt) -- (40pt,212pt); 
\draw [dbblue, very thick] (109pt,143pt) -- (122pt,263pt); 
\draw [dbblue, very thick] (109pt,143pt) -- (95pt,24pt); 
\draw [ggreen, very thick] (189pt,330pt) -- (309pt,343pt); 
\draw [bluee, very thick] (189pt,330pt) -- (122pt,263pt); 
\node at (309pt,343pt) [circle,draw=ggreen,fill=ggreen,minimum size=6pt,inner sep=1pt]{}; 
\node at (28pt,357pt) [circle,draw=ggreen,fill=white,minimum size=6pt,inner sep=1pt]{}; 
\node at (95pt,24pt) [circle,draw=dbblue,fill=dbblue,minimum size=6pt,inner sep=1pt]{}; 
\node at (348pt,40pt) [circle,draw,fill=white,minimum size=6pt,inner sep=1pt]{}; 
\node at (12pt,106pt) [circle,draw,fill=black,minimum size=6pt,inner sep=1pt]{}; 
\node at (109pt,143pt) [circle,draw=dbblue,fill=white,minimum size=6pt,inner sep=1pt]{}; 
\node at (377pt,275pt) [circle,draw=ggreen,fill=ggreen,minimum size=6pt,inner sep=1pt]{}; 
\node at (271pt,247pt) [circle,draw,fill=white,minimum size=6pt,inner sep=1pt]{}; 
\node at (205pt,181pt) [circle,draw,fill=black,minimum size=6pt,inner sep=1pt]{}; 
\node at (230pt,89pt) [circle,draw=dbblue,fill=dbblue,minimum size=6pt,inner sep=1pt]{}; 
\node at (363pt,222pt) [circle,draw=ggreen,fill=white,minimum size=6pt,inner sep=1pt]{}; 
\node at (177pt,75pt) [circle,draw=dbblue,fill=white,minimum size=6pt,inner sep=1pt]{}; 
\node at (253pt,65pt) [circle,draw=bluee,fill=white,minimum size=6pt,inner sep=1pt]{}; 
\node at (387pt,199pt) [circle,draw=bluee,fill=bluee,minimum size=6pt,inner sep=1pt]{}; 
\node at (40pt,212pt) [circle,draw=bluee,fill=white,minimum size=6pt,inner sep=1pt]{}; 
\node at (240pt,12pt) [circle,draw=bluee,fill=bluee,minimum size=6pt,inner sep=1pt]{}; 
\node at (122pt,263pt) [circle,draw=bluee,fill=bluee,minimum size=6pt,inner sep=1pt]{}; 
\node at (189pt,330pt) [circle,draw=bluee,fill=white,minimum size=6pt,inner sep=1pt]{}; 
\draw[very thick, dashed] (0pt,0pt) rectangle (400pt,400pt);
\end{scope}\end{tikzpicture}
\hspace{-.225cm}
\begin{tikzpicture} [thick,scale=0.3, every node/.style={scale=1}]
\begin{scope}\clip (0pt,0pt) rectangle (400pt,400pt);
\draw [ggrey,-stealthnew,arrowhead=6pt,shorten >=5pt] (26pt,26pt) to node [rectangle,draw,fill=white,sloped,inner sep=1pt] {{\tiny 1}} (108pt,-56pt); 
\draw [ggrey,-stealthnew,arrowhead=6pt,shorten >=5pt] (26pt,426pt) to node [rectangle,draw,fill=white,sloped,inner sep=1pt] {{\tiny 1}} (108pt,344pt);  
\draw [ggrey,-stealthnew,arrowhead=6pt, shorten >=5pt] (26pt,26pt) to node [rectangle,draw,fill=white,sloped,inner sep=1pt] {{\tiny 2}} (67pt,185pt); 
\draw [ggrey,-stealthnew,arrowhead=6pt,shorten >=5pt] (426pt,426pt) to node [rectangle,draw,fill=white,sloped,inner sep=1pt] {{\tiny 3}} (267pt,385pt); 
\draw [ggrey,-stealthnew,arrowhead=6pt,shorten >=5pt] (26pt,26pt) to node [rectangle,draw,fill=white,sloped,inner sep=1pt] {{\tiny 3}} (-133pt,-15pt); 
\draw [ggrey,-stealthnew,arrowhead=6pt,shorten >=5pt] (426pt,26pt) to node [rectangle,draw,fill=white,sloped,inner sep=1pt] {{\tiny 3}} (267pt,-15pt); 
\draw [ggrey,-stealthnew,arrowhead=6pt,shorten >=5pt] (350pt,302pt) to node [rectangle,draw,fill=white,sloped,inner sep=1pt] {{\tiny 5}} (273pt,179pt); 
\draw [ggrey,-stealthnew,arrowhead=6pt,shorten >=5pt] (308pt,144pt) to node [rectangle,draw,fill=white,sloped,inner sep=1pt] {{\tiny 9}} (273pt,179pt);  
\draw [ggrey,-stealthnew,arrowhead=6pt,shorten >=5pt] (467pt,185pt) to node [rectangle,draw,fill=white,sloped,inner sep=1pt] {{\tiny 20}} (344pt,108pt); 
\draw [ggrey,-stealthnew,arrowhead=6pt,shorten >=5pt] (67pt,185pt) to node [rectangle,draw,fill=white,sloped,inner sep=1pt] {{\tiny 20}} (-56pt,108pt);   
\draw [ggrey,-stealthnew,arrowhead=6pt,shorten >=5pt] (150pt,102pt) to node [rectangle,draw,fill=white,sloped,inner sep=1pt] {{\tiny 23}} (273pt,179pt); 
\draw [ggrey,-stealthnew,arrowhead=6pt,shorten >=5pt] (267pt,-15pt) to node [rectangle,draw,fill=white,sloped,inner sep=1pt] {{\tiny 24}} (344pt,108pt); 
\draw [ggrey,-stealthnew,arrowhead=6pt,shorten >=5pt] (267pt,385pt) to node [rectangle,draw,fill=white,sloped,inner sep=1pt] {{\tiny 24}} (344pt,508pt);  
\draw [ggrey,-stealthnew,arrowhead=6pt,shorten >=5pt] (191pt,261pt) to node [rectangle,draw,fill=white,sloped,inner sep=1pt] {{\tiny 26}} (350pt,302pt);  
\draw [ggrey,-stealthnew,arrowhead=6pt,shorten >=5pt] (191pt,261pt) to node [rectangle,draw,fill=white,sloped,inner sep=1pt] {{\tiny 28}} (150pt,102pt); 
\node at (26pt,26pt) [circle,draw=ggrey,draw,fill=white,minimum size=10pt,inner sep=1pt,text=ggrey] {\mbox{\tiny $0$}};  
\node at (350pt,302pt) [circle,draw=ggrey,draw,fill=white,minimum size=10pt,inner sep=1pt,text=ggrey] {\mbox{\tiny $1$}}; 
\node at (273pt,179pt) [circle,draw=ggrey,draw,fill=white,minimum size=10pt,inner sep=1pt,text=ggrey] {\mbox{\tiny $2$}}; 
\node at (308pt,144pt) [circle,draw=ggrey,draw,fill=white,minimum size=10pt,inner sep=1pt,text=ggrey] {\mbox{\tiny $3$}};  
\node at (108pt,344pt) [circle,draw=ggrey,draw,fill=white,minimum size=10pt,inner sep=1pt,text=ggrey] {\mbox{\tiny $4$}}; 
\node at (344pt,108pt) [circle,draw=ggrey,draw,fill=white,minimum size=10pt,inner sep=1pt,text=ggrey] {\mbox{\tiny $5$}}; 
\node at (67pt,185pt) [circle,draw=ggrey,draw,fill=white,minimum size=10pt,inner sep=1pt,text=ggrey] {\mbox{\tiny $6$}}; 
\node at (150pt,102pt) [circle,draw=ggrey,draw,fill=white,minimum size=10pt,inner sep=1pt,text=ggrey] {\mbox{\tiny $7$}}; 
\node at (267pt,385pt) [circle,draw=ggrey,draw,fill=white,minimum size=10pt,inner sep=1pt,text=ggrey] {\mbox{\tiny $8$}}; 
\node at (191pt,261pt) [circle,draw=ggrey,draw,fill=white,minimum size=10pt,inner sep=1pt,text=ggrey] {\mbox{\tiny $9$}}; 
\draw [ggreen, very thick] (309pt,343pt) -- (428pt,357pt); 
\draw [ggreen, very thick] (-91pt,343pt) -- (28pt,357pt); 
\draw [dbblue, very thick] (177pt,75pt) -- (230pt,89pt); 
\draw [ggreen, very thick] (377pt,275pt) -- (363pt,222pt); 
\draw [very thick] (205pt,181pt) -- (271pt,247pt); 
\draw [bluee, very thick] (253pt,65pt) -- (387pt,199pt); 
\draw [ggreen, very thick] (377pt,275pt) -- (428pt,357pt); 
\draw [ggreen, very thick] (-23pt,275pt) -- (28pt,357pt); 
\draw [ggreen, very thick] (387pt,199pt) -- (363pt,222pt); 
\draw [dbblue, very thick] (253pt,65pt) -- (230pt,89pt); 
\draw [bluee, very thick] (40pt,212pt) -- (122pt,263pt); 
\draw [dbblue, very thick] (95pt,24pt) -- (177pt,75pt); 
\draw [bluee, very thick] (189pt,330pt) -- (240pt,412pt); 
\draw [bluee, very thick] (189pt,-70pt) -- (240pt,12pt); 
\draw [very thick] (348pt,40pt) -- (412pt,106pt); 
\draw [very thick] (-52pt,40pt) -- (12pt,106pt); 
\draw [bluee, very thick] (253pt,65pt) -- (240pt,12pt); 
\draw [bluee, very thick] (387pt,199pt) -- (440pt,212pt); 
\draw [bluee, very thick] (-13pt,199pt) -- (40pt,212pt); 
\draw [dbblue, very thick] (109pt,143pt) -- (122pt,263pt); 
\draw [dbblue, very thick] (109pt,143pt) -- (95pt,24pt); 
\draw [ggreen, very thick] (189pt,330pt) -- (309pt,343pt); 
\draw [bluee, very thick] (189pt,330pt) -- (122pt,263pt); 
\node at (309pt,343pt) [circle,draw=ggreen,fill=ggreen,minimum size=6pt,inner sep=1pt]{}; 
\node at (28pt,357pt) [circle,draw=ggreen,fill=white,minimum size=6pt,inner sep=1pt]{}; 
\node at (95pt,24pt) [circle,draw=dbblue,fill=dbblue,minimum size=6pt,inner sep=1pt]{}; 
\node at (348pt,40pt) [circle,draw,fill=white,minimum size=6pt,inner sep=1pt]{}; 
\node at (12pt,106pt) [circle,draw,fill=black,minimum size=6pt,inner sep=1pt]{}; 
\node at (109pt,143pt) [circle,draw=dbblue,fill=white,minimum size=6pt,inner sep=1pt]{}; 
\node at (377pt,275pt) [circle,draw=ggreen,fill=ggreen,minimum size=6pt,inner sep=1pt]{}; 
\node at (271pt,247pt) [circle,draw,fill=white,minimum size=6pt,inner sep=1pt]{}; 
\node at (205pt,181pt) [circle,draw,fill=black,minimum size=6pt,inner sep=1pt]{}; 
\node at (230pt,89pt) [circle,draw=dbblue,fill=dbblue,minimum size=6pt,inner sep=1pt]{}; 
\node at (363pt,222pt) [circle,draw=ggreen,fill=white,minimum size=6pt,inner sep=1pt]{}; 
\node at (177pt,75pt) [circle,draw=dbblue,fill=white,minimum size=6pt,inner sep=1pt]{}; 
\node at (253pt,65pt) [circle,draw=bluee,fill=white,minimum size=6pt,inner sep=1pt]{}; 
\node at (387pt,199pt) [circle,draw=bluee,fill=bluee,minimum size=6pt,inner sep=1pt]{}; 
\node at (40pt,212pt) [circle,draw=bluee,fill=white,minimum size=6pt,inner sep=1pt]{}; 
\node at (240pt,12pt) [circle,draw=bluee,fill=bluee,minimum size=6pt,inner sep=1pt]{}; 
\node at (122pt,263pt) [circle,draw=bluee,fill=bluee,minimum size=6pt,inner sep=1pt]{}; 
\node at (189pt,330pt) [circle,draw=bluee,fill=white,minimum size=6pt,inner sep=1pt]{}; 
\draw[very thick, dashed] (0pt,0pt) rectangle (400pt,400pt);
\end{scope}\end{tikzpicture}
\caption{The edges of $\bigcup_{7\leq i\leq 10}\Pi_i$ shown in the universal cover of $\mathbb{T}$.}
\label{fig:jigsaw}
\end{figure}
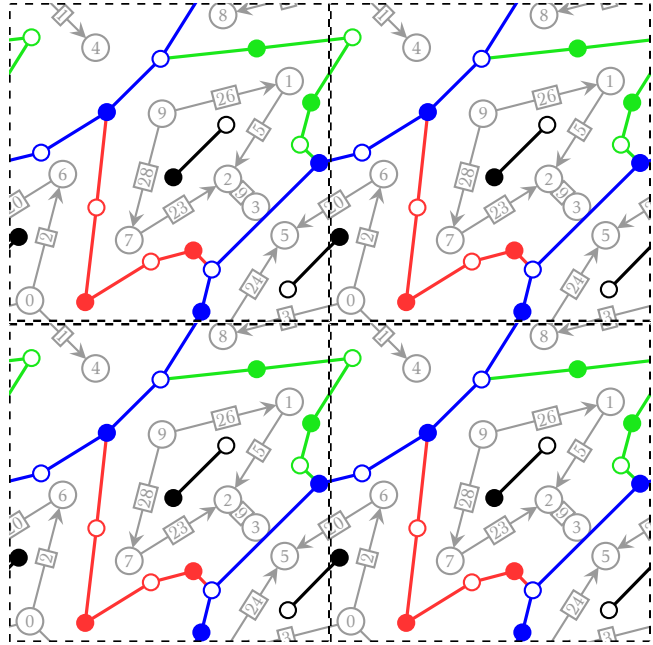  
 The boundary of $\hex(\sigma_-)$ comprises the edges in red and blue (compare Figure~\ref{fig:8910}); notice that the green edges of $\mathfrak{c}_+$ split this hexagon into two jigsaw pieces.  If we fix the position of the jigsaw piece containing the tile dual to the vertex of the quiver labelled $0$, and translate the second jigsaw piece along the direction of the blue edges until it sits on the opposite side of the first jigsaw piece, then we cover precisely $\hex(\sigma_+)$ whose boundary edges are blue and green (with the red edges of $\mathfrak{c}_-$ cutting it in two).
\end{example}

 For each cone $\sigma_\pm\in \Sigma(3)$ from Figure~\ref{fig:jigsawcones}, recall the subquiver $Q^{\sigma_\pm}$ of $Q$ that has vertex set $Q_0$ and arrow set $Q^{\sigma_\pm}_1$ comprising the arrows $a\in Q_1$ that support the corresponding $A$-module $M_{\sigma_{\pm}}$.

\begin{corollary}
\label{cor:edgeinPi3}
 Suppose that the head and tail of $a\in Q_1$ are dual to tiles in different jigsaw pieces
for $\tau$. If $a\in Q^{\sigma_+}_1$ then $\mathfrak{e}_a\in \Pi_3$, and similarly, if $a\in Q^{\sigma_-}_1$ then $\mathfrak{e}_a\in \Pi_0$
\end{corollary}
\begin{proof}
For $a\in Q^{\sigma_+}_1$, Lemma~\ref{lem:cutHex} shows that
$\mathfrak{e}_a\in \mathfrak{c}_- = \mathfrak{m}_{1,3}\cap \mathfrak{m}_{2,3}$,
so either $\mathfrak{e}_a\in \Pi_3$ or $\mathfrak{e}_a\in \Pi_1\cap \Pi_2$.
However, $\mathfrak{e}_a\notin \bigcup_{i=0}^2\Pi_i$ because
$a\in Q^{\sigma_+}_1$, giving $\mathfrak{e}_a\in \Pi_3$ as required. The case $a\in Q^{\sigma_-}_1$ is similar.
\end{proof}

Let $\hex(\sigma_+)^\circ$ and $\partial \hex(\sigma_+)$ denote the interior and the boundary of $\hex(\sigma_+)$ respectively. 

\begin{lemma} 
\label{lem:cuts} 
The cuts $\mathfrak{c}_{-}$ of $\hex(\sigma_+)$ and $\mathfrak{c}_+$ of $\hex(\sigma_-)$ have the following properties.
\begin{enumerate}
\item[\one] \label{lem:cut1} Neither $\mathfrak{c}_-$ nor $\mathfrak{c}_+$ intersects itself;
\item[\two] \label{lem:cut2} If an edge lies in either $\mathfrak{c}_{-}\cap \partial\hex(\sigma_+)$ or $\mathfrak{c}_+\cap \partial\hex(\sigma_-)$, then this edge lies in $\mathfrak{c}_-\cap\mathfrak{c}_+$, and the connected component of $\mathfrak{c}_-\cap\mathfrak{c}_+$ containing this edge comprises an odd number of edges; 
\item[\three] \label{lem:cut3}
 If we follow $\mathfrak{c}_-$ $($resp.\ $\mathfrak{c}_+)$ in either direction around the torus, then the list of nodes where we depart from and arrive at $\partial\hex(\sigma_+)$ \emph{(}resp.\ $\partial\hex(\sigma_-))$ have colours that alternate between black and white.
 \item[\four] The set of nodes that are simultaneously an endnode of an edge in $\mathfrak{c}_-$ $($resp.\ $\mathfrak{c}_+)$ and an endnode of an edge in $\mathfrak{m}_{1,2}$ comprises one black node and one white node; these are the trivalent nodes of $\tiling(\sigma_-)$ $($resp.\ $\tiling(\sigma_+))$.
 \end{enumerate}
\end{lemma}
\begin{proof}
We prove each statement for $\mathfrak{c}_-$ and $\sigma_+$; swap signs throughout to obtain the proof for $\mathfrak{c}_+$ and $\sigma_-$. For \one, if $\mathfrak{c}_-$ intersects itself at a node $\mathfrak{n}\in \Gamma_0$, then at least three edges of $\mathfrak{c}_-$ touch $\mathfrak n$. Since edges in $\mathfrak{c}_-$ belong alternately to $\Pi_3$ and $\Pi_1\cap\Pi_2$, it follows that two of the three edges in $\mathfrak{c}_-$ touching $\mathfrak n$ belong to the same perfect matching, a contradiction. 
For \two, let $\mathfrak{e}$ be an edge of $\mathfrak{c}_{-}\cap \partial\hex(\sigma_+)$. Edges of $\mathfrak{c}_-$ belong alternately to $\Pi_3$ and $\Pi_1\cap\Pi_2$, and Figure~\ref{fig:fundhexpm} illustrates that $\mathfrak{e}$ is forced to lie in $\mathfrak{c}_+\subset \partial\hex(\sigma_+)$. Moreover, since edges of $\mathfrak{c}_+$ belong alternately to $\Pi_0$ and $\Pi_1\cap\Pi_2$, the connected component of $\mathfrak{c}_-\cap\mathfrak{c}_+$ containing $\mathfrak{e}$ must begin and end with an edge in $\Pi_1 \cap \Pi_2$, so this component comprises an odd number of edges. For \three, we deduce from \two\ that the endnodes of each connected component of $\mathfrak{c}_-\cap \mathfrak{c}_+$ must be of different colours. It suffices therefore to prove that each connected component of $\mathfrak{c}_-\cap\hex(\sigma_+)^\circ$ also contains an odd number of edges, as again the endnodes will be of different colours. Every node in $\partial\hex(\sigma_+)$ is an endnode of an edge in $\partial\hex(\sigma_+)$ that lies in either $\Pi_1$ or $\Pi_2$ (or both); see Figure~\ref{fig:fundhexpm}. Since the edges of $\mathfrak{c}_-$ alternate between edges of $\Pi_3$ and $\Pi_1\cap\Pi_2$, an edge in $\mathfrak{c}_-\cap \hex(\sigma_+)^\circ$ that has an endnode in $\partial \hex(\sigma_+)$ must therefore lie in $\Pi_3$. It follows that each connected component of $\mathfrak{c}_-\cap\hex(\sigma_+)^\circ$ contains an odd number of edges. Statement \four\ follows by combining statement \three\ with an examination of the endnodes of $\mathfrak{c}_-$.
\end{proof}

\begin{figure}[!ht]
\centering
\begin{tikzpicture} [thick,scale=0.35, every node/.style={scale=1}] 
\node at (-45pt,-210pt) [rectangle,draw=white,fill=white,minimum size=6pt,inner sep=1pt]{\scriptsize $\mathfrak{m}_{1,2}$};
\node at (105pt,-120pt) [rectangle,draw=white,fill=white,minimum size=6pt,inner sep=1pt]{\scriptsize $\mathfrak{m}_{1,2}$};
\node at (-45pt,25pt) [rectangle,draw=white,fill=white,minimum size=6pt,inner sep=1pt]{\scriptsize $\mathfrak{m}_{1,2}$};
\node at (105pt,115pt) [rectangle,draw=white,fill=white,minimum size=6pt,inner sep=1pt]{\scriptsize $\mathfrak{m}_{1,2}$};
\draw [thick] (0pt,0pt) -- (50pt,90pt) -- (150pt,90pt) -- (200pt,0pt) -- (150pt,-90pt) -- (50pt,-90pt) -- (0pt, 0pt);
\draw [red, very thick] (150pt,90pt) -- (200pt,0pt);
\draw [red, very thick] (50pt,-90pt) -- (0pt, 0pt);
\draw [red, very thick] (-150pt,-90pt) -- (-100pt, -180pt);
\draw [thick] (0pt,0pt) -- (-100pt,0pt) -- (-150pt,-90pt);
\draw [thick] (-100pt,-180pt) -- (0pt,-180pt) -- (50pt,-90pt);
\draw [ggreen, dashed, very thick] (-30pt,-180pt) -- (-30pt,-150pt) -- (-10pt,-140pt) -- (-5pt,-80pt) -- (30pt,-54pt) -- (15pt,-27pt)-- (50pt,00pt) --  (60pt,50pt)--(75pt,60pt) -- (90pt,90pt);
\draw [ggreen, dashed, very thick] (120pt,-90pt) -- (120pt,-60pt) -- (140pt,-50pt) -- (145pt,10pt) -- (180pt,36pt) -- (165pt,63pt);-- (200pt,90pt) --  (210pt,140pt)--(225pt,150pt) -- (240pt,180pt);
\draw [ggreen, dashed, very thick] (-135pt,-117pt)-- (-100pt,-90pt) --  (-90pt,-40pt)--(-75pt,-30pt) -- (-60pt,0pt);
\draw [ggreen, dashed, very thick] (-120pt,-144pt) -- (-135pt,-117pt);
\node at (0pt,0pt) [circle,draw,fill=black,minimum size=6pt,inner sep=1pt]{};;
\node at (150pt,90pt) [circle,draw,fill=black,minimum size=6pt,inner sep=1pt]{};
\node at (150pt,-90pt) [circle,draw,fill=black,minimum size=6pt,inner sep=1pt]{};
\node at (-150pt,-90pt) [circle,draw,fill=black,minimum size=6pt,inner sep=1pt]{};;
\node at (0pt,-180pt) [circle,draw,fill=black,minimum size=6pt,inner sep=1pt]{};
\node at (-120pt,-144pt) [circle,draw,fill=black,minimum size=3pt,inner sep=1pt]{};
\node at (30pt,-54pt) [circle,draw,fill=black,minimum size=3pt,inner sep=1pt]{};
\node at (180pt,36pt) [circle,draw,fill=black,minimum size=3pt,inner sep=1pt]{};
\node at (-60pt,0pt) [circle,draw,fill=black,minimum size=3pt,inner sep=1pt]{};
\node at (90pt,90pt) [circle,draw,fill=black,minimum size=3pt,inner sep=1pt]{};
\node at (90pt,-90pt) [circle,draw,fill=black,minimum size=3pt,inner sep=1pt]{};
\node at (-60pt,-180pt) [circle,draw,fill=black,minimum size=3pt,inner sep=1pt]{};
\node at (50pt,90pt) [circle,draw,fill=white,minimum size=6pt,inner sep=1pt]{};
\node at (200pt,0pt) [circle,draw,fill=white,minimum size=6pt,inner sep=1pt]{};
\node at (50pt,-90pt) [circle,draw,fill=white,minimum size=6pt,inner sep=1pt]{};
\node at (-100pt,0pt) [circle,draw,fill=white,minimum size=6pt,inner sep=1pt]{};
\node at (-100pt,-180pt) [circle,draw,fill=white,minimum size=6pt,inner sep=1pt]{};
\node at (-30pt,-180pt) [circle,draw,fill=white,minimum size=3pt,inner sep=1pt]{};
\node at (120pt,-90pt) [circle,draw,fill=white,minimum size=3pt,inner sep=1pt]{};
\node at (-135pt,-117pt) [circle,draw,fill=white,minimum size=3pt,inner sep=1pt]{};
\node at (15pt,-27pt) [circle,draw,fill=white,minimum size=3pt,inner sep=1pt]{};
\node at (165pt,63pt) [circle,draw,fill=white,minimum size=3pt,inner sep=1pt]{};
\node at (-30pt,0pt) [circle,draw,fill=white,minimum size=3pt,inner sep=1pt]{};
\node at (120pt,90pt) [circle,draw,fill=white,minimum size=3pt,inner sep=1pt]{};
\end{tikzpicture}
\caption{$\hex(\sigma_+)$ and one of its translates are both pictured with the cut $\mathfrak{c}_-$ (in dashed green); we highlight the edges of $\mathfrak{c}_+$ (in red) to illustrate the statements of  Lemma~\ref{lem:cuts}.}
\label{fig:cutsthroughhexagons}
\end{figure}
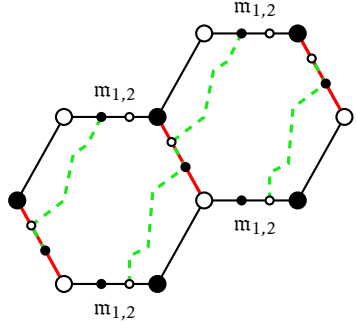

\section{Combinatorial Reid's recipe}
\subsection{The $0$-generated stability condition}
 Our normalisation of the tautological bundle $T$ is chosen so that $L_0\cong \mathcal{O}_{\mathcal{M}_\theta}$ for the vertex $0\in Q_0$;  we refer to the tile of $\Dimer$ dual to $0$ as the \emph{zero tile}. Choose a lift of $0\in Q_0$ to the universal cover that we also denote $0\in \widetilde{Q}$; again, we refer to the tile of $\widetilde{\Dimer}$ dual to $0$ as the \emph{zero tile}. For each cone $\sigma\in \Sigma_\theta(3)$, choose the lift $\widetilde{Q^\sigma}$ of the subquiver $Q^\sigma\subset Q$ so that $0\in \widetilde{Q^\sigma}$. In particular, the zero tile in $\widetilde{\Dimer}$ is contained in the fundamental hexagon $\hex(\sigma)$ for every $\sigma\in \Sigma_\theta(3)$. 

 Fix once and for all a \emph{0-generated} stability parameter, \emph{i.\,e.} a parameter of the form
 \begin{equation}
     \label{eqn:0generated}
 \theta = (\theta_i)\in \Theta\text{ such that }\theta_i>0\text{ for all }i\neq 0.
 \end{equation}
 It is well known that $0$-generated stability conditions are generic. Furthermore, we have the following:
 
 \begin{lemma}
 \label{lem:0generated}
 A quiver representation $(v_a)\in \mathbb{C}^{Q_1}$ is $\theta$-stable if and only if for every vertex $i\in Q_0\setminus \{0\}$, there is a path $p=a_\ell\cdots a_1$ in $Q$ from $0$ to $i$ such that $\prod_{1\leq j\leq \ell} v_{a_j}\neq 0$.
 \end{lemma}  
\begin{proof}
 See, for example, \cite[Lemma~3.1(ii)]{BCQ15}.
 \end{proof}

We now study jigsaw transformations of torus-invariant $\theta$-stable $A$-modules for a 0-generated stability parameter $\theta\in \Theta$. As in Figure~\ref{fig:jigsawcones}, choose adjacent three-dimensional cones $\sigma_\pm\in \Sigma_\theta(3)$ and let $\tau:=\sigma_+\cap \sigma_-$ be the common face of dimension two. As above, the subquivers $Q^{\sigma_\pm}$ of $Q$ from $\mathbb{T}$ lift to define subquivers $\widetilde{Q^{\sigma_\pm}}$ of $\widetilde{Q}$ in $\RR^2$ such that $\hex(\sigma_{\pm})$ both contain the zero tile.

\begin{definition}
 The \emph{zero jigsaw piece} for $\tau$, denoted $\jig_0$, is the jigsaw piece in $\widetilde{\Dimer}$ containing the zero tile. 
\end{definition}

 Both $\hex(\sigma_+)$ and $\hex(\sigma_-)$ can be obtained by gluing jigsaw pieces of $\tau$ along parts of their boundaries, so the zero jigsaw piece $\jig_0$ lies in $\hex(\sigma_+)\cap \hex(\sigma_-)$. The next result records an important consequence of
Lemma~\ref{lem:0generated} for certain edges in the boundary of $\jig_0$. Recall that we write $\hex(\sigma_+)^\circ$ and $\partial\hex(\sigma_+)$ for the interior and the boundary respectively of $\hex(\sigma_+)$.
 
 \begin{lemma}
 \label{lem:J1}
 Let $\mathfrak{e}_a\in \Dimer_1$ be an edge in the boundary of $\jig_0$ that lies in the interior of $\hex(\sigma_+)$. The dual arrow $a\in Q_1^{\sigma_+}$ is such that $\tail(a)$ is dual to a tile in $\jig_0$ and $\head(a)$ is dual to a tile that does not lie in $\jig_0$.
 \end{lemma}
 \begin{proof}
Every edge in the boundary of $\jig_0$ is contained in $\partial\hex(\sigma_+)$ or $\mathfrak{c}_-\cap\hex(\sigma_+)^\circ$ by Lemma~\ref{lem:cutHex}. The boundary of $\jig_0$ is not contained in $\partial\hex(\sigma_+)$, otherwise there would only be one jigsaw piece; nor is it contained in $\mathfrak{c}_-\cap\hex(\sigma_+)^\circ$, because $\mathfrak{c}_-$ does not intersect itself by Lemma~\ref{lem:cuts}\one. Therefore, the boundary of $\jig_0$ comprises edges from both $\mathfrak{c}_-\cap \hex(\sigma_+)^\circ$ and $\partial \hex(\sigma_+)$. Let $\gamma$ be any connected component of $\mathfrak{c}_-\cap \hex(\sigma_+)^\circ$ that lies in the boundary of $\jig_0$.  Note that $\gamma$ cuts $\hex(\sigma_+)$ into two connected components, say $C_0$ and $C_1$; assume without loss of generality that the interior of $\jig_0$ is contained in $C_0$. Suppose for a contradiction that the arrow $a$ dual to an edge $\mathfrak{e}_a$ in the path $\gamma$ is such that $\tail(a)$ is dual to a tile in $C_1$ and $\head(a)$ is dual to a tile in $\jig_0$. Since edges of $\gamma$ belong alternately to $\Pi_3$ and $\Pi_1\cap \Pi_2$, every arrow of $Q^{\sigma_+}$ dual to an edge in $\gamma$ also has tail dual to a tile in $C_1$ and head dual to a tile in $\jig_0$. But then for any vertex $i$ dual to a tile in $C_1$, there does not exist a path $p$ in $Q^{\sigma_+}$ from $0\in \jig_0$ to $i\in C_1$; this contradicts Lemma~\ref{lem:0generated}. 
 \end{proof}
 
 The statement of Lemma~\ref{lem:J1} also holds if we replace $+$ by $-$ throughout.
 We now construct an $A$-submodule $N(\tau)_+$ of the torus-invariant $A$-module $M_{\sigma_+}$. 
 
\begin{definition}
\label{def:Qpmtau}
Let $Q(\tau)^+$ denote the subquiver of $Q^{\sigma_+}$ with  vertex set and arrow set satisfying
\begin{eqnarray*}
Q(\tau)_0^+& := & \big\{i\in Q_0 \mid \text{the tile dual to }i\text{ does not lie in }\jig_0\big\}
\\
Q(\tau)_1^+& := & \big\{a\in Q_1^{\sigma_+} \mid \tail(a), \head(a)\text{ are dual to tiles that do not lie in }\jig_0\big\}.
\end{eqnarray*}
 Consider the dimension vector $\underline{d}:=(d_i)\in \NN^{Q_0}$ with $d_i=1$ for $i\in Q(\tau)_0^+$ and $d_i=0$ otherwise, and define a representation of $Q$ of dimension vector $\underline{d}$ by associating to each arrow $a\in Q_1$ the scalar 
 \[
 v_a=\left\{\begin{array}{cr} 1 & \text{if } a\in Q(\tau)_1^+; \\
0 & \text{otherwise.}
\end{array}\right.
\]
 Let $N(\tau)_+$ denote the corresponding $\CC Q$-module of dimension vector $\underline{d}$.
\end{definition}

\begin{proposition}
\label{prop:quotient} 
The $\CC Q$-module $N(\tau)_+$ is a proper, nonzero $A$-submodule of $M_{\sigma_+}$. 
\end{proposition} 
\begin{proof}
 The $\CC Q$-module $M_{\sigma_+}$ is an $A$-module because it satisfies the relations \eqref{eqn:Jacobian}. If we prove that $N(\tau)_+$ is a $\CC Q$-submodule of $M_{\sigma_+}$, then it will be an $A$-submodule because it will satisfy the same relations. For this, it suffices to show every arrow $a\in Q^{\sigma_+}_1$ with $\tail(a)\in Q(\tau)_0^+$ also has $\head(a)\in Q(\tau)_0^+$. If we suppose otherwise, then there exists $a\in Q_1^{\sigma_+}$ with $\tail(a)\in Q(\tau)_0^+$ and $\head(a)\not\in Q(\tau)_0^+$. The dual edge $\mathfrak{e}_a\in \Dimer_1$ lies in the boundary of $\jig_0$ and in the interior of $\hex(\sigma_+)$, so Lemma~\ref{lem:J1} gives a contradiction. It follows that $N(\tau)_+$ is an $A$-submodule of $M_{\sigma_+}$. Since $\hex(\sigma_+)$ and $\hex(\sigma_-)$ are distinct, there are at least two jigsaw pieces for $\tau$, so $N(\tau)_+$ is nonzero. It's a proper submodule of $M_{\sigma_+}$ because $\jig_0$ contains the zero tile. 
\end{proof}

\begin{remark}
\label{rem:Qtaum} 
If we replace + by $-$ throughout Definition~\ref{def:Qpmtau} and Proposition~\ref{prop:quotient}, then we obtain a proper nonzero submodule $N(\tau)_-$ of $M_{\sigma_-}$. The $A$-modules $N(\tau)_+$ and $N(\tau)_-$ are isomorphic if and only if there are precisely two jigsaw pieces.
\end{remark}

\subsection{Tautological bundles of degree one on the curve}
  Recall that for each vertex $i\in Q_0$, the tautological line bundle $L_i$ on $\mathcal{M}_\theta$ is of the form $L_i\cong \mathcal{O}_{\mathcal{M}_\theta}(\div(p))$ where $p$ is any path in the quiver $Q$ from vertex $0$ to vertex $i$ and where $\div(p)$ is the divisor of zeroes of the section $t^{\div(p)}=\prod_{a\in \supp(p)} t^{\div(a)}$ determined by the sections \eqref{eqn:xa} that label arrows $a$ in the path $p$.
 
\begin{lemma}
\label{lem:charjig+(1)} 
For $i\in Q_0$, let $p$ be a path in $Q^{\sigma_+}$ from vertex 0 to vertex $i$. Then $t_{\rho_3}$ appears with multiplicity one in $t^{\div(p)}$ if and only if vertex $i$ is dual to a tile in a jigsaw piece that lies adjacent to $\jig_0$ in $\hex(\sigma_+)$. 
\end{lemma}
\begin{proof}
 We claim that the multiplicity of $t_{\rho_3}$ in $t^{\div(p)}$ is equal to the number of arrows $a$ in the support of the path $p$ such that the dual edge satisfies $\mathfrak{e}_a\in \mathfrak{c}_-$. To prove the claim, let $a\in Q^{\sigma_+}_1$ and consider two cases: 
 \begin{enumerate}
     \item[\one] if $\mathfrak{e}_a\in \mathfrak{c}_-$, then Corollary~\ref{cor:edgeinPi3} gives $\mathfrak{e}_a\in \Pi_3$, in which case equation~\eqref{eqn:xa} shows that $t_{\rho_3}$ appears with multiplicity one in $t^{\div(a)}$;
     \item[\two] if $\mathfrak{e}_a\not\in \mathfrak{c}_-$, then we must show that $t_{\rho_3}$ does not divide $t^{\div(a)}$ or, equivalently, that $\mathfrak{e}_a\not\in \Pi_3$. Since $a\in Q^{\sigma_+}_1$, the edge $\mathfrak{e}_a$ lies in neither $\Pi_1$ nor $\Pi_2$ (nor $\mathfrak{c}_-$), so $\mathfrak{e}_a$ does not lie in the boundary of $\hex(\sigma_-)$. The proof of  Proposition~\ref{prop:IUvalency} applied to $\hex(\sigma_-)$ shows that $\mathfrak{e}_a\not\in \bigcup_{1\leq i\leq 3} \Pi_{\rho_i}$ as required.
 \end{enumerate}
 This proves the claim. Since the edges of $\mathfrak{c}_-$ provide the cuts that separate jigsaw pieces in $\hex(\sigma_+)$, and since $p$ has tail at vertex $0$, it follows that the number of arrows $a$ in the support of $p$ satisfying $\mathfrak{e}_a\in \mathfrak{c}_-$ is equal to one if and only if the head of $p$ lies in a jigsaw piece that lies adjacent to $\jig_0$ in $\hex(\sigma_+)$.
\end{proof}

\begin{remark}
\label{rem:charjig-(1)} 
 Similarly, if $p$ is a path in $Q^{\sigma_-}$ from vertex 0 to vertex $i$, then $t_{\rho_0}$ appears with multiplicity one in $t^{\div(p)}$ iff vertex $i$ is dual to a tile in a jigsaw piece that lies adjacent to $\jig_0$ in $\hex(\sigma_-)$.
\end{remark}

\begin{proposition}
\label{prop:onejigsaw}
 Let $C_\tau\subseteq \mathcal{M}_\theta$ be the rational curve determined by the cone $\tau=\sigma_+\cap\sigma_-\subseteq \Sigma_\theta$. A tile of $\Dimer$ lies in a jigsaw piece adjacent to $\jig_0$ in $\hex(\sigma_+)$ iff the vertex $i\in Q_0$ dual to the tile satisfies $\deg(L_i\vert_{C_\tau}) = 1$.
\end{proposition}
\begin{proof}
We first compute the coordinate function on the toric chart in $\mathcal{M}_\theta$ corresponding to $\tau$. 
Let $m\in M$ be the primitive vector that is perpendicular to $\tau$ such that $\langle m,n\rangle\geq 0$ for all $n\in \sigma_+$. 
Since $\rho_1, \rho_2\subset \tau$, we have $\langle m,v_{\rho_1}\rangle=\langle m,v_{\rho_2}\rangle=0$, where $v_{\rho_i}\in N=\Hom(M,\ZZ)$ is the primitive lattice point on the ray $\rho_i\in \Sigma_\theta(1)$. 
Our choice of $m$ gives $\langle m,v_{\rho_0}\rangle>0$ and $\langle m,v_{\rho_3}\rangle<0$, see Figure~\ref{fig:jigsawcones}. Since $m$ is a primitive generator and both cones $\sigma_\pm$ are basic, we have $\langle m,v_{\rho_0}\rangle=1$ and $\langle m,v_{\rho_3}\rangle=-1$. 
Thus, if we use the natural inclusion $M\hookrightarrow \ZZ^{\Sigma_\theta(1)}$ to identify $\kk[M]$ with a subring of the ring of Laurent monomials $\kk[t_\rho^{\pm 1} \mid \rho \in \Sigma_\theta(1)]$ in the variables of the Cox ring of $\mathcal{M}_\theta$, then 
\begin{equation}
    \label{eqn:xmxl}
t^m=\frac{t_{\rho_0}}{t_{\rho_3}}\cdot t^{\ell}
\end{equation}
where $\langle \ell,v_{\rho_i}\rangle = 0$ for $0\leq i\leq 3$, \emph{i.\,e.}\ $t^\ell$ is independent of $t_{\rho_0},t_{\rho_1},t_{\rho_2},t_{\rho_3}$. 

Let $i\in Q_0$ be any vertex, let $p_\pm$ be any path in $Q^{\sigma_\pm}$ from 0 to $i$ and write $U_{\sigma_\pm}:=\Spec \kk[\sigma_\pm^\vee\cap M]$ for the toric chart in $\mathcal{M}_\theta$. Given the section $t^{\div(p_+)}$ that generates $ H^0(U_{\sigma_+},L_i)$ as a $\kk[\sigma_+^\vee\cap M]$-module, we obtain the  generating section $t^{\div(p_-)}\in H^0(U_{\sigma_-},L_i)$ directly using the transition function from $U_{\sigma_+}$ to $U_{\sigma_-}$. Explicitly, $t^{\div(p_-)}$ is obtained by multiplying $t^{\div(p_+)}$ by the highest power of the above toric coordinate function $t^m$ such that the resulting Laurent monomial lies in $\kk[\sigma_-^\vee\cap M]$, \emph{i.\,e.}\
\begin{equation}
\label{eqn:degreetransition}
t^{\div(p_-)} = (t^m)^d\cdot t^{\div(p_+)};
\end{equation}
 here $d=\deg(L_i\vert_{C_\tau})$ is the greatest positive integer such that $(t^m)^d\cdot t^{\div(p_+)}\in \kk[\sigma_-^\vee\cap M]$. Since $p_-$ is a path in the quiver $Q^{\sigma_-}$ and since any arrow $a\in Q_1$ satisfying $\mathfrak{e}_a\in \Pi_3$ does not appear in $Q^{\sigma_-}$ by construction, it follows that the left hand side of \eqref{eqn:degreetransition} is independent of $t_{\rho_3}$. The Laurent monomial $t^{\ell}$ is also independent of $t_{\rho_3}$, so Equation \eqref{eqn:xmxl} implies that the multiplicity of the variable $t_{\rho_3}$ in $t^{\div(p_+)}$ must be $d=\deg(L_i\vert_{C_\tau})$. The result follows immediately from Lemma~\ref{lem:charjig+(1)}.
\end{proof}

\begin{remark}
\label{rem:adjacentjig0}
 The statement of Proposition~\ref{prop:onejigsaw} holds with $-$ replacing $+$ because the analogue of Lemma~\ref{lem:charjig+(1)} holds in this case, see Remark~\ref{rem:charjig-(1)}. In particular, the question of whether a tile lies in a jigsaw piece adjacent to $\jig_0$ is independent of whether one views the jigsaw piece in $\hex(\sigma_+)$ or in $\hex(\sigma_-)$. We may therefore refer to a jigsaw piece as being \emph{adjacent to $\jig_0$} without further qualification.
\end{remark}

\begin{corollary} 
\label{cor:commonsource}
The source vertices of the quivers $Q(\tau)^+$ and $Q(\tau)^-$ from Definition~\ref{def:Qpmtau} and Remark~\ref{rem:Qtaum} coincide. 
\end{corollary} 
\begin{proof}
It is enough to prove that the source vertices of the quivers $Q(\tau)^\pm$ are dual to tiles that lie in jigsaw pieces adjacent to $\jig_0$, because then the result follows from Remark~\ref{rem:adjacentjig0}. Let $i\in Q(\tau)_0^+=Q(\tau)_0^-$ be any vertex. Lemma~\ref{lem:0generated} shows that there exist paths $p_+$ and $p_-$ in the quivers $Q^{\sigma_+}$ and $Q^{\sigma_-}$ respectively, such that $\tail(p_+)=0=\tail(p_-)$ and $\head(p_+)=i=\head(p_-)$. The tile dual to the vertex $i$ does not lie in $\jig_0$ by Definition~\ref{def:Qpmtau}, so both $p_+$ and $p_-$ must each traverse precisely one arrow $a_+$ and $a_-$ respectively such that the vertices $\tail(a_+), \tail(a_-)$ are dual to tiles in $\jig_0$ and $\head(a_+), \head(a_-)$ are dual to tiles in jigsaw pieces that lie adjacent to $\jig_0$. Thus, after arriving at the vertices $\head(a_+)$ and $\head(a_-)$, the paths $p_+$ and $p_-$ go on to traverse paths in $Q(\tau)^+$ and $Q(\tau)^-$ respectively that end at vertex $i$. In particular, for every vertex $i\in Q(\tau)^+_0=Q(\tau)^+_0$, there exists a path in $Q(\tau)^+$ and a path in $Q(\tau)^-$ that start at a vertex dual to a tile in a jigsaw piece adjacent to $\jig_0$ and that end at vertex $i$. Thus, the source vertices of $Q(\tau)^+$ and $Q(\tau)^-$ are those dual to tiles in jigsaw pieces adjacent to $\jig_0$. 
\end{proof}

 \begin{example}
 Suppose that a given cone $\tau=\sigma_+\cap\sigma_-$ determined three jigsaw pieces. Figure~\ref{fig:cuttinghexagon} illustrates (lifts to the universal cover of) the jigsaw pieces in $\hex(\sigma_+)$, where they are denoted $\jig_0, \jig_1, \jig_2$, each separated from the others by lifts of the cut $\mathfrak{c}_-$ shown in green.  Figure~\ref{fig:cuttinghexagon} also illustrates how different lifts of the jigsaw pieces form $\hex(\sigma_-)$, where they are denoted $\jig_0, \jig_1^\prime, \jig_2^\prime$. Since $\jig_1$ and $\jig_1^\prime$ are different lifts of the same jigsaw piece, we see that only one jigsaw piece lies \emph{adjacent to $\jig_0$} in the sense of Remark~\ref{rem:adjacentjig0}.
 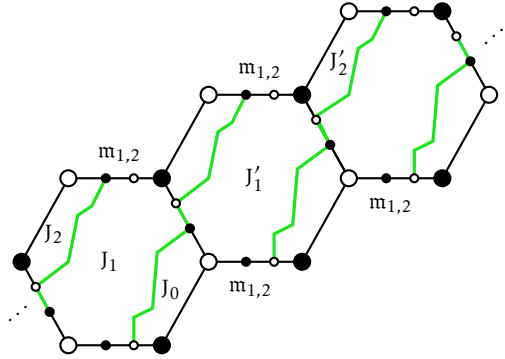
\begin{figure}[!ht]
\centering
\begin{tikzpicture} [thick,scale=0.35, every node/.style={scale=1}] 
\node at (355pt,160pt) [rectangle,draw=white,fill=white,minimum size=6pt,inner sep=1pt]{\scriptsize $\iddots$};
\node at (-154pt,-139pt) [rectangle,draw=white,fill=white,minimum size=6pt,inner sep=1pt]{\scriptsize $\iddots$};
\node at (245pt,-30pt) [rectangle,draw=white,fill=white,minimum size=6pt,inner sep=1pt]{\scriptsize $\mathfrak{m}_{1,2}$};
\node at (95pt,-120pt) [rectangle,draw=white,fill=white,minimum size=6pt,inner sep=1pt]{\scriptsize $\mathfrak{m}_{1,2}$};
\node at (-45pt,25pt) [rectangle,draw=white,fill=white,minimum size=6pt,inner sep=1pt]{\scriptsize $\mathfrak{m}_{1,2}$};
\node at (105pt,115pt) [rectangle,draw=white,fill=white,minimum size=6pt,inner sep=1pt]{\scriptsize $\mathfrak{m}_{1,2}$};
\node at (10pt,-120pt) {\scriptsize $\jig_0$};
\node at (-55pt,-90pt) {\scriptsize $\jig_1$};
\node at (-115pt,-60pt) {\scriptsize $\jig_2$};
\node at (100pt,0pt) {\scriptsize $\jig_1^\prime$};
\node at (190pt,125pt) {\scriptsize $\jig_2^\prime$};
\draw [thick] (0pt,0pt) -- (50pt,90pt) -- (150pt,90pt) -- (200pt,0pt) -- (150pt,-90pt) -- (50pt,-90pt) -- (0pt, 0pt);
\draw [thick] (0pt,0pt) -- (-100pt,0pt) -- (-150pt,-90pt) -- (-100pt,-180pt) -- (0pt,-180pt) -- (50pt,-90pt);
\draw [thick] (150pt,90pt) -- (200pt,180pt) -- (300pt,180pt) -- (350pt,90pt) -- (300pt,0pt) -- (200pt,0pt) -- (150pt, 90pt);
\draw [ggreen, very thick] (-30pt,-180pt) -- (-30pt,-150pt) -- (-10pt,-140pt) -- (-5pt,-80pt) -- (30pt,-54pt) -- (15pt,-27pt)-- (50pt,00pt) --  (60pt,50pt)--(75pt,60pt) -- (90pt,90pt);
\draw [ggreen, very thick] (120pt,-90pt) -- (120pt,-60pt) -- (140pt,-50pt) -- (145pt,10pt) -- (180pt,36pt) -- (165pt,63pt);-- (200pt,90pt) --  (210pt,140pt)--(225pt,150pt) -- (240pt,180pt);
\draw [ggreen, very thick] (270pt,0pt) -- (270pt,30pt) -- (290pt,40pt) -- (295pt,100pt) -- (330pt,126pt)-- (315pt,153pt);
\draw [ggreen, very thick] (-120pt,-144pt) -- (-135pt,-117pt)-- (-100pt,-90pt) --  (-90pt,-40pt)--(-75pt,-30pt) -- (-60pt,0pt);
\draw [ggreen, very thick] (180pt,36pt) -- (165pt,67pt)-- (200pt,90pt) --  (210pt,140pt) -- (225pt,150pt) -- (240pt,180pt);
\node at (0pt,0pt) [circle,draw,fill=black,minimum size=6pt,inner sep=1pt]{};;
\node at (150pt,90pt) [circle,draw,fill=black,minimum size=6pt,inner sep=1pt]{};
\node at (150pt,-90pt) [circle,draw,fill=black,minimum size=6pt,inner sep=1pt]{};
\node at (-150pt,-90pt) [circle,draw,fill=black,minimum size=6pt,inner sep=1pt]{};;
\node at (0pt,-180pt) [circle,draw,fill=black,minimum size=6pt,inner sep=1pt]{};
\node at (300pt,180pt) [circle,draw,fill=black,minimum size=6pt,inner sep=1pt]{};
\node at (300pt,0pt) [circle,draw,fill=black,minimum size=6pt,inner sep=1pt]{};
\node at (-120pt,-144pt) [circle,draw,fill=black,minimum size=3pt,inner sep=1pt]{};
\node at (30pt,-54pt) [circle,draw,fill=black,minimum size=3pt,inner sep=1pt]{};
\node at (180pt,36pt) [circle,draw,fill=black,minimum size=3pt,inner sep=1pt]{};
\node at (330pt,126pt) [circle,draw,fill=black,minimum size=3pt,inner sep=1pt]{};
\node at (-60pt,0pt) [circle,draw,fill=black,minimum size=3pt,inner sep=1pt]{};
\node at (90pt,90pt) [circle,draw,fill=black,minimum size=3pt,inner sep=1pt]{};
\node at (240pt,180pt) [circle,draw,fill=black,minimum size=3pt,inner sep=1pt]{};
\node at (240pt,0pt) [circle,draw,fill=black,minimum size=3pt,inner sep=1pt]{};
\node at (90pt,-90pt) [circle,draw,fill=black,minimum size=3pt,inner sep=1pt]{};
\node at (-60pt,-180pt) [circle,draw,fill=black,minimum size=3pt,inner sep=1pt]{};
\node at (50pt,90pt) [circle,draw,fill=white,minimum size=6pt,inner sep=1pt]{};
\node at (200pt,0pt) [circle,draw,fill=white,minimum size=6pt,inner sep=1pt]{};
\node at (50pt,-90pt) [circle,draw,fill=white,minimum size=6pt,inner sep=1pt]{};
\node at (-100pt,0pt) [circle,draw,fill=white,minimum size=6pt,inner sep=1pt]{};
\node at (-100pt,-180pt) [circle,draw,fill=white,minimum size=6pt,inner sep=1pt]{};
\node at (200pt,180pt) [circle,draw,fill=white,minimum size=6pt,inner sep=1pt]{};
\node at (350pt,90pt) [circle,draw,fill=white,minimum size=6pt,inner sep=1pt]{};
\node at (-30pt,-180pt) [circle,draw,fill=white,minimum size=3pt,inner sep=1pt]{};
\node at (120pt,-90pt) [circle,draw,fill=white,minimum size=3pt,inner sep=1pt]{};
\node at (270pt,0pt) [circle,draw,fill=white,minimum size=3pt,inner sep=1pt]{};
\node at (-135pt,-117pt) [circle,draw,fill=white,minimum size=3pt,inner sep=1pt]{};
\node at (15pt,-27pt) [circle,draw,fill=white,minimum size=3pt,inner sep=1pt]{};
\node at (165pt,63pt) [circle,draw,fill=white,minimum size=3pt,inner sep=1pt]{};
\node at (315pt,153pt) [circle,draw,fill=white,minimum size=3pt,inner sep=1pt]{};
\node at (-30pt,0pt) [circle,draw,fill=white,minimum size=3pt,inner sep=1pt]{};
\node at (120pt,90pt) [circle,draw,fill=white,minimum size=3pt,inner sep=1pt]{};
\node at (270pt,180pt) [circle,draw,fill=white,minimum size=3pt,inner sep=1pt]{};
\end{tikzpicture}
\caption{Translates of $\hex(\sigma_+)$ and several lifts of the cut $\mathfrak{c}_-$ illustrate how $\hex(\sigma_-)$ is formed from jigsaw pieces $\jig_0$ and the translates $\jig_1^\prime$ and $\jig_2^\prime$ of $\jig_1$ and $\jig_2$ respectively.}
\label{fig:cuttinghexagon}
\end{figure}
\end{example}

\subsection{Combinatorial Reid's recipe}
We now introduce Reid's recipe for a consistent dimer model $\Dimer$. We begin by introducing the recipe for marking interior lattice points of $\Sigma_\theta$ with vertices of $Q$. For $i\in Q_0$, write $S_i:=\kk e_i$ for the vertex simple $A$-module corresponding to the vertex $i$. For any $0$-generated $A$-module $M$,  the module $S_i$ is contained in the socle of $M$ if and only if $S_i$ is a submodule of $M$.

\begin{definition}[\textbf{Marking interior lattice points}]
\label{def:RRpoints}
Let $\rho\in \Sigma_\theta(1)$ be an interior lattice point. A vertex $i\in Q_0$ \emph{marks} $\rho$ if $S_i$ lies in the socle of the $A$-module $M_\sigma$ for every cone $\sigma\in \Sigma_\theta(3)$ satisfying $\rho\subset \sigma$.
\end{definition}

\begin{remark}
\label{rem:0doesn'tmark}
The vertex $0$ never marks a lattice point since $\theta$ is a $0$-generated stability parameter.
\end{remark}

\begin{lemma}
 Reid's recipe marks every interior lattice point of $\Sigma_\theta$ with at least one nonzero vertex of $Q$.
\end{lemma}
\begin{proof}
 Let $\rho\in \Sigma_\theta(1)$ be an interior lattice point. We claim that the $A$-module $M_\rho$ is nilpotent, \emph{i.\,e.}\ there exists $n\in\NN$ such that any element $a\in A$ represented by a path of length greater than $n$ in $Q$ satisfies $a\cdot m=0$ for all $m\in M_\rho$. Indeed, any path in $Q$ of length $n:=\vert Q_0\vert +1$ contains a cycle $\gamma$, and it suffices to show that $\gamma\cdot m=0$ for all $m\in M_\rho$. Le Bruyn--Procesi~\cite[Theorem~1]{LP90} associates to $\gamma$ a polynomial function $f_\gamma\in \CC[V]^{T_B}$ on the affine variety $X\cong V\git_0 T_B$. Since $\rho$ is interior, the point $y\in \mathcal{M}_\theta$ associated to the $A$-module $M_\rho$ by Lemma~\ref{lem:conepoint} satisfies $\tau_\theta(y)=x_0$, where $\tau_\theta\colon \mathcal{M}_\theta\to X$ is the crepant resolution from \eqref{eqn:crepantres} and $x_0\in X$ is the unique torus-invariant point.
 It follows that the value of $f_\gamma$ on the quiver representation $(v_a)$ associated to $M_\rho$ by \eqref{eqn:AmoduleOfCone} is zero. In particular, there exists a nonzero vertex $i\in Q_0$ such that $M_\rho$ contains the vertex simple $A$-module $S_i$ as a submodule. This means that $S_i$ lies in the socle of $M_\rho$.

 
 Let $\sigma\in \Sigma_\theta(3)$ satisfy $\rho\subset \sigma$, in which case the module structure on $M_\sigma$ is obtained from $M_\rho$ by setting some of the scalars $v_a$ from \eqref{eqn:AmoduleOfCone} to zero, \emph{i.\,e.}\ $M_\sigma$ is a quotient module of $M_\rho$. Since $M_\sigma$ is isomorphic to $M_\rho$ as a $\kk$-vector space, it follows that having $S_i$ in the socle of $M_\rho$ implies that $S_i$ is in the socle of $M_\sigma$. As a result, vertex $i$ marks the lattice point $\rho$ by Definition~\ref{def:RRpoints}.
\end{proof}

\begin{definition}[\textbf{Marking interior line segments}]
\label{def:RRlines}
Let $\tau\in\Sigma_\theta(2)$ be an interior line segment. A vertex $i\in Q_0$ \emph{marks} $\tau$ if $i$ is one of the common source vertices of the quivers $Q(\tau)^\pm$ from Corollary~\ref{cor:commonsource}.
\end{definition}

\begin{remarks}
\leavevmode
\begin{enumerate}
    \item Vertex $0$ never marks an interior line segment, because $0$ cannot be a vertex of $Q(\tau)^\pm$.
    \item The quivers $Q(\tau)^\pm$ are contained in the acyclic quivers $Q^{\sigma_\pm}$, so $Q(\tau)^\pm$ are themselves acyclic and hence have at least one source. It follows that each interior line segment $\tau\in\Sigma_\theta(2)$ is marked by at least one nonzero vertex of $Q$.
\end{enumerate}
\end{remarks}

\begin{corollary}
\label{cor:RRcompatible}
Let $G$ be a finite abelian subgroup of $\SL(3,\kk)$. Reid's recipe for the fan $\Sigma$ of $\ghilb$ as given by  Definitions~\ref{def:RRpoints} and ~\ref{def:RRlines} agrees with the classical Recipe~\ref{rec:RR}.
\end{corollary}
\begin{proof}
 Recall from Corollary~\ref{cor:Gigsaw} that the isomorphism $\mathcal{M}_\theta\cong \ghilb$ identifies the torus-invariant quiver representations $M_{\sigma_\pm}$ with the $G$-clusters $\kk[x,y,z]/I_{\pm}$ obtained as the fibres of the universal family on $\ghilb$ over the origins in the charts $U_{\sigma_\pm}$. It follows that the marking of an interior lattice point $\rho\in \Sigma(1)$ according to Recipe~\ref{rec:RR}(2) agrees with the marking from Definition~\ref{def:RRpoints}. 

To show that the recipes agree for an interior line segment $\tau\in \Sigma(2)$, consider first the classical recipe. Let $\sigma_\pm\in \Sigma(3)$ satisfy $\tau=\sigma_+\cap \sigma_-$ and let $m=(m_1, m_2, m_3)\in M$ denote the primitive vector in the normal direction to the hyperplane spanned by $\tau$ satisfying $\langle m,n\rangle\geq 0$ for all $n\in \sigma_+$. Nakamura~\cite[Lemma~1.8]{Nakamura01} implies that the denominator of the $G$-invariant Laurent monomial $x^{m_1}y^{m_2}z^{m_3}$ is one of the monomials of $\mathcal{S}_+$, say $x^ay^bz^c$, and the classical Recipe~\ref{rec:RR}(1) marks $\tau$ with the vertex $i\in Q_0\setminus \{0\}$ such that $x^ay^bz^c$ lies in the $i$-character space. Crucially,  Proposition~\ref{prop:Nakamura} shows that a monomial in $\mathcal{S}_+$ moves during the $G$-igsaw transformation across $\tau$ if and only if it is divisible by $x^ay^bz^c$, \emph{i.\,e.}\ the directed subgraph of $\mathcal{S}_+$ comprising monomials that move in the $G$-igsaw transformation has a unique source given by the monomial $x^ay^bz^c$. 

 To compare this with the recipe from Definition~\ref{def:RRlines}, identify the $G$-graphs $\mathcal{S}_\pm$ with the quivers $Q^{\sigma_{\pm}}$ as in Corollary~\ref{cor:Gigsaw} and let $x, y, z$ denote the variables in the Cox ring $\kk[t_\rho\mid \rho\in \Sigma(1)]$ indexed by the lattice points $\rho\in \Sigma(1)$ defined by corners of the junior simplex as in Remark~\ref{rem:Gigsaw}. The jigsaw transformation from $M_{\sigma_+}$ to $M_{\sigma_-}$ as in Theorem~\ref{thm:genjigsaw} moves the tile dual to a vertex $j\in Q_0$ if and only if the section $t^{\div(p)}$ labelling a path from $0$ to $j$ in $Q^{\sigma_+}$ is divisible by $x^{a}y^{b}z^{c}$. This jigsaw transformation coincides with that of Nakamura by Corollary~\ref{cor:Gigsaw}, so the subquiver $Q(\tau)^+$ of $Q^{\sigma_+}$ whose vertex set comprises vertices dual to tiles that move in the jigsaw transformation across $\tau$ has a unique source vertex $i$ corresponding to the eigencharacter of the monomial $x^ay^bz^c$. Definition~\ref{def:RRlines} marks $\tau$ with vertex $i$, so the recipes agree.
\end{proof}

\subsection{Examples}
\label{sec:examples}
We now present two examples where Combinatorial Reid's recipe from Definitions~\ref{def:RRpoints} and \ref{def:RRlines} is used to mark the interior lattice points and line segments of the toric fan $\Sigma_\theta$ with vertices of the quiver $Q$ dual to a consistent dimer model.
 
\begin{example}
\label{exa:LongHexCRR}
Consider the dimer model from Example~\ref{exa:LongHex}. The stability parameter $\theta=(-9,1,1,\ldots,1)$ is $0$-generated and the fan $\Sigma_\theta$ of $\mathcal{M}_\theta$ in Figure~\ref{fig:LHDivisorsTriangulation}(b) shows how we list the lattice points $\rho_1, \dots, \rho_{10}\in \Sigma_\theta(1)$. 

To implement Reid's recipe, first compute the quiver $Q^{\sigma}$ for each cone $\sigma\in \Sigma_\theta(3)$. One such quiver $Q^{\sigma}$ is shown in green in Figure~\ref{fig:8910}: notice that vertices 2 and 5 are both sinks of that quiver, so both vertices are candidates to mark the interior lattice point $\rho_8$ according to Definition~\ref{def:RRpoints}. In fact, both 2 and 5 are sinks in every quiver $Q^\sigma$ determined by a cone $\sigma\in \Sigma_\theta(3)$ such that $\rho_8\subset \sigma$, so both 2 and 5 mark lattice point $\rho_8$ in Reid's recipe. A similar calculation shows that vertex 2 also marks the lattice point $\rho_9$ as shown in Figure~\ref{fig:LongHexRR}. 
\begin{figure}[!ht]
\centering
\begin{tikzpicture} 
\draw (0,0) -- (0,1) -- (-1,2) -- (-2,3) -- (-3,3) -- (-3,2) -- (-2,1) -- (-1,0) -- (0,0); 
\draw (0,0) -- node [rectangle,draw,fill=white,sloped,inner sep=1pt] {{\tiny 9}}(-1,1) -- node [rectangle,draw,fill=white,sloped,inner sep=1pt] {{\tiny 3,9}}(-2,2) -- node [rectangle,draw,fill=white,sloped,inner sep=1pt] {{\tiny 3}}(-3,3); 
\draw (0,1) -- node [rectangle,draw,fill=white,sloped,inner sep=1pt] {{\tiny 7}}(-1,1) -- node [rectangle,draw,fill=white,sloped,inner sep=1pt] {{\tiny 4}}(-2,1); 
\draw (-1,2) -- node [rectangle,draw,fill=white,sloped,inner sep=1pt] {{\tiny 6}}(-2,2) -- node [rectangle,draw,fill=white,sloped,inner sep=1pt] {{\tiny 6}}(-3,2); 
\draw (-2,3) -- node [rectangle,draw,fill=white,sloped,inner sep=1pt,rotate=90] {{\tiny 8}}(-2,2) -- node [rectangle,draw,fill=white,sloped,inner sep=1pt,rotate=90] {{\tiny 8}}(-2,1); 
\draw (-1,2) -- node [rectangle,draw,fill=white,sloped,inner sep=1pt,rotate=90] {{\tiny 4}}(-1,1) -- node [rectangle,draw,fill=white,sloped,inner sep=1pt,rotate=90] {{\tiny 1}}(-1,0);
\draw (0,0) node[circle,draw,fill=white,minimum size=5pt,inner sep=1pt] {{}};
\draw (0,1) node[circle,draw,fill=white,minimum size=5pt,inner sep=1pt] {{}};
\draw (-2,3) node[circle,draw,fill=white,minimum size=5pt,inner sep=1pt] {{}};
\draw (-3,3) node[circle,draw,fill=white,minimum size=5pt,inner sep=1pt] {{}};
\draw (-3,2) node[circle,draw,fill=white,minimum size=5pt,inner sep=1pt] {{}};
\draw (-1,0) node[circle,draw,fill=white,minimum size=5pt,inner sep=1pt] {{}};
\draw (-2,1) node[circle,draw,fill=white,minimum size=5pt,inner sep=1pt] {{}};
\draw (-2,2) node[circle,draw,fill=white,minimum size=5pt,inner sep=1pt] {{\tiny2, 5}};
\draw (-1,1) node[circle,draw,fill=white,minimum size=5pt,inner sep=1pt] {{\tiny2}};
\draw (-1,2) node[circle,draw,fill=white,minimum size=5pt,inner sep=1pt] {{}};
\end{tikzpicture}
\caption{Marking of $\Sigma_\theta$ using combinatorial Reid's recipe for the dimer model from Example~\ref{exa:LongHex}.}
\label{fig:LongHexRR}
\end{figure}
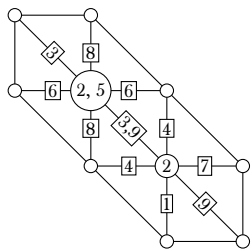
 As for the interior line segments, consider the cone $\tau$ generated by the rays $\rho_8$ and $\rho_9$. Figure~\ref{fig:jigsaw} illustrates the jigsaw pieces for $\tau$: there is precisely one jigsaw piece $\jig_1$ that lies adjacent to $\jig_0$, and the quivers $Q(\tau)^\pm$ supported in $\jig_1$ coincide. The vertices 3 and 9 are the sources of this quiver and, according to Definition~\ref{def:RRlines}, Reid's recipe marks the interior line segment $\tau$ with both 3 and 9. More generally, Combinatorial Reid's recipe marks all interior lattice points and interior line segments of $\Sigma_\theta$ with vertices as shown in Figure~\ref{fig:LongHexRR}. 
\end{example}

\begin{example}
\label{exa:heptrr}
Consider the consistent dimer model shown in black in Figure~\ref{fig:HeptDbQg}(a); the dual quiver $Q$ is illustrated in grey in the same figure. The stability condition $\theta=(-25,1,1,\ldots,1)$ is 0-generated and the labelling sections $t^{\div(a)}$ from equation \eqref{eqn:xa} that determine the tautological isomorphism $\phi\colon A\to \End(T)$ are written explicitly on (or can be deduced from) \cite[Figures~4.2(a), 4.3]{TapiaAmador15}. The toric fan $\Sigma_\theta$ is shown in Figure~\ref{fig:HeptDbQg}(b), where we have marked every interior line segment and interior lattice point of $\Sigma_\theta$ with vertices of the quiver $Q$ according to Combinatorial Reid's recipe.
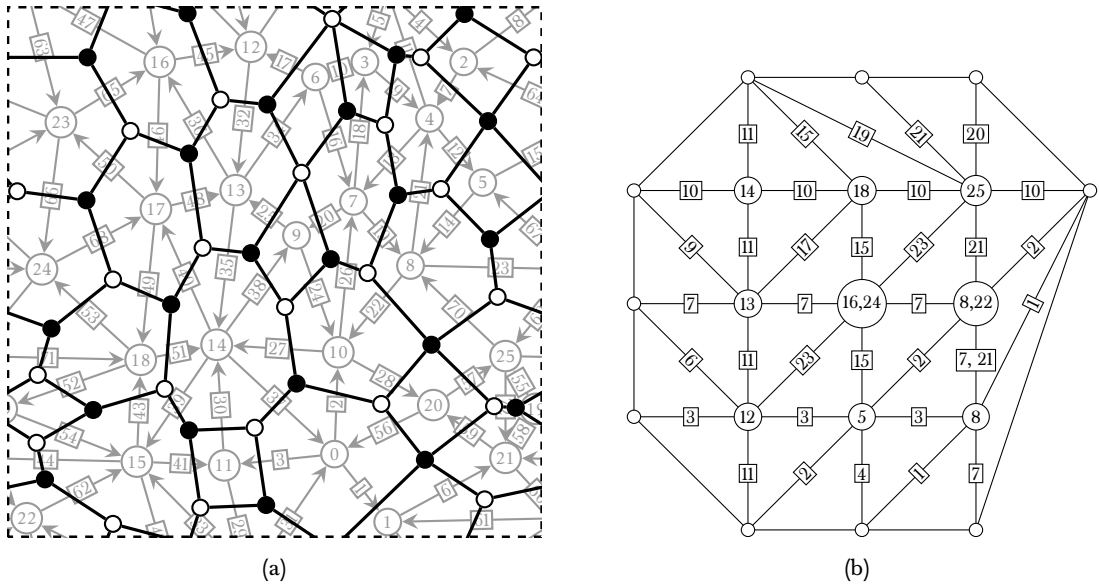
\begin{figure}[ht!]
\centering
    \subfigure[]{\begin{tikzpicture} [thick, scale=0.5, every node/.style={scale=1}] 
\begin{scope}\clip (0pt,0pt) rectangle (400pt,400pt);
\node at (8.6cm,2.2cm) [circle,ggrey,draw=ggrey,fill=white,minimum size=10pt,inner sep=1pt] (0) {\mbox{\tiny $0$}}; 
\node at (10cm,0.4cm) [circle,ggrey,draw=ggrey,fill=white,minimum size=10pt,inner sep=1pt] (1) {\mbox{\tiny $1$}}; 
\node at (12cm,12.6cm) [circle,ggrey,draw=ggrey,fill=white,minimum size=10pt,inner sep=1pt] (2) {\mbox{\tiny $2$}}; 
\node at (9.4cm,12.6cm) [circle,ggrey,draw=ggrey,fill=white,minimum size=10pt,inner sep=1pt] (3) {\mbox{\tiny $3$}}; 
\node at (11.1cm,11.1cm) [circle,ggrey,draw=ggrey,fill=white,minimum size=10pt,inner sep=1pt] (4) {\mbox{\tiny $4$}}; 
\node at (12.5cm,9.4cm) [circle,ggrey,draw=ggrey,fill=white,minimum size=10pt,inner sep=1pt] (5) {\mbox{\tiny $5$}}; 
\node at (8.1cm,12.2cm) [circle,ggrey,draw=ggrey,fill=white,minimum size=10pt,inner sep=1pt] (6) {\mbox{\tiny $6$}}; 
\node at (9.1cm,8.9cm) [circle,ggrey,draw=ggrey,fill=white,minimum size=10pt,inner sep=1pt] (7) {\mbox{\tiny $7$}}; 
\node at (10.6cm,7.2cm) [circle,ggrey,draw=ggrey,fill=white,minimum size=10pt,inner sep=1pt] (8) {\mbox{\tiny $8$}}; 
\node at (7.6cm,8cm) [circle,ggrey,draw=ggrey,fill=white,minimum size=10pt,inner sep=1pt] (9) {\mbox{\tiny $9$}}; 
\node at (8.7cm,4.9cm) [circle,ggrey,draw=ggrey,fill=white,minimum size=10pt,inner sep=1pt] (10) {\mbox{\tiny $10$}}; 
\node at (5.7cm,1.9cm) [circle,ggrey,draw=ggrey,fill=white,minimum size=10pt,inner sep=1pt] (11) {\mbox{\tiny $11$}}; 
\node at (6.4cm,13cm) [circle,ggrey,draw=ggrey,fill=white,minimum size=10pt,inner sep=1pt] (12) {\mbox{\tiny $12$}}; 
\node at (6cm,9.2cm) [circle,ggrey,draw=ggrey,fill=white,minimum size=10pt,inner sep=1pt] (13) {\mbox{\tiny $13$}}; 
\node at (5.5cm,5.1cm) [circle,ggrey,draw=ggrey,fill=white,minimum size=10pt,inner sep=1pt] (14) {\mbox{\tiny $14$}}; 
\node at (3.4cm,2cm) [circle,ggrey,draw=ggrey,fill=white,minimum size=10pt,inner sep=1pt] (15) {\mbox{\tiny $15$}}; 
\node at (4cm,12.6cm) [circle,ggrey,draw=ggrey,fill=white,minimum size=10pt,inner sep=1pt] (16) {\mbox{\tiny $16$}}; 
\node at (3.9cm,8.7cm) [circle,ggrey,draw=ggrey,fill=white,minimum size=10pt,inner sep=1pt] (17) {\mbox{\tiny $17$}}; 
\node at (3.5cm,4.7cm) [circle,ggrey,draw=ggrey,fill=white,minimum size=10pt,inner sep=1pt] (18) {\mbox{\tiny $18$}}; 
\node at (-0.16cm,3.4cm) [circle,ggrey,draw=ggrey,fill=white,minimum size=10pt,inner sep=1pt] (19a) {\mbox{\tiny $19$}}; 
\node at (13.9cm,3.4cm) [circle,ggrey,draw=ggrey,fill=white,minimum size=10pt,inner sep=1pt] (19b) {\mbox{\tiny $19$}}; 
\node at (11.2cm,3.5cm) [circle,ggrey,draw=ggrey,fill=white,minimum size=10pt,inner sep=1pt] (20) {\mbox{\tiny $20$}}; 
\node at (13.1cm,2.1cm) [circle,ggrey,draw=ggrey,fill=white,minimum size=10pt,inner sep=1pt] (21) {\mbox{\tiny $21$}}; 
\node at (0.5cm,0.5cm) [circle,ggrey,draw=ggrey,fill=white,minimum size=10pt,inner sep=1pt] (22) {\mbox{\tiny $22$}}; 
\node at (1.4cm,11cm) [circle,ggrey,draw=ggrey,fill=white,minimum size=10pt,inner sep=1pt] (23) {\mbox{\tiny $23$}}; 
\node at (0.9cm,7.1cm) [circle,ggrey,draw=ggrey,fill=white,minimum size=10pt,inner sep=1pt] (24) {\mbox{\tiny $24$}}; 
\node at (13.1cm,4.8cm) [circle,ggrey,draw=ggrey,fill=white,minimum size=10pt,inner sep=1pt] (25) {\mbox{\tiny $25$}}; 

\draw [-stealthnew,ggrey,arrowhead=6pt] (0) to node [rectangle,draw=ggrey,fill=white,sloped,inner sep=1pt] {{\tiny 1}} (1); 
\draw [-stealthnew,ggrey,arrowhead=6pt] (0) to node [rectangle,draw=ggrey,fill=white,sloped,inner sep=1pt] {{\tiny 2}} (10); 
\draw [-stealthnew,ggrey,arrowhead=6pt] (0) to node [rectangle,draw=ggrey,fill=white,sloped,inner sep=1pt] {{\tiny 3}} (11); 
\draw [-stealthnew,ggrey,arrowhead=6pt] (1) to node [rectangle,draw=ggrey,fill=white,sloped,inner sep=1pt] {{\tiny 4}} (11.5cm,-1.46cm); 
\draw [-stealthnew,ggrey,arrowhead=6pt] (10cm, 14.46cm) to node [rectangle,draw=ggrey,fill=white,sloped,inner sep=1pt] {{\tiny 4}} (2); 
\draw [-stealthnew,ggrey,arrowhead=6pt] (1) to node [rectangle,draw=ggrey,fill=white,sloped,inner sep=1pt] {{\tiny 5}} (9.6cm,-2.74cm); 
\draw [-stealthnew,ggrey,arrowhead=6pt] (10,14.46cm) to node [rectangle,draw=ggrey,fill=white,sloped,inner sep=1pt] {{\tiny 5}} (3); 
\draw [-stealthnew,ggrey,arrowhead=6pt] (1) to node [rectangle,draw=ggrey,fill=white,sloped,inner sep=1pt] {{\tiny 6}} (21); 
\draw [-stealthnew,ggrey,arrowhead=6pt] (2) to node [rectangle,draw=ggrey,fill=white,sloped,inner sep=1pt] {{\tiny 7}} (4); 
\draw [-stealthnew,ggrey,arrowhead=6pt] (2) to node [rectangle,draw=ggrey,fill=white,sloped,inner sep=1pt] {{\tiny 8}} (14.56cm,14.56cm);
\draw [-stealthnew,ggrey,arrowhead=6pt] (-2.56cm,-1.46cm) to node [rectangle,draw=ggrey,fill=white,sloped,inner sep=1pt] {{\tiny 8}} (22);  
\draw [-stealthnew,ggrey,arrowhead=6pt] (11.5cm,-1.46cm) to node [rectangle,draw=ggrey,fill=white,sloped,inner sep=1pt] {{\tiny 8}} (14.56cm,0.5cm);
\draw [-stealthnew,ggrey,arrowhead=6pt] (3) to node [rectangle,draw=ggrey,fill=white,sloped,inner sep=1pt] {{\tiny 9}} (4); 
\draw [-stealthnew,ggrey,arrowhead=6pt] (3) to node [rectangle,draw=ggrey,fill=white,sloped,inner sep=1pt] {{\tiny 10}} (6); 
\draw [-stealthnew,ggrey,arrowhead=6pt] (4) to node [rectangle,draw=ggrey,fill=white,sloped,inner sep=1pt] {{\tiny 11}} (10cm,14.46cm);
\draw [-stealthnew,ggrey,arrowhead=6pt] (11.1cm,-2.96cm) to node [rectangle,draw=ggrey,fill=white,sloped,inner sep=1pt] {{\tiny 11}} (1);  
\draw [-stealthnew,ggrey,arrowhead=6pt] (4) to node [rectangle,draw=ggrey,fill=white,sloped,inner sep=1pt] {{\tiny 12}} (5); 
\draw [-stealthnew,ggrey,arrowhead=6pt] (4) to node [rectangle,draw=ggrey,fill=white,sloped,inner sep=1pt] {{\tiny 13}} (7); 
\draw [-stealthnew,ggrey,arrowhead=6pt] (5) to node [rectangle,draw=ggrey,fill=white,sloped,inner sep=1pt] {{\tiny 14}} (8); 
\draw [-stealthnew,ggrey,arrowhead=6pt] (5) to node [rectangle,draw=ggrey,fill=white,sloped,inner sep=1pt,pos=0.4] {{\tiny 15}} (15.46cm,11cm); 
\draw [-stealthnew,ggrey,arrowhead=6pt] (-1.56cm,9.4cm) to node [rectangle,draw=ggrey,fill=white,sloped,inner sep=1pt,pos=0.4] {{\tiny 15}} (23); 
\draw [-stealthnew,ggrey,arrowhead=6pt] (6) to node [rectangle,draw=ggrey,fill=white,sloped,inner sep=1pt] {{\tiny 16}} (7);  
\draw [-stealthnew,ggrey,arrowhead=6pt] (6) to node [rectangle,draw=ggrey,fill=white,sloped,inner sep=1pt] {{\tiny 17}} (12); 
\draw [-stealthnew,ggrey,arrowhead=6pt] (7) to node [rectangle,draw=ggrey,fill=white,sloped,inner sep=1pt] {{\tiny 18}} (3); 
\draw [-stealthnew,ggrey,arrowhead=6pt] (7) to node [rectangle,draw=ggrey,fill=white,sloped,inner sep=1pt] {{\tiny 19}} (8); 
\draw [-stealthnew,ggrey,arrowhead=6pt] (7) to node [rectangle,draw=ggrey,fill=white,sloped,inner sep=1pt] {{\tiny 20}} (9); 
\draw [-stealthnew,ggrey,arrowhead=6pt] (8) to node [rectangle,draw=ggrey,fill=white,sloped,inner sep=1pt] {{\tiny 21}} (4);  
\draw [-stealthnew,ggrey,arrowhead=6pt] (8) to node [rectangle,draw=ggrey,fill=white,sloped,inner sep=1pt] {{\tiny 22}} (10); 
\draw [-stealthnew,ggrey,arrowhead=6pt] (8) to node [rectangle,draw=ggrey,fill=white,sloped,inner sep=1pt] {{\tiny 23}} (14.96cm,7.1cm);
\draw [-stealthnew,ggrey,arrowhead=6pt] (-3.46cm,7.2cm) to node [rectangle,draw=ggrey,fill=white,sloped,inner sep=1pt] {{\tiny 23}} (24); 
\draw [-stealthnew,ggrey,arrowhead=6pt] (9) to node [rectangle,draw=ggrey,fill=white,sloped,inner sep=1pt] {{\tiny 24}} (10); 
\draw [-stealthnew,ggrey,arrowhead=6pt] (9) to node [rectangle,draw=ggrey,fill=white,sloped,inner sep=1pt] {{\tiny 25}} (13); 
\draw [-stealthnew,ggrey,arrowhead=6pt] (10) to node [rectangle,draw=ggrey,fill=white,sloped,inner sep=1pt] {{\tiny 26}} (7);  
\draw [-stealthnew,ggrey,arrowhead=6pt] (10) to node [rectangle,draw=ggrey,fill=white,sloped,inner sep=1pt] {{\tiny 27}} (14); 
\draw [-stealthnew,ggrey,arrowhead=6pt] (10) to node [rectangle,draw=ggrey,fill=white,sloped,inner sep=1pt] {{\tiny 28}} (20); 
\draw [-stealthnew,ggrey,arrowhead=6pt] (11) to node [rectangle,draw=ggrey,fill=white,sloped,inner sep=1pt] {{\tiny 29}} (6.4cm,-1.06cm);
\draw [-stealthnew,ggrey,arrowhead=6pt] (5.4cm,16.46cm) to node [rectangle,draw=ggrey,fill=white,sloped,inner sep=1pt] {{\tiny 29}} (12); 
\draw [-stealthnew,ggrey,arrowhead=6pt] (11) to node [rectangle,draw=ggrey,fill=white,sloped,inner sep=1pt] {{\tiny 30}} (14); 
\draw [-stealthnew,ggrey,arrowhead=6pt] (12) to node [rectangle,draw=ggrey,fill=white,sloped,inner sep=1pt] {{\tiny 31}} (8.6cm,16.26cm);  
\draw [-stealthnew,ggrey,arrowhead=6pt] (6.4cm,-1.06cm) to node [rectangle,draw=ggrey,fill=white,sloped,inner sep=1pt] {{\tiny 31}} (0);  
\draw [-stealthnew,ggrey,arrowhead=6pt] (12) to node [rectangle,draw=ggrey,fill=white,sloped,inner sep=1pt] {{\tiny 32}} (13); 
\draw [-stealthnew,ggrey,arrowhead=6pt] (12) to node [rectangle,draw=ggrey,fill=white,sloped,inner sep=1pt] {{\tiny 33}} (3.4cm,16.06cm); 
\draw [-stealthnew,ggrey,arrowhead=6pt] (6.4cm,-1.06cm) to node [rectangle,draw=ggrey,fill=white,sloped,inner sep=1pt] {{\tiny 33}} (15);   
\draw [-stealthnew,ggrey,arrowhead=6pt] (13) to node [rectangle,draw=ggrey,fill=white,sloped,inner sep=1pt] {{\tiny 34}} (6); 
\draw [-stealthnew,ggrey,arrowhead=6pt] (13) to node [rectangle,draw=ggrey,fill=white,sloped,inner sep=1pt] {{\tiny 35}} (14); 
\draw [-stealthnew,ggrey,arrowhead=6pt] (13) to node [rectangle,draw=ggrey,fill=white,sloped,inner sep=1pt] {{\tiny 36}} (16);  
\draw [-stealthnew,ggrey,arrowhead=6pt] (14) to node [rectangle,draw=ggrey,fill=white,sloped,inner sep=1pt] {{\tiny 37}} (0); 
\draw [-stealthnew,ggrey,arrowhead=6pt] (14) to node [rectangle,draw=ggrey,fill=white,sloped,inner sep=1pt] {{\tiny 38}} (9); 
\draw [-stealthnew,ggrey,arrowhead=6pt] (14) to node [rectangle,draw=ggrey,fill=white,sloped,inner sep=1pt] {{\tiny 39}} (15); 
\draw [-stealthnew,ggrey,arrowhead=6pt] (14) to node [rectangle,draw=ggrey,fill=white,sloped,inner sep=1pt] {{\tiny 40}} (17); 
\draw [-stealthnew,ggrey,arrowhead=6pt] (15) to node [rectangle,draw=ggrey,fill=white,sloped,inner sep=1pt] {{\tiny 41}} (11);  
\draw [-stealthnew,ggrey,arrowhead=6pt] (15) to node [rectangle,draw=ggrey,fill=white,sloped,inner sep=1pt] {{\tiny 42}} (4.4cm,-1.46cm); 
\draw [-stealthnew,ggrey,arrowhead=6pt] (3.4cm,16.06cm) to node [rectangle,draw=ggrey,fill=white,sloped,inner sep=1pt] {{\tiny 42}} (16); 
\draw [-stealthnew,ggrey,arrowhead=6pt] (15) to node [rectangle,draw=ggrey,fill=white,sloped,inner sep=1pt] {{\tiny 43}} (18); 
\draw [-stealthnew,ggrey,arrowhead=6pt] (15) to node [rectangle,draw=ggrey,fill=white,sloped,inner sep=1pt] {{\tiny 44}} (-0.96cm,2.1cm); 
\draw [-stealthnew,ggrey,arrowhead=6pt] (17.46cm,2cm) to node [rectangle,draw=ggrey,fill=white,sloped,inner sep=1pt] {{\tiny 44}} (21); 
\draw [-stealthnew,ggrey,arrowhead=6pt] (16) to node [rectangle,draw=ggrey,fill=white,sloped,inner sep=1pt] {{\tiny 45}} (12); 
\draw [-stealthnew,ggrey,arrowhead=6pt] (16) to node [rectangle,draw=ggrey,fill=white,sloped,inner sep=1pt] {{\tiny 46}} (17);  
\draw [-stealthnew,ggrey,arrowhead=6pt] (16) to node [rectangle,draw=ggrey,fill=white,sloped,inner sep=1pt] {{\tiny 47}} (0.5cm,14.56cm);
\draw [-stealthnew,ggrey,arrowhead=6pt] (4.4cm,-1.46cm) to node [rectangle,draw=ggrey,fill=white,sloped,inner sep=1pt] {{\tiny 47}} (22); 
\draw [-stealthnew,ggrey,arrowhead=6pt] (17) to node [rectangle,draw=ggrey,fill=white,sloped,inner sep=1pt] {{\tiny 48}} (13); 
\draw [-stealthnew,ggrey,arrowhead=6pt] (17) to node [rectangle,draw=ggrey,fill=white,sloped,inner sep=1pt] {{\tiny 49}} (18); 
\draw [-stealthnew,ggrey,arrowhead=6pt] (17) to node [rectangle,draw=ggrey,fill=white,sloped,inner sep=1pt] {{\tiny 50}} (23); 
\draw [-stealthnew,ggrey,arrowhead=6pt] (18) to node [rectangle,draw=ggrey,fill=white,sloped,inner sep=1pt] {{\tiny 51}} (14);  
\draw [-stealthnew,ggrey,arrowhead=6pt] (18) to node [rectangle,draw=ggrey,fill=white,sloped,inner sep=1pt] {{\tiny 52}} (19a); 
\draw [-stealthnew,ggrey,arrowhead=6pt] (18) to node [rectangle,draw=ggrey,fill=white,sloped,inner sep=1pt] {{\tiny 53}} (24); 
\draw [-stealthnew,ggrey,arrowhead=6pt] (19a) to node [rectangle,draw=ggrey,fill=white,sloped,inner sep=1pt] {{\tiny 54}} (15); 
\draw [-stealthnew,ggrey,arrowhead=6pt] (19b) to node [rectangle,draw=ggrey,fill=white,sloped,inner sep=1pt] {{\tiny 55}} (25); 
\draw [-stealthnew,ggrey,arrowhead=6pt] (20) to node [rectangle,draw=ggrey,fill=white,sloped,inner sep=1pt] {{\tiny 56}} (0);  
\draw [-stealthnew,ggrey,arrowhead=6pt] (20) to node [rectangle,draw=ggrey,fill=white,sloped,inner sep=1pt] {{\tiny 57}} (25); 
\draw [-stealthnew,ggrey,arrowhead=6pt] (21) to node [rectangle,draw=ggrey,fill=white,sloped,inner sep=1pt] {{\tiny 58}} (19b); 
\draw [-stealthnew,ggrey,arrowhead=6pt] (21) to node [rectangle,draw=ggrey,fill=white,sloped,inner sep=1pt] {{\tiny 59}} (20); 
\draw [-stealthnew,ggrey,arrowhead=6pt] (21) to node [rectangle,draw=ggrey,fill=white,sloped,inner sep=1pt] {{\tiny 60}} (14.56cm,0.5cm);
\draw [-stealthnew,ggrey,arrowhead=6pt] (-0.96,2.1cm) to node [rectangle,draw=ggrey,fill=white,sloped,inner sep=1pt] {{\tiny 60}} (22);  
\draw [-stealthnew,ggrey,arrowhead=6pt] (22) to node [rectangle,draw=ggrey,fill=white,sloped,inner sep=1pt] {{\tiny 61}} (-4.06cm,0.4cm);    
\draw [-stealthnew,ggrey,arrowhead=6pt] (14.56cm,0.5cm) to node [rectangle,draw=ggrey,fill=white,sloped,inner sep=1pt] {{\tiny 61}} (1);
\draw [-stealthnew,ggrey,arrowhead=6pt] (22) to node [rectangle,draw=ggrey,fill=white,sloped,inner sep=1pt] {{\tiny 62}} (15); 
\draw [-stealthnew,ggrey,arrowhead=6pt] (22) to node [rectangle,draw=ggrey,fill=white,sloped,inner sep=1pt] {{\tiny 63}} (1.4cm,-3.06cm); 
\draw [-stealthnew,ggrey,arrowhead=6pt] (0.5cm,14.56cm) to node [rectangle,draw=ggrey,fill=white,sloped,inner sep=1pt] {{\tiny 63}} (23); 
\draw [-stealthnew,ggrey,arrowhead=6pt] (23) to node [rectangle,draw=ggrey,fill=white,sloped,inner sep=1pt] {{\tiny 64}} (-2.56cm,12.6cm);
\draw [-stealthnew,ggrey,arrowhead=6pt] (15.46cm,11cm) to node [rectangle,draw=ggrey,fill=white,sloped,inner sep=1pt] {{\tiny 64}} (2); 
\draw [-stealthnew,ggrey,arrowhead=6pt] (23) to node [rectangle,draw=ggrey,fill=white,sloped,inner sep=1pt] {{\tiny 65}} (16); 
\draw [-stealthnew,ggrey,arrowhead=6pt] (23) to node [rectangle,draw=ggrey,fill=white,sloped,inner sep=1pt] {{\tiny 66}} (24);  
\draw [-stealthnew,ggrey,arrowhead=6pt] (24) to node [rectangle,draw=ggrey,fill=white,sloped,inner sep=1pt] {{\tiny 67}} (-1.56cm,9.4cm);
\draw [-stealthnew,ggrey,arrowhead=6pt] (14.96cm,7.1cm) to node [rectangle,draw=ggrey,fill=white,sloped,inner sep=1pt] {{\tiny 67}} (5); 
\draw [-stealthnew,ggrey,arrowhead=6pt] (24) to node [rectangle,draw=ggrey,fill=white,sloped,inner sep=1pt] {{\tiny 68}} (17); 
\draw [-stealthnew,ggrey,arrowhead=6pt] (24) to node [rectangle,draw=ggrey,fill=white,sloped,inner sep=1pt,pos=0.62] {{\tiny 69}} (-0.96cm,4.8cm); 
\draw [-stealthnew,ggrey,arrowhead=6pt] (14.96cm,7.1cm) to node [rectangle,draw=ggrey,fill=white,sloped,inner sep=1pt,pos=0.62] {{\tiny 69}} (25); 
\draw [-stealthnew,ggrey,arrowhead=6pt] (25) to node [rectangle,draw=ggrey,fill=white,sloped,inner sep=1pt] {{\tiny 70}} (8); 
\draw [-stealthnew,ggrey,arrowhead=6pt] (25) to node [rectangle,draw=ggrey,fill=white,sloped,inner sep=1pt] {{\tiny 71}} (17.56cm,4.7cm);
\draw [-stealthnew,ggrey,arrowhead=6pt] (-0.96cm,4.8cm) to node [rectangle,draw=ggrey,fill=white,sloped,inner sep=1pt] {{\tiny 71}} (18);
\draw [-stealthnew,ggrey,arrowhead=6pt] (25) to node [rectangle,draw=ggrey,fill=white,sloped,inner sep=1pt] {{\tiny 72}} (21);
\node at (8.54cm,13.74cm) [circle,draw=black,fill=white,minimum size=6pt,inner sep=1pt](1d){};
\node at (6.8cm,0.85cm) [circle,draw=black,fill=black,minimum size=6pt,inner sep=1pt](2d){};
\node at (6.5cm,2.9cm) [circle,draw=black,fill=white,minimum size=6pt,inner sep=1pt](3d){};
\node at (7.6cm,4.07cm) [circle,draw=black,fill=black,minimum size=6pt,inner sep=1pt](4d){};
\node at (9.83cm,3.53cm) [circle,draw=black,fill=white,minimum size=6pt,inner sep=1pt](5d){};
\node at (10.98cm,2.05cm) [circle,draw=black,fill=black,minimum size=6pt,inner sep=1pt](6d){};
\node at (12.55cm,1cm) [circle,draw=black,fill=white,minimum size=6pt,inner sep=1pt](7d){};
\node at (12.02cm,13.87cm) [circle,draw=black,fill=black,minimum size=6pt,inner sep=1pt](8d){};
\node at (10.87cm,12.72cm) [circle,draw=black,fill=white,minimum size=6pt,inner sep=1pt](9d){};
\node at (10.23cm,12.79cm) [circle,draw=black,fill=black,minimum size=6pt,inner sep=1pt](10d){};
\node at (13.84cm,12.72cm) [circle,draw=black,fill=white,minimum size=6pt,inner sep=1pt](11d){};
\node at (12.64cm,11.03cm) [circle,draw=black,fill=black,minimum size=6pt,inner sep=1pt](12d){};
\node at (11.4cm,9.23cm) [circle,draw=black,fill=white,minimum size=6pt,inner sep=1pt](13d){};
\node at (10.27cm,9.07cm) [circle,draw=black,fill=black,minimum size=6pt,inner sep=1pt](14d){};
\node at (9.93cm,10.93cm) [circle,draw=black,fill=white,minimum size=6pt,inner sep=1pt](15d){};
\node at (8.93cm,11.3cm) [circle,draw=black,fill=black,minimum size=6pt,inner sep=1pt](16d){};
\node at (0.25cm,9.17cm) [circle,draw=black,fill=white,minimum size=6pt,inner sep=1pt](17d){};
\node at (12.69cm,7.9cm) [circle,draw=black,fill=black,minimum size=6pt,inner sep=1pt](18d){};
\node at (7.73cm,9.66cm) [circle,draw=black,fill=white,minimum size=6pt,inner sep=1pt](19d){};
\node at (6.83cm,11.47cm) [circle,draw=black,fill=black,minimum size=6pt,inner sep=1pt](20d){};
\node at (9.47cm,7cm) [circle,draw=black,fill=white,minimum size=6pt,inner sep=1pt](21d){};
\node at (8.5cm,7.37cm) [circle,draw=black,fill=black,minimum size=6pt,inner sep=1pt](22d){};
\node at (12.89cm,6.37cm) [circle,draw=black,fill=white,minimum size=6pt,inner sep=1pt](23d){};
\node at (11.15cm,5.1cm) [circle,draw=black,fill=black,minimum size=6pt,inner sep=1pt](24d){};
\node at (7.3cm,6.1cm) [circle,draw=black,fill=white,minimum size=6pt,inner sep=1pt](25d){};
\node at (6.4cm,7.53cm) [circle,draw=black,fill=black,minimum size=6pt,inner sep=1pt](26d){};
\node at (5.07cm,0.78cm) [circle,draw=black,fill=white,minimum size=6pt,inner sep=1pt](27d){};
\node at (4.77cm,2.83cm) [circle,draw=black,fill=black,minimum size=6pt,inner sep=1pt](28d){};
\node at (5.60cm,11.60cm) [circle,draw=black,fill=white,minimum size=6pt,inner sep=1pt](29d){};
\node at (4.73cm,13.89cm) [circle,draw=black,fill=black,minimum size=6pt,inner sep=1pt](30d){};
\node at (5.13cm,7.67cm) [circle,draw=black,fill=white,minimum size=6pt,inner sep=1pt](31d){};
\node at (4.77cm,10.17cm) [circle,draw=black,fill=black,minimum size=6pt,inner sep=1pt](32d){};
\node at (4.13cm,3.93cm) [circle,draw=black,fill=white,minimum size=6pt,inner sep=1pt](33d){};
\node at (4.30cm,6.17cm) [circle,draw=black,fill=black,minimum size=6pt,inner sep=1pt](34d){};
\node at (2.77cm,0.35cm) [circle,draw=black,fill=white,minimum size=6pt,inner sep=1pt](35d){};
\node at (0.98cm,1.53cm) [circle,draw=black,fill=black,minimum size=6pt,inner sep=1pt](36d){};
\node at (0.76cm,2.50cm) [circle,draw=black,fill=white,minimum size=6pt,inner sep=1pt](37d){};
\node at (2.25cm,3.37cm) [circle,draw=black,fill=black,minimum size=6pt,inner sep=1pt](38d){};
\node at (3.23cm,10.77cm) [circle,draw=black,fill=white,minimum size=6pt,inner sep=1pt](39d){};
\node at (2.10cm,12.72cm) [circle,draw=black,fill=black,minimum size=6pt,inner sep=1pt](40d){};
\node at (2.77cm,6.83cm) [circle,draw=black,fill=white,minimum size=6pt,inner sep=1pt](41d){};
\node at (2.07cm,8.93cm) [circle,draw=black,fill=black,minimum size=6pt,inner sep=1pt](42d){};
\node at (0.79cm,4.30cm) [circle,draw=black,fill=white,minimum size=6pt,inner sep=1pt](43d){};
\node at (1.15cm,5.53cm) [circle,draw=black,fill=black,minimum size=6pt,inner sep=1pt](44d){};
\node at (12.80cm,3.47cm) [circle,draw=black,fill=white,minimum size=6pt,inner sep=1pt](45d){};
\node at (13.37cm,3.43cm) [circle,draw=black,fill=black,minimum size=6pt,inner sep=1pt](46d){};

\draw [very thick] (1d) -- (10.98cm,16.11cm); 
\draw [very thick] (8.54cm,-0.32cm) -- (6d); 
\draw [very thick] (4d) -- (5d); 
\draw [very thick] (2d) -- (3d); 
\draw [very thick] (8d) -- (9d); 
\draw [very thick] (1d) -- (10d); 
\draw [very thick] (6d) -- (7d); 
\draw [very thick] (9d) -- (12d); 
\draw [very thick] (8d) -- (11d); 
\draw [very thick] (10d) -- (15d); 
\draw [very thick] (1d) -- (16d); 
\draw [very thick] (9d) -- (10d); 
\draw [very thick] (12d) -- (13d); 
\draw [very thick] (14d) -- (15d); 
\draw [very thick] (13d) -- (18d); 
\draw [very thick] (12d) -- (14.31,9.17cm); 
\draw [very thick] (-1.42cm,11.03cm) -- (17d); 
\draw [very thick] (16d) -- (19d); 
\draw [very thick] (1d) -- (20d); 
\draw [very thick] (15d) -- (16d); 
\draw [very thick] (14d) -- (21d); 
\draw [very thick] (19d) -- (22d); 
\draw [very thick] (13d) -- (14d); 
\draw [very thick] (21d) -- (24d); 
\draw [very thick] (18d) -- (23d); 
\draw [very thick] (22d) -- (25d); 
\draw [very thick] (19d) -- (26d); 
\draw [very thick] (21d) -- (22d); 
\draw [very thick] (4d) -- (25d); 
\draw [very thick] (5d) -- (24d); 
\draw [very thick] (2d) -- (27d); 
\draw [very thick] (3d) -- (28d); 
\draw [very thick] (1d) -- (6.8cm,14.91cm); 
\draw [very thick] (8.54cm,-0.32cm) -- (2d); 
\draw [very thick] (20d) -- (29d); 
\draw [very thick] (27d) -- (4.73cm,-0.27cm); 
\draw [very thick] (5.07cm,14.84cm) -- (30d); 
\draw [very thick] (19d) -- (20d); 
\draw [very thick] (26d) -- (31d); 
\draw [very thick] (29d) -- (32d); 
\draw [very thick] (3d) -- (4d); 
\draw [very thick] (25d) -- (26d); 
\draw [very thick] (28d) -- (33d); 
\draw [very thick] (31d) -- (34d); 
\draw [very thick] (27d) -- (28d); 
\draw [very thick] (30d) -- (2.77cm,14.41cm); 
\draw [very thick] (4.73cm,-0.17cm) -- (35d); 
\draw [very thick] (33d) -- (38d); 
\draw [very thick] (36d) -- (37d); 
\draw [very thick] (29d) -- (30d); 
\draw [very thick] (32d) -- (39d); 
\draw [very thick] (35d) -- (2.10cm,-1.34cm); 
\draw [very thick] (2.77cm,14.41cm) -- (40d); 
\draw [very thick] (31d) -- (32d); 
\draw [very thick] (34d) -- (41d); 
\draw [very thick] (39d) -- (42d); 
\draw [very thick] (33d) -- (34d); 
\draw [very thick] (38d) -- (43d); 
\draw [very thick] (41d) -- (44d); 
\draw [very thick] (37d) -- (38d); 
\draw [very thick] (46d) -- (14.85cm,4.30cm); 
\draw [very thick] (-0.69cm,3.43cm) -- (43d); 
\draw [very thick] (5d) -- (6d); 
\draw [very thick] (24d) -- (45d); 
\draw [very thick] (46d) -- (14.82cm,2.50cm); 
\draw [very thick] (-0.69cm,3.43cm) -- (37d); 
\draw [very thick] (6d) -- (45d); 
\draw [very thick] (7d) -- (15.04cm,1.53cm); 
\draw [very thick] (-1.51,1cm) -- (36d); 
\draw [very thick] (7d) -- (12.02cm,-0.19cm); 
\draw [very thick] (12.55cm,15.06cm) -- (8d); 
\draw [very thick] (35d) -- (36d); 
\draw [very thick] (11d) -- (16.10cm,12.72cm); 
\draw [very thick] (-0.22cm,12.72) -- (40d); 
\draw [very thick] (11d) -- (12d); 
\draw [very thick] (39d) -- (40d); 
\draw [very thick] (17d) -- (42d); 
\draw [very thick] (18d) -- (14.31cm,9.17cm); 
\draw [very thick] (-1.37cm,7.90cm) -- (17d); 
\draw [very thick] (41d) -- (42d); 
\draw [very thick] (23d) -- (15.21cm,5.53cm); 
\draw [very thick] (-1.17cm,6.37cm) -- (44d); 
\draw [very thick] (23d) -- (24d); 
\draw [very thick] (43d) -- (44d); 
\draw [very thick] (45d) -- (46d); 

\draw[very thick, dashed] (0pt,0pt) rectangle (400pt,400pt);
\end{scope}
\end{tikzpicture}
 }
      \qquad
      \subfigure[]{
      \begin{tikzpicture}[scale=1.5, every node/.style={scale=1}] 
\draw (0,0) -- (-1,1) -- (-3,1) -- (-4,0) -- (-4,-2) -- (-3,-3) -- (-1,-3) -- (0,0); 
\draw (-1,-3) -- node [rectangle,draw,fill=white,sloped,inner sep=1pt,rotate=-90] {\tiny 7} (-1,-2) -- node [rectangle,draw,fill=white,sloped,inner sep=1pt,rotate=-90] {\tiny 7, 21} (-1,-1) -- node [rectangle,draw,fill=white,sloped,inner sep=1pt,rotate=-90] {\tiny 21} (-1,0) -- node [rectangle,draw,fill=white,sloped,inner sep=1pt,rotate=-90] {\tiny 20} (-1,1);
\draw (-2,-3) -- node [rectangle,draw,fill=white,sloped,inner sep=1pt,rotate=-90] {\tiny 4} (-2,-2) -- node [rectangle,draw,fill=white,sloped,inner sep=1pt,rotate=-90] {\tiny 15} (-2,-1) -- node [rectangle,draw,fill=white,sloped,inner sep=1pt,rotate=-90] {\tiny 15} (-2,0);
\draw (-3,-3) -- node [rectangle,draw,fill=white,sloped,inner sep=1pt,rotate=-90] {\tiny 11} (-3,-2) -- node [rectangle,draw,fill=white,sloped,inner sep=1pt,rotate=-90] {\tiny 11} (-3,-1) -- node [rectangle,draw,fill=white,sloped,inner sep=1pt,rotate=-90] {\tiny 11} (-3,0) -- node [rectangle,draw,fill=white,sloped,inner sep=1pt,rotate=-90] {\tiny 11} (-3,1);
\draw (0,0) -- node [rectangle,draw,fill=white,sloped,inner sep=1pt] {\tiny 10} (-1,0) -- node [rectangle,draw,fill=white,sloped,inner sep=1pt] {\tiny 10} (-2,0) -- node [rectangle,draw,fill=white,sloped,inner sep=1pt] {\tiny 10} (-3,0) -- node [rectangle,draw,fill=white,sloped,inner sep=1pt] {\tiny 10} (-4,0);
\draw (-1,-1) -- node [rectangle,draw,fill=white,sloped,inner sep=1pt] {\tiny 7} (-2,-1) -- node [rectangle,draw,fill=white,sloped,inner sep=1pt] {\tiny 7} (-3,-1) -- node [rectangle,draw,fill=white,sloped,inner sep=1pt] {\tiny 7} (-4,-1);
\draw (-1,-2) -- node [rectangle,draw,fill=white,sloped,inner sep=1pt] {\tiny 3} (-2,-2) -- node [rectangle,draw,fill=white,sloped,inner sep=1pt] {\tiny 3} (-3,-2) -- node [rectangle,draw,fill=white,sloped,inner sep=1pt] {\tiny 3} (-4,-2);
\draw (0,0) -- node [rectangle,draw,fill=white,sloped,inner sep=1pt] {\tiny 1} (-1,-2) -- node [rectangle,draw,fill=white,sloped,inner sep=1pt] {\tiny 1} (-2,-3);
\draw (0,0) -- node [rectangle,draw,fill=white,sloped,inner sep=1pt] {\tiny 2} (-1,-1) -- node [rectangle,draw,fill=white,sloped,inner sep=1pt] {\tiny 2} (-2,-2) -- node [rectangle,draw,fill=white,sloped,inner sep=1pt] {\tiny 2} (-3,-3);
\draw (-1,0) -- node [rectangle,draw,fill=white,sloped,inner sep=1pt] {\tiny 23} (-2,-1) -- node [rectangle,draw,fill=white,sloped,inner sep=1pt] {\tiny 23} (-3,-2) -- node [rectangle,draw,fill=white,sloped,inner sep=1pt] {\tiny 6} (-4,-1); 
\draw (-2,0) -- node [rectangle,draw,fill=white,sloped,inner sep=1pt] {\tiny 17} (-3,-1) -- node [rectangle,draw,fill=white,sloped,inner sep=1pt] {\tiny 9} (-4,0);
\draw (-1,0) -- node [rectangle,draw,fill=white,sloped,inner sep=1pt] {\tiny 21} (-2,1);
\draw (-1,0) -- node [rectangle,draw,fill=white,sloped,inner sep=1pt] {\tiny 19} (-3,1);
\draw (-2,0) -- node [rectangle,draw,fill=white,sloped,inner sep=1pt] {\tiny 15} (-3,1);

\draw (0,0) node[circle,draw,fill=white,minimum size=5pt,inner sep=1pt] {};
\draw (-1,-3) node[circle,draw,fill=white,minimum size=5pt,inner sep=1pt] {};
\draw (-1,-2) node[circle,draw,fill=white,minimum size=10pt,inner sep=1pt] {\tiny8};
\draw (-1,-1) node[circle,draw,fill=white,minimum size=10pt,inner sep=1pt] {\tiny8,22};
\draw (-1,0) node[circle,draw,fill=white,minimum size=10pt,inner sep=1pt] {\tiny25};
\draw (-1,1) node[circle,draw,fill=white,minimum size=5pt,inner sep=1pt] {};
\draw (-2,-3) node[circle,draw,fill=white,minimum size=5pt,inner sep=1pt] {};
\draw (-2,-2) node[circle,draw,fill=white,minimum size=10pt,inner sep=1pt] {\tiny5};
\draw (-2,-1) node[circle,draw,fill=white,minimum size=10pt,inner sep=1pt] {\tiny16,24};
\draw (-2,0) node[circle,draw,fill=white,minimum size=10pt,inner sep=1pt] {\tiny18};
\draw (-2,1) node[circle,draw,fill=white,minimum size=5pt,inner sep=1pt] {};
\draw (-3,-3) node[circle,draw,fill=white,minimum size=5pt,inner sep=1pt] {};
\draw (-3,-2) node[circle,draw,fill=white,minimum size=10pt,inner sep=1pt] {\tiny12};
\draw (-3,-1) node[circle,draw,fill=white,minimum size=10pt,inner sep=1pt] {\tiny13};
\draw (-3,0) node[circle,draw,fill=white,minimum size=10pt,inner sep=1pt] {\tiny14};
\draw (-3,1) node[circle,draw,fill=white,minimum size=5pt,inner sep=1pt] {};
\draw (-4,-2) node[circle,draw,fill=white,minimum size=5pt,inner sep=1pt] {};
\draw (-4,-1) node[circle,draw,fill=white,minimum size=5pt,inner sep=1pt] {};
\draw (-4,0) node[circle,draw,fill=white,minimum size=5pt,inner sep=1pt] {};
\end{tikzpicture}
}
\caption{(a) A consistent dimer model; (b) the fan $\Sigma_\theta$ marked according to Reid's recipe.}
\label{fig:HeptDbQg}
\end{figure}
\end{example}

\begin{remark} 
\label{rem:features}
These examples exhibit several new phenomena that are not present in the classical Reid's recipe \cite{Reid97,Craw05} for the toric fan $\Sigma$ of the $G$-Hilbert scheme for a finite abelian subgroup $G$ in $\SL(3,\kk)$. Indeed:
 \begin{enumerate}
     \item distinct interior lattice points can be marked with the same vertex: in Figure~\ref{fig:LongHexRR}, vertex $2$ marks both interior lattice points; while in Figure~\ref{fig:HeptDbQg}(b), vertex 8 marks a pair of lattice points. Contrast this with the classical case, where any vertex of the McKay quiver that marks an interior lattice point of $\Sigma$ does not mark another lattice point or line segment in $\Sigma$ (see \cite[Corollary~4.6]{Craw05}); 
     \item interior line segments can be marked with more than one vertex: in Figure~\ref{fig:LongHexRR}, vertices 3 and 9 both mark the same line segment; in Figure~\ref{fig:HeptDbQg}(b), vertices 7 and 21 mark the same line segment. In the classical case, every interior line segment of $\Sigma$ is marked with a unique vertex of the McKay quiver.
     \item the marking of an interior line segment $\tau\in \Sigma_\theta(2)$ is not determined by the hyperplane containing $\tau$: in Figure~\ref{fig:LongHexRR}, three coplanar cones in $\Sigma_\theta(2)$ are marked with 3, with 3 and 9 and with 9 respectively. In the classical case, the marking of an interior line segment $\tau$ is determined by the normal vector of the hyperplane containing the cone, so each cone in a given hyperplane is marked identically.
    \item the marking of an interior lattice point $\rho\in \Sigma_\theta(1)$ is not determined by the geometry of the toric surface $D_\rho$: in Figure~\ref{fig:HeptDbQg}(b), one del Pezzo surface of degree six is marked by a pair of vertices 16 and 24, while a second del Pezzo surface of degree six is marked by the single vertex 5. In the classical case, the marking is determined by the geometry of the surface (see \cite[Section~3]{Craw05}).
     \item the Euler number of an irreducible component of the exceptional divisor is not bounded above by six:  in Figure~\ref{fig:HeptDbQg}(b), the unique interior lattice point of $\Sigma_\theta$ marked by vertex 25 determines a toric surface with Euler number 7. In the classical case, a result of Craw--Reid~\cite[Corollary~1.4]{CR02} shows that the Euler number of every torus-invariant projective surface $D_\rho$ in $\ghilb$ is at most six.
   \end{enumerate}
 \end{remark}

\section{Compatibility with Geometric Reid's Recipe}
Motivated by Logvinenko's geometric version of Reid's recipe (see Theorem~\ref{thm:Logvinenko}), we now study the relation between Reid's recipe for a consistent dimer model as in Definitions~\ref{def:RRpoints} and \ref{def:RRlines} and the statement of Geometric Reid's recipe given by Bocklandt--Craw--Quintero V\'{e}lez~\cite[Theorem~1.4]{BCQ15}. We conclude with a pair of natural conjectures.

\subsection{Compatibility with Geometric Reid's Recipe}
  To begin we generalise a statement from Cautis--Craw--Logvinenko~\cite[Proposition~4.8]{CCL17} to consistent dimer models.

\begin{lemma}
\label{lem:RRtriangle}
Let $\sigma\in \Sigma_\theta(3)$ and $i\in Q_0$. Assume that the vertex simple $A$-module $S_i$ lies in the socle of the torus-invariant $\theta$-stable $A$-module $M_\sigma$. According to Reid's recipe, vertex $i$ marks either:
\begin{enumerate}
\item[\one] an edge of the triangle $\sigma$, \emph{i.\,e.}\ there exists $\tau\in \Sigma_\theta(2)$ with $\tau\subset \sigma$ such that $i$ marks $\tau$. In this case, $S_i$ lies in the socle of every $\theta$-stable $A$-module in the torus-invariant curve $C_\tau$; or
\item[\two] a node of the triangle $\sigma$, \emph{i.\,e.}\ there exists $\rho\in \Sigma_\theta(1)$ with $\rho\subset \sigma$ such that $i$ marks $\rho$. In this case, $S_i$ lies in the socle of every $\theta$-stable $A$-module in the torus-invariant divisor $D_\rho$.
\end{enumerate}
\end{lemma}
\begin{proof}
Let $y\in \mathcal{M}_\theta$ be the torus-invariant point defined by the cone $\sigma\in \Sigma_\theta(3)$, so $M_y=M_\sigma$ by Lemma~\ref{lem:conepoint}. Since $S_i$ lies in the socle of $M_\sigma$, the vertex $i$ cannot be the zero vertex (unless $\Gamma$ has only one node, in which case $X=\mathbb{C}^3$ and Reid's recipe says nothing of interest). We may therefore apply several results of Bocklandt--Craw--Quintero-V\'{e}lez~\cite[Propositions~3.4, 4.7, Lemma~4.8, Correction]{BCQ15} to deduce that the intersection of 
\[
Z_i:=\big\{y\in \mathcal{M}_\theta \mid S_i\subseteq \soc(M_y)\big\}
\]
with the locus $\tau_\theta^{-1}(x_0)$ is either a single $(-1,-1)$-curve in $\mathcal{M}_\theta$, a single $(0,-2)$-curve in $\mathcal{M}_\theta$ or a connected union of compact torus-invariant divisors in $\mathcal{M}_\theta$. We have $S_i\subseteq \soc(M_y)$ by assumption, hence $y\in Z_i$. We consider each case in turn.

If this locus is a $(-1,-1)$-curve in $\mathcal{M}_\theta$ then there exists $\tau\in \Sigma_\theta(2)$ such that $Z_i=C_\tau$, so $S_i$ lies in the socle of every $\theta$-stable $A$-module parametrised by points of $C_\tau$. Also, $y\in Z_i=C_\tau$, so the inclusion-reversing correspondence between orbit-closures in $\mathcal{M}_\theta$ and cones in the fan $\Sigma_\theta$ gives $\tau\subset \sigma$. To see that Reid's recipe marks $\tau$ with vertex $i$, the result of \cite[Lemma~4.10]{BCQ15} implies that the jigsaw transformation across $\tau$ defines precisely two jigsaw pieces, where the jigsaw piece adjacent to $\jig_0$ comprises a single tile dual to vertex $i$. In particular, $i$ is the source vertex of the quivers $Q^\pm$ associated to $\tau$, so $i$ marks $\tau$ according to Definition~\ref{def:RRlines}.

 The case of a $(0,-2)$-curve in $\mathcal{M}_\theta$ is almost identical, except now the curve $C_\tau$ defined by the cone $\tau\in \Sigma_\theta(2)$ coincides with the intersection $Z_i\cap \tau_\theta^{-1}(x_0)$ as noted in \cite[Correction]{BCQ15}. Otherwise, by applying the analogue of \cite[Lemma~4.10]{BCQ15} as stated in \cite[Correction, Lemma~3]{BCQ15}, the argument from the previous paragraph goes through verbatim to give that vertex $i$ marks the cone $\tau$ as required.
 
Otherwise, $Z_i$ is a compact torus-invariant divisor. Since $y\in Z_i$,  there exists $\rho\in \Sigma_\theta(1)$ such that $y\in D_\rho\subseteq Z_i$, so $S_i$ lies in the socle of every $\theta$-stable $A$-module parametrised by a point in $D_\rho$. It follows from  Definition~\ref{def:RRpoints} that vertex $i$ marks $\rho$. Applying the inclusion-reversing correspondence between orbit closures in $\mathcal{M}_\theta$ and cones in $\Sigma_\theta$ to the inclusion $y\in D_\rho$ gives $\rho\subset \sigma$ as required.
\end{proof}

 Recall that the derived equivalence $\Psi_\theta$ from \eqref{eqn:Psi} sends the vertex simple $A$-module $S_i$ for $i\in Q_0$ to an object $\Psi_\theta(S_i)$ in the bounded derived category of coherent sheaves on $\mathcal{M}_\theta$. In particular, for each $k\in \ZZ$ we obtain a coherent sheaf $H^k(\Psi_\theta(S_i))$.

\begin{proposition}
\label{prop:GRR}
Let $i\in Q_0$ be a nonzero vertex and assume $H^0(\Psi_\theta(S_i))\neq 0$. According to Reid's recipe, vertex $i$ marks either:
\begin{enumerate}
    \item[\one]  a unique interior line segment $\tau\in \Sigma_\theta(2)$, in which case the corresponding torus-invariant curve $C_\tau$ satisfies $\Psi_\theta(S_i)\cong L_i^{-1}\vert_{C_\tau}$; or 
    \item[\two] at least one interior lattice point $\rho\in \Sigma_\theta(1)$, in which case the corresponding torus-invariant divisor $D_\rho$ is contained in the support of $\Psi_\theta(S_i)$. In fact, if $Z_i$ denotes the union of all torus-invariant divisors marked by vertex $i$, then $\Psi_\theta(S_i)\cong L_i^{-1}\vert_{Z_i}$.
\end{enumerate}
\end{proposition}
\begin{proof}
Since $i$ is nonzero and $H^0(\Psi(S_i))\neq 0$, \cite[Proposition~4.7]{BCQ15} gives $\sigma\in \Sigma_\theta(3)$ such that the corresponding torus-invariant $\theta$-stable $A$-module $M_\sigma$ contains $S_i$ in its socle. Now Lemma~\ref{lem:RRtriangle} applies, giving two cases. If Lemma~\ref{lem:RRtriangle}\one\ occurs, then there exists an interior line segment $\tau\in \Sigma_\theta(2)$ such that vertex $i$ marks $\tau$, and $S_i$ lies in the socle of every $\theta$-stable $A$-module in the curve $C_\tau=\{y\in \mathcal{M}_\theta \mid S_i\subseteq \soc(M_y)\}\cap \tau_\theta^{-1}(x_0)$. In this case, \cite[Proposition~1.3, Correction]{BCQ15} gives $\Psi_\theta(S_i)\cong L_i^{-1}\vert_{C_\tau}$ as required. Otherwise Lemma~\ref{lem:RRtriangle}\two\ occurs, giving an interior lattice point $\rho\in \Sigma_\theta(1)$ such that vertex $i$ marks $\rho$, and $S_i$ lies in the socle of every $\theta$-stable $A$-module corresponding to a point of $D_\rho$.  This holds if and only if $D_\rho$ is contained in the locus $Z_i=\{y\in \mathcal{M}_\theta \mid S_i\subseteq \soc(M_y)\}$.  In particular, $Z_i$ is non-empty, so \cite[Proposition~1.3]{BCQ15} applies to show that $Z_i$ is equal to the support of $\Psi_\theta(S_i)$ and in fact $\Psi_\theta(S_i)\cong L_i^{-1}\vert_{Z_i}$. This completes the proof.
 \end{proof}

\begin{remarks}
\label{rem:Ziconnected}
\leavevmode
\begin{enumerate}
    \item The support of $\Psi_\theta(S_i)$ in Proposition~\ref{prop:GRR}\two\ is known to be connected \cite[Lemma~4.8]{BCQ15}. Therefore Proposition~\ref{prop:GRR}\two\ implies that the union of all divisors $D_\rho$ in $\mathcal{M}_\theta$ such that $\rho$ is marked by vertex $i$ is connected.  This statement is not an obvious consequence of Definition~\ref{def:RRpoints}.
    \item The examples from Section~\ref{sec:examples} show that  $Z_i$ from Proposition~\ref{prop:GRR}\two, may be reducible. That the support of $\Psi_\theta(S_i)$ may be reducible was surely known to Ishii and Ueda, see~\cite[Lemma~11.2]{IshiiUeda16}.  
\end{enumerate}
\end{remarks}

 As a result, we may enhance slightly the statement of Geometric Reid's recipe for consistent dimer models from \cite[Theorem~1.4]{BCQ15} to one that makes reference to the marking of lattice points and \emph{some} line segments of the fan $\Sigma_\theta$ according to Definitions~\ref{def:RRpoints} and \ref{def:RRlines}:

\begin{corollary}[Geometric Reid's recipe for dimer models]
\label{cor:CRRGRR}
Let $\Gamma$ be a consistent dimer model and let $i\in Q_0$ be a vertex in the dual quiver. Precisely one of the following statements holds:
\begin{enumerate}
 \item[\one] vertex $i$ marks at least one interior lattice point in $\Sigma_\theta(1)$, in which case $\Psi_\theta(S_i)\cong L_i^{-1}\otimes \mathcal{O}_{Z_i}$, where $Z_i$ is the connected union of torus-invariant divisors indexed by lattice points $\rho$ marked with $i$;
 \item[\two] vertex $i$ marks a unique interior line segment $\tau\in \Sigma_\theta(2)$, in which case $\Psi_\theta(S_i)\cong L_i^{-1}\otimes \mathcal{O}_{C_\tau}$; 
  \item[\three]  $\Psi_\theta(S_i)\cong\mathcal{F}[1]$, where $\mathcal{F}$ is a coherent sheaf whose support is a connected union of compact torus-invariant divisors; or
  \item[\four] $i=0$, in which case $\Psi_\theta(S_0)$ is quasi-isomorphic to the dualising complex of the locus $\tau_\theta^{-1}(x_0)$.
\end{enumerate}
\end{corollary}
\begin{proof}
 The only enhancement to  \cite[Theorem~1.4]{BCQ15} is in parts \one\ and \two, but for convenience we explain the logic of the proof. If $i=0$, then case \four\ holds by \cite[Proposition~3.7]{BCQ15}. Otherwise, $i\neq 0$ and \cite[Theorem~1.1, Proposition~3.4]{BCQ15} gives a unique $k\in \{-1,0\}$ such that $H^k(\Psi_\theta(S_i))\neq 0$. If $k=-1$, then case \three\ holds by \cite[Corollary~5.2, Proposition~5.6]{BCQ15}. Otherwise $k=0$, in which case \cite[Proposition~4.11, Correction]{BCQ15} gives $\Psi_\theta(S_i)\cong L_i^{-1}\vert_{Z_i^c}$ for  $Z_i^c:=\big\{y\in \tau_\theta^{-1}(x_0) \mid S_i\subseteq \soc(M_y)\big\}$. Proposition~\ref{prop:GRR} implies that $Z_i^c$ is determined by the interior lattice points and line segments of $\Sigma_\theta$ marked by $i$ according to Definitions~\ref{def:RRpoints} and \ref{def:RRlines}. Note that $Z_i^c$ is equal to $Z_i=\{y\in \mathcal{M}_\theta \mid S_i\subseteq \soc(M_y)\}$ in case \one, and to $C_\tau$ in case \two.
\end{proof}

\subsection{Open questions}
 After comparing Corollary~\ref{cor:CRRGRR} with the statement of Theorem~\ref{thm:Logvinenko}, it is natural to suggest the following.

\begin{conjecture}
\label{conj:CHT}
Let $i\in Q_0$ be nonzero. Then $\mathcal{F}:=H^{-1}(\Psi_\theta(S_i))\neq 0$ if and only if $i$ marks two or more interior line segments in $\Sigma_\theta$ according to Definition~\ref{def:RRlines}, in which case the support of $\Psi_\theta(S_i)$ is the union of all torus-invariant divisors $D_\rho$ such that two line segments $\tau$ in $\Sigma_\theta$ containing $\rho$ are marked by vertex $i$.  
\end{conjecture}

 \begin{remarks}
 \leavevmode
\begin{enumerate}
     \item A proof of Conjecture~\ref{conj:CHT} would imply that every nonzero vertex of $Q$ appears `once' on the fan $\Sigma_\theta$ in a sense similar to the statement from \cite[Corollary~4.6]{Craw05}. In particular, each nonzero vertex $i\in Q$ would either mark lattice points in $\Sigma_\theta$ (this happens when $H^{0}(\Psi_\theta(S_i))\neq 0$) or line segments in $\Sigma_\theta$ (this happens when $H^{-1}(\Psi_\theta(S_i))\neq 0$). Note that $i$ would not be able to mark both lattice points and line segments, because \cite[Theorem~1.1, Proposition~3.4]{BCQ15} gives a unique $k\in \{-1,0\}$ such that $H^k(\Psi_\theta(S_i))\neq 0$. The examples from Section~\ref{sec:examples} illustrate this dichotomy.
      \item When $H^{-1}(\Psi_\theta(S_i))\neq 0$, the support of the object $\Psi_\theta(S_i)$ is known to be connected \cite[Corollary~5.2]{BCQ15}. A proof of Conjecture~\ref{conj:CHT} would therefore also imply that the union of line segments marked by $i$ is a connected graph in the triangulation. Such a statement would be similar to that from Remark~\ref{rem:Ziconnected}(1) which applies when $i$ marks lattice points.
 \end{enumerate}
 \end{remarks}
 
 Conjecture~\ref{conj:CHT} is not strong enough to be called \emph{Derived Reid's recipe} for consistent dimer models because it does not describe explicitly the objects $\Psi_\theta(S_i)$ in terms of the labelling of lattice points and line segments in $\Sigma_\theta$ as was done for classical Reid's recipe in \cite[Theorem~1.2]{CCL17}.  Such an explicit description would require that one computes the cohomology of the \emph{wheel} associated to $L_i^{-1}$ as defined by Craw--Quintero-V\'{e}lez~\cite{CQ15}. This calculation becomes increasingly complicated as the number of arrows with head at vertex $i$ increases; see Example~3.5 from \emph{ibid.} for one such example.
 
 \medskip
 
 We also anticipate that Combinatorial Reid's recipe will enable one to compute a minimal set of relations in $\Pic(\mathcal{M}_\theta)$ between the tautological line bundles on $\mathcal{M_\theta}$. A proof of the following would extend the result of \cite[Theorem~6.1]{Craw05} to any consistent dimer model.

\begin{conjecture}
\label{conj:rels}
The Picard group of $\mathcal{M}_\theta$ is generated by the nontrivial tautological bundles $\{L_i\}_{i\neq 0}$, modulo the subgroup generated by one relation for each interior lattice point $\rho\in \Sigma_\theta(1)$, namely
\[
\bigotimes_{i\text{ marks }\rho} L_{i}\cong
\bigotimes_{i\in Q_0}
L_{i}^{\otimes(n(i,\rho)-1)}, 
\]
 where $n(i,\rho)$ is the number of interior line segments $\tau\in \Sigma_\theta(2)$ marked by vertex $i$ such that $\rho\subset \tau$.
 \end{conjecture}

\begin{example}
In Example~\ref{exa:LongHexCRR}, one computes explicitly that there are two such relations:
\[
L_2\otimes L_5\cong L_3\otimes L_6\otimes L_8\quad\text{and}\quad L_2\cong L_4\otimes L_9.
\]
Similarly, in Example~\ref{exa:heptrr} there are nine such relations:
\[
L_{14} \cong L_{10}\otimes L_{11}; \quad
L_{18} \cong L_{10}\otimes L_{15};\quad
L_{25} \cong L_{10}\otimes L_{21}; \quad
L_{13} \cong L_{7}\otimes L_{11}; \quad
L_{8} \cong L_{1}\otimes L_{7};
\]
\[
L_{16}\otimes L_{24} \cong  L_{7}\otimes L_{15}\otimes L_{23};\quad
L_{8}\otimes L_{22} \cong L_{2}\otimes L_{7}\otimes L_{21}; \quad
L_{12} \cong L_{3}\otimes L_{11}; \quad
L_{5} \cong L_{2}\otimes L_{3}.
\]
 Additional examples where Conjectures~\ref{conj:CHT} and \ref{conj:rels} have been verified can be found in \cite{TapiaAmador15}.
\end{example}

A proof of Conjecture~\ref{conj:rels} would enable one to cook-up virtual bundles in the spirit of the classical case (see Reid~\cite{Reid97} and Craw~\cite[Section~7]{Craw05}), leading to a $\ZZ$-basis of the integral cohomology  $H^*(\mathcal{M}_\theta,\ZZ)$ indexed by the vertices of the quiver $Q$.  This would provide the analogue for consistent dimer models of the geometric McKay correspondence for a finite subgroup of $\SL(2,\CC)$ given by Gonzalez-Sprinberg--Verdier~\cite{GSV83}.


 \end{document}